\documentclass[a4paper, 10pt]{amsart}

\usepackage{amsmath,amssymb,amscd}

\usepackage{amsfonts}
\usepackage[english]{babel}
\usepackage[leqno]{amsmath}
\usepackage{amssymb,amsthm}
\usepackage{amscd}
\usepackage{enumerate}
\usepackage{epsfig}

\usepackage{layout}
\usepackage{fullpage}

\usepackage[usenames,dvipsnames]{color}

\usepackage[backref=page]{hyperref}

\hypersetup{
 colorlinks,
 citecolor=Green,
 linkcolor=blue,
 urlcolor=Blue}

\usepackage[matrix,arrow,tips,curve]{xy}

\input{xy}
\xyoption{all}

\newenvironment{sis}{\left\{\begin{aligned}}{\end{aligned}\right.}

\newtheorem{thm}{Theorem}[section]
\newtheorem{lemma}[thm]{Lemma}

\newtheorem{prop}[thm]{Proposition}
\newtheorem{cor}[thm]{Corollary}

\newtheorem{fact}[thm]{Fact}

\newtheorem*{thmA}{Theorem A}
\newtheorem*{thmB}{Theorem B}

\numberwithin{equation}{section}
\theoremstyle{definition}
\newtheorem{defi}[thm]{Definition}
\newtheorem{convention}[thm]{}

\theoremstyle{remark}
\newtheorem{remark}[thm]{Remark}

\newcommand{\Z}{\mathbb{Z}}
\newcommand{\Q}{\mathbb{Q}}

\newcommand{\HH}{\mathbb{H}}

\newcommand{\Pic}{\operatorname{Pic}}

\newcommand{\un}{\underline}
\newcommand{\ov}{\overline}
\newcommand{\wt}{\widetilde}
\newcommand{\wh}{\widehat}

\newcommand{\Hom}{{\rm Hom}}

\newcommand{\Spec}{\operatorname{Spec}}

\def \Im{{\rm Im}}
\def \Supp{{\rm Supp}}

\def \sign{{\rm sign}}

\def \id{{\rm id}}

\def \GL{{\rm GL}}

\def\J{\overline J}

\def \F{\mathcal F}
\def \G{\mathcal G}
\def \N {\mathcal N}
\def \X{\mathcal X}

\def\I{\mathcal I}
\def \L{\mathcal L}

\def\O{\mathcal O}
\def \A{\mathcal A}
\def \D{\mathcal D}

\def\M0{\mathcal M^0}

\def\bJ{{\mathbb J}}
\def\bJbar{\ov{{\mathbb J}}}

\def\calI{{\mathcal I}}
\def\calJ{{\mathcal J}}

\def \M{\mathcal M}

\def\m {\mathfrak{m}}

\def \Xsing{X_{\rm sing}}

\def \P{\mathcal P}

\def \Q{\mathcal Q}

\def \Pbun{\ov\P^{\rm un}}

\newcommand{\HC}{\operatorname{HC}}

\newcommand{\Hilb}{\operatorname{Hilb}}
\newcommand{\Hilbc}{\operatorname{{}^c Hilb}}
\newcommand{\Hilbl}{\operatorname{{}^l Hilb}}
\newcommand{\Hilbs}{\operatorname{{}^s Hilb}}
\newcommand{\Hilbr}{\operatorname{{}^r Hilb}}
\newcommand{\Hilbcs}{\operatorname{{}^{cs} Hilb}}
\newcommand{\Hilbcl}{\operatorname{{}^{cl} Hilb}}

\newcommand{\Flag}{\operatorname{Flag}}
\newcommand{\Sym}{\operatorname{Sym}}

\newcommand{\Open}{\left(\bJbar_X^{1-g}\times \bJbar_X^{1-g}\right)^{\natural}}

\newcommand{\Dc}{D_{\rm coh}}
\newcommand{\Dqc}{D_{\rm qcoh}}

\newcommand{\Spl}{\operatorname{Spl}}
\newcommand{\Pm}{\operatorname{Pic}^-}
\newcommand{\Pmm}{\operatorname{Pic}^{=}}
\newcommand{\Picbar}{\operatorname{\ov{\Pic}^o}}

\newcommand{\SHom}{{\mathcal{H}om}}
\newcommand{\RHom}{{\mathcal{R}\mathcal{H}om}}

\newcommand{\lotimes}{{\,\stackrel{\mathbf L}{\otimes}\,}}

\DeclareMathOperator{\codim}{{codim}}

\DeclareMathOperator{\supp}{{supp}}

\newcommand{\bbN}{{\mathbb N}}
\newcommand{\bbC}{{\mathbb C}}
\newcommand{\bbQ}{{\mathbb Q}}
\newcommand{\bbZ}{{\mathbb Z}}

\newcommand{\bbD}{{\mathbb D}}

\newcommand{\calH}{{\mathcal H}}

\newcommand{\cK}{{\mathcal K}}

\newcommand{\cO}{{\mathcal O}}

\newcommand{\bR}{{\mathbf R}}

\newcommand{\cplx}[1]{{{\mathcal #1}^{\scriptscriptstyle\bullet}}}

\title{Fourier-Mukai and autoduality for compactified Jacobians. II}

\author[Melo]{Margarida Melo}
\address{ Dipartimento di Matematica, Universit\`a Roma Tre, Largo S. Leonardo Murialdo 1, 00146 Roma (Italy); Departamento de Matem\'atica, Universidade de Coimbra, Largo D. Dinis, Apartado 3008, 3001 Coimbra (Portugal)}
\email{melo@mat.uniroma3.it, mmelo@mat.uc.pt}

\author[Rapagnetta]{Antonio Rapagnetta}
\address{Dipartimento di Matematica, Universit\`a di Roma II - Tor Vergata, via della Ricerca Scientifica 1, 00133 Roma (Italy)}
\email{rapagnet@mat.uniroma2.it}

\author[Viviani]{Filippo Viviani}
\address{ Dipartimento di Matematica, Universit\`a Roma Tre, Largo S. Leonardo Murialdo 1, 00146 Roma (Italy)}
\email{viviani@mat.uniroma3.it}

\begin{document}



\begin{abstract}

To every reduced (projective) curve $X$ with planar singularities one can associate, following E. Esteves, many fine compactified Jacobians, depending on the choice of a polarization on $X$, which are birational (possibly non-isomorphic) Calabi-Yau projective varieties with locally complete intersection singularities.
We define a Poincar\'e sheaf on the product of any two (possibly equal) fine compactified Jacobians of $X$ and show that the integral transform with kernel the Poincar\'e sheaf is an equivalence of their derived categories, hence it defines a Fourier-Mukai transform. As a corollary of this result, we prove that there is a natural equivariant open embedding of the connected component of the scheme parametrizing rank-$1$ torsion-free sheaves on $X$ into the connected component of the algebraic space parametrizing rank-$1$ torsion-free sheaves on a given fine compactified Jacobian of $X$.

The main result can be interpreted in two ways.
First of all, when the two fine compactified Jacobians are equal, the above Fourier-Mukai transform provides a natural autoequivalence of the derived category of any fine compactified Jacobian of $X$, which generalizes the classical result of S. Mukai for Jacobians of smooth curves and the more recent result of D. Arinkin for compactified Jacobians of integral curves with planar singularities. This provides further evidence for the classical limit of the geometric Langlands conjecture (as formulated by R. Donagi and T. Pantev).
Second, when the two fine compactified Jacobians are different (and indeed possibly non-isomorphic), the above Fourier-Mukai transform provides a natural equivalence of their derived categories, thus it implies that any two fine compactified Jacobians of $X$ are derived equivalent. This is in line with Kawamata's conjecture that birational Calabi-Yau (smooth) varieties should be derived equivalent and it seems to suggest an extension of this conjecture to (mildly) singular Calabi-Yau varieties.

\end{abstract}

\maketitle

\tableofcontents

\section{Introduction}

Let $C$ be a smooth irreducible projective curve over an algebraically closed field $k$ and let $J(C)$ be its Jacobian variety. Since $J(C)$ is an autodual abelian variety, i.e. it is canonically isomorphic to its
dual abelian variety, there exists a Poincar\'e line bundle $\P$ on $J(C)\times J(C)$ which is universal as a family of algebraically trivial line bundles on $J(C)$. In the breakthrough work \cite{mukai}, S. Mukai proved that the integral transform with kernel $\P$ is an auto-equivalence of the bounded derived category of  coherent sheaves on $J(C)$, or in other words it defines what is, nowadays, called a Fourier-Mukai transform.
 \footnote{More generally, for an arbitrary abelian variety $A$ with dual abelian variety $A^{\vee}$, Mukai proved that the Fourier-Mukai transform associated to the Poincar\'e line bundle on $A\times A^{\vee}$ gives an equivalence  between the bounded derived category of $A$ and that of $A^{\vee}$.}


\vspace{0,1cm}

Motivated by the classical limit of the geometric Langlands duality (see \cite{DP} and the discussion below), D. Arinkin extended in \cite{arin1} and \cite{arin2} the above Fourier-Mukai transform to the compactified Jacobians of integral projective curves with planar singularities.

\vspace{0,1cm}

The aim of this paper, which is heavily based on our previous manuscripts \cite{MRV1} and \cite{MRV2},  is to extend this autoequivalence to fine compactified Jacobians (as defined by E. Esteves in \cite{est1}) of {\em reduced projective curves with planar singularities.} The main novelty for reducible curves is that compactified Jacobians are not canonically defined but they depend on the choice of a polarization on the
curve itself. Indeed we also prove that given any two fine compactified Jacobians (which are always birational but possibly non isomorphic) of a reduced curve $X$ with planar singularities,
there is a Fourier-Mukai transform between  their derived categories, hence all fine compactified Jacobians of $X$ are derived equivalent.


\subsection{Fine compactified Jacobians of singular curves}

Before stating our main result, we need to briefly recall how Esteves's fine compactified Jacobians of \emph{reduced} curves are defined in \cite{est1}; we refer the reader to \S\ref{S:comp-Jac} for more details.
Fine compactified Jacobians of a reduced projective curve $X$ parametrize torsion-free rank-$1$ sheaves on $X$ that are semistable with respect to a general polarization on $X$. More precisely, a \emph{polarization} on $X$ is a tuple of rational numbers $\un q=\{\un q_{C_i}\}$, one for each irreducible component $C_i$ of $X$, such that $|\un q|:=\sum_i \un q_{C_i}\in \Z$.
A torsion-free rank-1  sheaf $I$ on $X$ of Euler characteristic $\chi(I):=h^0(X, I)-h^1(X, I)$ equal to $|\un q|$ is called
\emph{$\un q$-semistable} (resp. \emph{$\un q$-stable}) if
for every non-trivial subcurve $Y\subset X$, we have that
\begin{equation*}
\chi(I_Y) \geq \sum_{C_i\subseteq Y} \un q_{C_i} \: \: (\text{resp. } > ),
\end{equation*}
where $I_Y$ is the biggest torsion-free quotient of the restriction $I_{|Y}$ of $I$ to the subcurve $Y$. A polarization $\un q$ is called {\em general} if there are no strictly $\un q$-semistable sheaves, i.e. if every $\un q$-semistable sheaf is also $\un q$-stable (see Definition \ref{pola-def} for a numerical characterization of general polarizations). A  fine compactified Jacobian of $X$ is the fine moduli space $\ov{J}_X(\un q)$ of torsion-free rank-$1$ sheaves of degree $|\un q|$ on $X$ that are $\un q$-semistable (or equivalently $\un q$-stable) with respect to a general polarization $\un q$ on $X$.


If the curve $X$ has planar singularities, then we proved in \cite[Thm. A]{MRV1} that any fine compactified Jacobian $\ov{J}_X(\un q)$ of $X$ has the following remarkable properties (see Fact \ref{F:prop-Jac}):
\begin{itemize}
\item $\ov{J}_X(\un q)$ is a connected reduced scheme with locally complete intersection singularities and trivial canonical sheaf (i.e. it is a Calabi-Yau singular variety in the weak sense);
\item The smooth locus of $\ov{J}_X(\un q)$ coincides with the open subset $J_X(\un q)\subseteq \ov{J}_X(\un q)$ parametrizing line bundles; in particular $J_X(\un q)$ is dense in $\ov{J}_X(\un q)$ and $\ov{J}_X(\un q)$ is of pure dimension equal to the arithmetic genus $p_a(X)$ of $X$;
\item  $J_X(\un q)$ is the disjoint union of a number of copies of the generalized Jacobian $J(X)$ of $X$ (which is the smooth irreducible algebraic group parametrizing line bundles on $X$ of multidegree $0$)
and such a number  is independent of the chosen polarization $\un q$ and it is denoted by $c(X)$.
In particular, all the fine compactified Jacobian of $X$ have $c(X)$ irreducible components, all of  dimension equal to
$p_a(X)$, and they  are all birational among them.
\end{itemize}
Note also that the authors have found in
\cite{MRV1} examples of reducible curves (indeed even nodal curves) that admit non isomorphic (and even
non homeomorphic if $k=\bbC$) fine compactified Jacobians.


\subsection{Main results}


Let  $\J_X(\un q)$ and $\J_X(\un q')$ be two (possibly equal) fine compactified Jacobians of $X$ such that
$|\un q|=|\un q'|=0$. Starting from the universal sheaves on $X\times \J_X(\un q)$ and on $X\times \J_X(\un q')$,
it is possible to define, using the formalism of the determinant of cohomology, a (canonical) Poincar\'e line bundle $\P$ on $\ov{J}_X(\un q)\times J_X(\un q)\cup J_X(\un q')\times \J_X(\un q')$;
we refer the reader to \S\ref{S:Poinc1} for details.

Consider the inclusion $j: \ov{J}_X(\un q)\times J_X(\un q')\cup J_X(\un q)\times \J_X(\un q')\hookrightarrow \J_X(\un q)\times \J_X(\un q')$ and define $\ov\P:=j_*(\P)$.
In Theorem \ref{T:Poinc}, we prove that $\ov\P$ is a maximal Cohen-Macaulay (coherent) sheaf on $\J_X(\un q)\times \J_X(\un q')$, flat with respect to the projections over the two factors,
and whose restrictions over the fibers of each projection are again maximal Cohen-Macaulay sheaves.

The main result of this paper is the following


\begin{thmA}\label{T:MainThm}
Let $X$ be a reduced connected projective curve with planar singularities  over an algebraically closed field $k$ of characteristic either zero or bigger than the arithmetic genus $p_a(X)$ of $X$.
 Let  $\J_X(\un q)$ and $\J_X(\un q')$ be two (possibly equal) fine compactified Jacobians of
$X$ with $|\un q|=|\un q'|=1-p_a(X)$ and let $\Dqc^b(\ov{J}_X(\un q))$ and $\Dqc^b(\ov{J}_X(\un q'))$ (resp. $\Dc^b(\ov{J}_X(\un q))$ and $\Dc^b(\ov{J}_X(\un q'))$) be their bounded derived categories of quasi-coherent sheaves (resp. of coherent sheaves) .
The  integral transform with kernel  $\ov{\P}$ on $\ov{J}_X(\un q)\times \ov{J}_X(\un q')$
\begin{eqnarray*}
\Phi^{\ov{\P}}:\Dqc^b(\ov{J}_X(\un q))& \longrightarrow & \Dqc^b(\ov{J}_X(\un q'))\\
\cplx{E} &\longmapsto & \bR p_{2*}(p_1^*(\cplx{E})\lotimes \ov{\P})
\end{eqnarray*}
is an equivalence of triangulated categories (i.e. it defines a Fourier-Mukai transform) whose inverse is the integral transform $\Phi^{\ov\P^{\vee}[g]}$ with kernel
$\ov\P^{\vee}[g]:=\SHom(\ov\P, \O_{\J_X(\un q)\times \J_X(\un q')})[g]$. Moreover, $\Phi^{\ov{\P}}$ restricts to an equivalence of categories between $\Dc^b(\ov{J}_X(\un q))$ and $\Dc^b(\ov{J}_X(\un q'))$.
\end{thmA}

\noindent Some comments on the hypothesis of Theorem A are in order.

First of all, the assumption that $|\un q|=|\un q'|=1-p_a(X)$, i.e. that we are dealing with fine compactified Jacobians parametrizing sheaves of Euler characteristic $1-p_a(X)$ (or equivalently degree $0$) on $X$, guarantees that the Poincar\'e sheaf $\P$ (and hence a fortiori its extension $\ov \P$) is canonically defined, independently of the universal sheaves on $X\times \J_X(\un q)$ and on $X\times \J_X(\un q')$ which are used in its definition \eqref{E:Poin-sheaf} (recall that such universal sheaves are only well-defined up to the pull-back of a line bundle on $\J_X(\un q)$ or on $\J_X(\un q')$, respectively); see Remark \ref{R:dep-sheaf} for a discussion of this issue. However, if $|\un q|\neq 1-p_a(X)$ or
$|\un q'|\neq 1-p_a(X)$, then one can fix once and for all a Poincar\'e sheaf $\P$ (together with its extension $\ov\P$) and all our arguments go through,  giving also in this case a Fourier-Mukai transform, although not canonically defined.

Second, the assumption that either ${\rm char}(k)=0$ or ${\rm char}(k)>p_a(X)$ is needed because of the following two facts: first of all, the results of M. Haiman on the isospectral Hilbert scheme $\wt{\Hilb}^n(S)$ of a smooth surface $S$, originally proved under the assumption that  ${\rm char}(k)=0$, are known to hold also if
${\rm char}(k)>n$ (as pointed out by M. Groechenig in \cite[p. 18--19]{Gro1}); second, for any fine compactified Jacobian $\J_X(\un q)$ of $X$, the rational twisted Abel maps from
$\Hilb^{p_a(X)}$ to $\J_X(\un q)$ are, locally on the codomain, smooth and surjective (see Fact \ref{F:Abel-sur} for the precise statement).

\vspace{0,2cm}

Theorem A can be interpreted in two ways depending on whether $\ov{J}_X(\un q)=\ov{J}_X(\un q')$ or $\ov{J}_X(\un q)\neq \ov{J}_X(\un q')$.

On one hand, when applied to the case $\ov{J}_X(\un q)=\ov{J}_X(\un q')$, Theorem A provides a Fourier-Mukai autoequivalence on any fine compactified Jacobian of a reduced curve with planar singularities, thus extending the classical result of S. Mukai \cite{mukai} for Jacobians of smooth curves and the more recent result of D. Arinkin \cite[Thm. C]{arin2} for compactified Jacobians of integral curves with planar singularities. Note that in loc. cit. Arinkin states his result under the assumption that ${\rm char}(k)=0$; however it was observed by M. Groechenig in \cite[Thm. 4.8]{Gro1} that Arinkin's proof works verbatim also under the assumption that ${\rm char}(k)>2p_a(X)-1$. Our proof of Theorem A uses twisted Abel maps (see \eqref{E:Abel-map}) instead of the global Abel map used by Arinkin for integral curves; this explains why we are able to improve the hypothesis on the characteristic of the base field even for integral curves. As a consequence, it follows that all the results of \cite[Sec. 4]{Gro1} are true under the weaker assumption that ${\rm char}(k)\geq n^2(h-1)+1$.

The above Fourier-Mukai autoequivalence provides further evidence for the classical limit of the (conjectural) geometric Langlands correspondence for the general linear group $\GL_r$,
as formulated by Donagi-Pantev in \cite{DP}. More precisely, in loc. cit. the authors conjectured that there should exist a Fourier-Mukai autoequivalence, induced by a suitable Poincar\'e sheaf,
of the derived category of the moduli stack of Higgs bundles. Moreover, among other properties, such  a Fourier-Mukai autoequivalence is expected to induce an autoequivalence of the derived category
of the fibers of the Hitchin map. Since the fibers of the Hitchin map can be described in terms of compactified Jacobians of spectral curves (see \cite[Appendix]{MRV2} and the references therein for the precise
description), it is natural to expect that such a Fourier-Mukai autoequivalence  should exist on each fine compactified Jacobian of a spectral curve, which has always planar singularities.
Our Main Theorem shows that this is indeed the case for reduced spectral curves (i.e. over the so called regular locus of the Hitchin map), extending the result of D. Arinkin for integral spectral curves
(i.e. over the so called elliptic locus of the Hitchin map).

On the other hand, in the general case when $\ov{J}_X(\un q)$ is different from $\ov{J}_X(\un q')$ (and possibly non isomorphic to it, see the examples in \cite{MRV1}), Theorem A implies that $\ov{J}_X(\un q)$ and $\J_X(\un q')$
(which are birational Calabi-Yau singular projective varieties by what said above) are derived equivalent via a canonical Fourier-Mukai transform.
 This result seems to suggest an extension to singular varieties of the conjecture of Kawamata \cite{Kaw}, which predicts that birational Calabi-Yau
smooth projective varieties should be derived equivalent.

We point out that a topological counterpart of the above result is obtained by the third author, together with L. Migliorini and V. Schende, in \cite{MSV}: any two fine compactified Jacobians of $X$ (under the same assumptions on $X$) have the same perverse Leray filtration on their cohomology. This result again seems to suggest an extension to (mildly) singular varieties of the result of Batyrev \cite{Bat} which  says that birational Calabi-Yau smooth projective varieties have the same Betti numbers.

\vspace{0,2cm}

As a corollary of Theorem A, we can generalize the autoduality result of D. Arinkin \cite[Thm. B]{arin2} for compactified Jacobians of integral curves (which extends the previous result of Esteves-Kleimann \cite{EK} for integral curves with nodes and cusps). In order to state our autoduality result, we need first to introduce some notation.

For a projective $k$-scheme $Z$,  denote by $\Spl(Z)$ the (possibly non-separated) algebraic space, locally of finite type over $k$, parametrizing simples sheaves on $Z$ (see \cite[Thm. 7.4]{AK}). Denote by $\Pmm(Z)\subseteq \Spl(Z)$ the open subset parametrizing simple, torsion-free sheaves having  rank-$1$ on each irreducible component of $Z$ and by $\Pm(Z)\subseteq \Pmm(Z)$ the open
 subset parametrizing simple, Cohen-Macaulay sheaves having  rank-$1$ on each irreducible component of $Z$ (see \cite[Prop. 5.13]{AK}).
If $Z$ does not have embedded components (or, equivalently, if the structure sheaf $\O_Z$ is torsion-free) then $\Pmm(Z)$ contains the Picard group scheme $\Pic(Z)$ of $Z$ as an open subset; under this hypothesis,  we will denote by $\Picbar(Z)$ the connected component of $\Pmm(Z)$ that contains  $\O_Z\in \Pic(Z)\subseteq \Pmm(Z)$.
Clearly, $\Picbar(Z)$ contains as an open subset the connected component $\Pic^o(Z)$ of $\Pic(Z)$ that contains $\O_Z\in \Pic(Z)$.

If $X$ is a projective reduced curve with locally planar singularities, then $\Pm(X)=\Pmm(X)$  (since on a curve torsion-free sheaves are also  Cohen-Macaulay) is known to be a scheme
(which is denoted by $\bJbar_X$ in \S\ref{S:comp-Jac})  and $\Picbar(X)$ is contained in the subscheme $\bJbar_X^{1-g}\subset \bJbar_X$ parametrizing torsion-free rank-$1$ sheaves on $X$ of Euler characteristic $1-p_a(X)$ (or equivalently degree $0$), see \S\ref{S:comp-Jac}.
Note that every fine compactified Jacobian $\J_X(\un q)$ of $X$ such that $|\un q|=1-p_a(X)$ is an open and proper subscheme of $\bJbar_X^{1-p_a(X)}$ (see \S\ref{S:comp-Jac}) and that the Poincar\'e sheaf considered above is actually a restriction of a Cohen-Macaulay Poincar\'e sheaf  $\ov\P$ on $\bJbar_X^0\times \bJbar_X^0$ (see  \S\ref{S:Poinc}).


In our previous paper \cite{MRV2} we proved that there is an isomorphism of algebraic groups (see \cite[Thm. C]{MRV2})
$$\begin{aligned}
 \beta_{\un q}: J(X)=\Pic^o(X) & \longrightarrow \Pic^o(\ov{J}_X(\un q)) \\
L & \mapsto \P_L:=\P_{|\ov{J}_X(\un q) \times \{L\}}.
\end{aligned}
$$

In this paper, we prove the following theorem that can be seen as a natural generalization of the above autoduality result.

\begin{thmB}
Let $X$ be a reduced connected projective curve with planar singularities over an algebraically closed field $k$  of characteristic either zero or bigger than the arithmetic genus $p_a(X)$ of $X$.
 Let  $\J_X(\un q)$ be a fine compactified Jacobian of $X$ with $|\un q|=1-p_a(X)$.
Then the morphism
\begin{equation}\label{E:rho}
\begin{aligned}
 \rho_{\un q}: \Picbar(X) & \longrightarrow \Picbar(\ov{J}_X(\un q)) \\
I & \mapsto \ov\P_I:=\ov\P_{|\ov{J}_X(\un q) \times \{I\}}
\end{aligned}
\end{equation}
is an open embedding which is equivariant with respect to the isomorphism of algebraic groups $ \beta_{\un q}: J(X)=\Pic^o(X) \stackrel{\cong}{\longrightarrow} \Pic^o(\ov{J}_X(\un q))$, where
$\Pic^o(X)$ (resp. $\Pic^o(\J_X(\un q))$) acts on $\Picbar(X)$ (resp. on $\Picbar(\J_X(\un q))$) by tensor product. Moreover:
\begin{enumerate}[(i)]
\item \label{E:Bi} The image of $\rho_{\un q}$ is contained in $\Pm(\J_X(\un q))\cap \Picbar(\J_X(\un q))$.
\item \label{E:Bii} The morphism $\rho_{\un q}$ induces a morphism of algebraic groups
$$ \rho_{\un q}: \Picbar(X)\cap \Pic(X)  \longrightarrow \Picbar(\ov{J}_X(\un q))\cap \Pic(\J_X(\un q)).$$
\item \label{E:Biv} If  every singular point of $X$ that lies on at least two different irreducible components of $X$ is a separating node
(e.g. if $X$ is an irreducible curve or a nodal curve of compact type) then $\rho_{\un q}$ is an isomorphism between integral projective varieties.
\end{enumerate}
\end{thmB}

Part \eqref{E:Biv} of the above Theorem B is a slight generalization of the result for irreducible curves proved by D. Arinkin in \cite[Thm. B]{arin2}.
It would be interesting to know if $\rho_{\un q}$ is an isomorphism for any reduced curve $X$ with locally planar singularities.


\subsection{Sketch of the proof of Theorem A}

Let us now give a brief outline of the proof of Theorem A.

Using the well-known description of the kernel of a composition of two integral transforms,  Theorem A is equivalent to the following equality in
$\Dc^b(\J_{X}(\un q)\times \J_{X}(\un q))$:
\begin{equation}\label{Eq1}
\Psi[g]:=Rp_{13*}\left(p_{12}^*((\ov \P)^{\vee})\lotimes p_{23}^*(\ov\P)\right)[g] \cong \O_{\Delta},
\end{equation}
where $p_{ij}$ denotes the projection of $\J_{X}(\un q)\times \J_X(\un q')\times \J_{X}(\un q)$ onto the $i$-th and $j$-th factors and
$\O_{\Delta}$ is the structure sheaf of the diagonal $\Delta \subset \J_{X}(\un q)\times \J_{X}(\un q)$.

In order to prove \eqref{Eq1}, the key idea, which we learned from D. Arinkin in \cite{arin1} and \cite{arin2}, is to prove a similar formula for the effective semiuniversal deformation \footnote{In loc. cit., D. Arinkin considers the stack of all (integral) curves with planar singularities. Here (and in our previous related papers \cite{MRV1} and \cite{MRV2}), we need to work with the semiuniversal deformation space of $X$ in order to be able to define universal fine compactified Jacobians with respect to any general polarization on the central fiber, see \S \ref{S:defo}.}.
family $\pi:\X\to \Spec R_X$ of the curve $X$.
The fine compactified Jacobians $\J_X(\un q)$ and $\J_X(\un q')$ deform over $\Spec R_X$ to the universal fine compactified Jacobians $\J_{\X}(\un q)$ and $\J_{\X}(\un q')$, respectively;
see  \S\ref{S:defo}.
 Moreover, the Poincar\'e sheaf  $\ov \P$ on $\ov{J}_X(\un q)\times \J_X(\un q')$ deforms to a universal Poincar\'e sheaf $\ov{\P}^{\rm un}$ on the fiber product
 $\ov{J}_{\X}(\un q)\times_{\Spec R_X}  \J_{\X}(\un q')$.
Equation \eqref{Eq1} will follow, by restricting to the central fiber, from the following universal version of it (which we prove in Theorem \ref{T:psi-univ}):
\begin{equation}\label{Eq2}
\Psi^{\rm un}[g]:=Rp_{13*}\left(p_{12}^*((\Pbun)^{\vee})\lotimes p_{23}^*(\Pbun)\right)[g]\cong \O_{\Delta^{\rm un}}  \in \Dc^b(\J_{\X}(\un q)\times_{\Spec R_X} \J_{\X}(\un q)),
\end{equation}
where $p_{ij}$ denotes the projection of $\J_{\X}(\un q)\times_{\Spec R_X}\J_{\X}(\un q')\times_{\Spec R_X} \J_{\X}(\un q)$ onto the  $i$-th and $j$-th factor and
$\O_{\Delta^{\rm un}}$ is the structure sheaf of the universal diagonal $\Delta^{\rm un}\subset \J_{\X}(\un q)\times_{\Spec R_X}\J_{\X}(\un q)$.

A key intermediate step in proving \eqref{Eq2} consists in showing that
\begin{equation*}
\Psi^{\rm un}[g] \text{ is a Cohen-Macaulay sheaf such that } \supp \Psi^{\rm un}[g] =\Delta^{\rm un}. \tag{*}
\end{equation*}

The two main ingredients in proving (*) are the equigeneric stratification of $\Spec R_X$ (see Fact \ref{F:Diaz-Har}) and a lower bound for the codimension of the support of the restriction of
$\Psi^{\rm un}[g]$ on the fibers of $\J_{\X}(\un q)\times_{\Spec R_X}\J_{\X}(\un q)\to \Spec R_X$ (see Proposition \ref{P:codim-psi}).

\subsection{Outline of the paper}

The paper is organized as follows.

In Subsection \ref{S:comp-Jac} we collect several facts on fine compactified Jacobians of reduced curves, with special emphasis on the case of curves with planar singularities.
In Subsection \ref{S:defo} we recall some facts on deformation theory that will be crucial in the proof of Theorem A: the  equigeneric stratification of the semiuniversal deformation space of a curve with planar singularities (see Fact \ref{F:Diaz-Har}) and the  universal fine compactified Jacobians (see Fact \ref{F:Punivfine}).



Section \ref{S:Hilb} is devoted to Hilbert schemes of points on smooth surfaces and on curves with planar singularities. More precisely, in Subsection \ref{S:Hilb-sur} we recall some classical facts on the Hilbert scheme of points on a smooth surface and on the Hilbert-Chow morphism together with the recent results of M. Haiman on the isospectral Hilbert scheme. In Subsection \ref{S:Hilb-cur}, we recall some facts on the Hilbert scheme of a curve $X$ with planar singularities and ton the  local Abel map from the Hilbert scheme of $X$ to any fine compactified Jacobian of $X$.

In Section \ref{S:Poinc}, we define the Poincar\'e sheaf $\ov\P$ and we prove that it is a maximal Cohen-Macaulay sheaf, flat over each factor (see Theorem \ref{T:Poinc}).
The proof of Theorem \ref{T:Poinc} is based on the work of Arinkin \cite{arin2}, which uses in a crucial way the properties of  M. Haiman's isospectral Hilbert scheme of a surface.

In Section \ref{S:Poinc3}, we establish several properties of the Poincar\'e sheaf $\ov\P$ while Section \ref{S:proof} contains the proofs of Theorem A and of Theorem B.


\vspace{0,2cm}

The following notations will be used throughout the paper.

\subsection*{Notations}
\begin{convention}
 	$k$ will denote an algebraically closed field (of arbitrary characteristic), unless otherwise stated. All \textbf{schemes} are $k$-schemes, and all morphisms are implicitly assumed to respect
	the $k$-structure.
\end{convention}

\begin{convention}\label{N:curves}
	A \textbf{curve}  is a \emph{reduced} projective scheme over $k$ of pure dimension $1$. 

Given a curve $X$, we denote by $X_{\rm sm}$ the smooth locus of $X$, by $X_{\rm sing}$ its singular locus and by $\nu:X^{\nu}\to X$ the normalization morphism.
 We denote by $\gamma_X$, or simply by $\gamma$ where there is no danger of confusion, the number of irreducible components of $X$.

We denote by $p_a(X)$  the \emph{arithmetic genus} of $X$, i.e.  $p_a(X):=1-\chi(\O_X)=1-h^0(X,\O_X)+h^1(X, \O_X) $.
We denote by $g^{\nu}(X)$ the \emph{geometric genus} of $X$, i.e. the sum of the genera of the connected components of the normalization $X^{\nu}$, and by $p_a^{\nu}(X)$ the arithmetic genus of the
normalization $X^{\nu}$ of $X$. Note that $p_a^{\nu}(X)=g^{\nu}(X)+1-\gamma_X$.

\end{convention}

\begin{convention}
	A \textbf{subcurve} $Z$ of a curve $X$ is a closed $k$-scheme $Z \subseteq X$ that is reduced  and of pure dimension $1$.  We say that a subcurve $Z\subseteq X$ is non trivial if
	$Z\neq \emptyset, X$.
	
	Given two subcurves $Z$ and $W$ of $X$ without common irreducible components, we denote by $Z\cap W$ the $0$-dimensional subscheme of $X$ that is obtained as the
	scheme-theoretic intersection of $Z$ and $W$ and we denote by $|Z\cap W|$ its length.
	
	Given a subcurve $Z\subseteq X$, we denote by $Z^c:=\ov{X\setminus Z}$ the \textbf{complementary subcurve} of $Z$ and we set $\delta_Z=\delta_{Z^c}:=|Z\cap Z^c|$.
 \end{convention}

\begin{convention}
A curve $X$ is called \textbf{Gorenstein} if its dualizing sheaf $\omega_X$ is a line bundle.
\end{convention}

\begin{convention}
A curve $X$ has \textbf{locally complete intersection (l.c.i.) singularities at $p\in X$} if the completion $\wh{\O}_{X,p}$ of the local ring of $X$ at $p$ can be written as
$$\wh{\O}_{X,p}=k[[x_1,\ldots,x_r]]/(f_1,\ldots,f_{r-1}),$$
for some $r\geq 2$ and some $f_i\in k[[x_1,\ldots,x_r]]$. A curve $X$ has locally complete intersection (l.c.i.)
singularities if $X$ is l.c.i. at every $p\in X$.
Clearly, a curve with l.c.i. singularities is Gorenstein.
\end{convention}

\begin{convention}
A curve $X$ has \textbf{planar singularities at $p\in X$} if  the completion
$\wh{\O}_{X,p}$ of the local ring of $X$ at $p$ has embedded dimension two, or equivalently if it can be written
as
$$\wh{\O}_{X,p}=k[[x,y]]/(f),$$
for a reduced series $f=f(x,y)\in k[[x,y]]$.
A curve $X$ has planar singularities if $X$ has planar singularities at every $p\in X$.
Clearly, a curve with planar singularities has l.c.i. singularities, hence it is Gorenstein.
\end{convention}

\begin{convention}\label{N:Jac-gen}
Given a curve $X$, the \textbf{generalized Jacobian} of $X$, denoted by $J(X)$ or by $\Pic^{\un 0}(X)$,
is the algebraic group whose $k$-valued points are the group of line bundles on $X$ of multidegree $\un 0$ (i.e. having
degree $0$ on each irreducible component of $X$) together with the multiplication given by the tensor product.
The generalized Jacobian of $X$ is a connected commutative smooth algebraic group of dimension equal to
$h^1(X,\O_X)$ and it coincides with the connected component of the Picard scheme $\Pic(X)$ of $X$ containing the identity.
\end{convention}

\begin{convention}
Given a scheme $Y$, we will denote by $D(Y)$ the \textbf{derived category} of complexes of $\mathcal{O}_Y$-modules with quasi-coherent cohomology sheaves and by $D^b(Y)\subset D(Y)$ the \textbf{bounded derived category} consisting of complexes with only finitely many non-zero cohomology sheaves. We will denote by $\Dc(Y)\subset D(Y)$ (resp. $\Dc^b(Y)\subset D^b(Y)$) the full category consisting of complexes with coherent cohomology and by $\Dqc(Y)\subset D(Y)$ (resp. $\Dqc^b(Y)\subset D^b(Y)$) the full category consisting of complexes with quasi-coherent cohomology.
\end{convention}

\begin{convention}
Given a scheme $Y$ and a closed point $y\in Y$, we will denote by ${\bf k}(y)$ the \textbf{skyscraper sheaf} supported at $y$.
\end{convention}

\section{Fine compactified Jacobians and their universal deformations}

The aim of this section is to summarize some properties of fine (universal) compactified Jacobians of connected reduced curves with planar singularities
which were proved in \cite{MRV1} and \cite{MRV2}. Throughout this section, we fix a connected reduced curve $X$ over an algebraically closed field $k$.

\subsection{Fine compactified Jacobians} \label{S:comp-Jac}



The aim of this subsection is to review the definition and the main properties of fine compactified Jacobians of reduced curves with
planar singularities, referring to \cite[\S 2]{MRV1} for complete proofs.

Fine compactified Jacobians of a curve $X$ will parametrize sheaves on $X$ of a certain type, that we are now going to define.

\begin{defi}\label{D:sheaves}
A coherent sheaf $I$ on a connected reduced curve $X$ is said to be:
\begin{enumerate}[(i)]
\item \emph{rank-$1$} if $I$ has generic rank $1$ at every irreducible component of $X$;
\item \emph{torsion-free} if $\Supp(I)=X$ and every non-zero subsheaf $J\subseteq I$ is such that $\dim \Supp(J)=1$;
\item \emph{simple} if ${\rm End}_k(I)=k$.
\end{enumerate}
\end{defi}
\noindent Note that any line bundle on $X$ is a simple rank-$1$ torsion-free sheaf.

\vspace{0,1cm}

If the curve $X$ is Gorenstein, then rank-$1$ torsion-free sheaves on $X$ correspond to linear equivalence classes of generalized divisors in the sense of Hartshorne (see \cite[Prop. 2.8]{Har3}).
This allows us to describe these sheaves in terms of (usual) effective divisors as follows.

\begin{lemma}\label{L:sheaves}
Let $X$ be a (reduced) Gorenstein curve and let $I$ be a rank-$1$ torsion-free sheaf on $X$. Then there exist two disjoint effective divisors $E_1$ and $E_2$ on $X$, with $E_2$ being a Cartier divisor supported on the smooth locus of $X$, such that
$$I=I_{E_1}\otimes I_{E_2}^{-1},
$$
where $I_{E_i}$ denotes the ideal sheaf of $E_i$ (for $i=1,2$).
\end{lemma}
\begin{proof}
It follows from \cite[Prop. 2.11]{Har3} that we can write $I=I_{E_1}\otimes I_{E_2}^{-1}$ for two effective divisors $E_1$ and $E_2$ on $X$ such that $E_2$ is Cartier and linearly equivalent to a
an arbitrary high power of a fixed ample line bundle. Thus, up to by replacing $E_2$ with a divisor linear equivalent to it, we can assume that the support of $E_2$ is disjoint from the singular locus of
$X$ and from the support of $E_1$, q.e.d.
\end{proof}

Rank-$1$ torsion-free simple sheaves on $X$ can be parametrized by a scheme. More precisely,  there exists $k$-scheme $\bJbar_X$, locally of finite type and universally closed over $k$, which represents the
 Zariski (or, equivalently, \'etale or fppf) sheafification of the functor
 \begin{equation*}\label{E:func-Jbar}
\bJbar_X^* : \{{\rm Schemes}/k\}  \to \{{\rm Sets}\}
\end{equation*}
which associates to a $k$-scheme $T$ the set of isomorphism classes of $T$-flat, coherent sheaves on $X\times _k T$
whose fibers over $T$ are simple rank-$1$ torsion-free sheaves.
The fact that $\bJbar_X$ represents the  Zariski  sheafification of the functor $\bJbar_X^* $ amounts to the existence of a coherent sheaf $\I$ on $X\times \bJbar_X$, flat over $\bJbar_X$,  such that for every $\F\in \bJbar_X^*(T)$ there exists a unique map $\alpha_{\F}:T\to \bJbar_X$ with the property that
$\F=(\id_X\times \alpha_{\F})^*(\I)\otimes \pi_2^*(N)$ for some $N\in \Pic(T)$, where $\pi_2:X\times T\to T$ is the projection onto the second factor.
The sheaf $\I$ is uniquely determined up  to tensor product with the pullback of an invertible sheaf on $\bJbar_X$ and it is called a \emph{universal sheaf}.
Moreover, there exists a $k$-smooth open subset $\Pic(X)=\bJ_X\subseteq \bJbar_X$, whose $k$-points parametrizes line bundles on $X$.
The restriction of a universal sheaf $\I$ to  $X\times \bJ_X$ is a line bundle that enjoys a similar universal property  with respect to  families of line bundles on $X$. 
A proof of the above results can be found in \cite[Fact 2.2]{MRV1}, where they are deduced from results of Murre-Oort, Altmann-Kleiman \cite{AK}, \cite{AK2} and Esteves \cite{est1}.

Since the Euler characteristic $\chi(I):=h^0(X,I)-h^1(X,I)$ of a sheaf $I$ on $X$ is constant under deformations, we get a decomposition
\begin{equation}\label{E:dec}
\begin{sis}
& \bJbar_X=\coprod_{\chi \in \Z} \bJbar_X^{\chi},\\
& \bJ_X=\coprod_{\chi\in \Z} \bJ_X^{\chi}=\coprod_{\chi\in \Z} \Pic^{\chi+p_a(X)-1}(X),\\
\end{sis}
\end{equation}
where $\bJbar_X^{\chi}$ (resp. $\bJ_X^{\chi}$) denotes the open and closed subscheme of $\bJbar_X$ (resp. $\bJ_X$) parametrizing simple rank-1 torsion-free sheaves $I$ (resp. line bundles $L$) such that $\chi(I)=\chi$ (resp. $\chi(L)=\chi$ or equivalently degree $\deg(L)=\chi+p_a(X)-1$). We will sometimes refer to the degree of a rank-$1$ torsion-free sheaf $I$, which is defined by $\deg I:=\chi(I)+p_a(X)-1$. 

If $X$ has planar singularities, then $\bJbar_X$ has the following properties.

\begin{fact}\label{F:propJbig}
Let $X$ be a connected reduced curve with planar singularities. Then
\begin{enumerate}[(i)]
\item \label{F:propJbig1} $\bJbar_X$ is a reduced scheme with locally complete intersection singularities.
\item \label{F:propJbig2} $\bJ_X$ is the smooth locus of $\bJbar_X$. In particular, $\bJ_X$ is dense in $\bJbar_X$.
\end{enumerate}
\end{fact}
\begin{proof}
See \cite[Thm. 2.3]{MRV1}.
\end{proof}

For any integer $\chi\in \bbZ$, the scheme $\bJbar_X^{\chi}$ is not of finite type nor separated over $k$  if $X$ is not irreducible. However, it can be covered by
open subsets that are proper (and even projective) over $k$: the fine compactified Jacobians of $X$. The fine compactified Jacobians  depend on the choice of a general polarization, whose definition is as
follows  (using the notations of \cite{MRV1}).

\begin{defi}\label{pola-def}
Let $X$ be a connected reduced curve. 
\begin{enumerate}
\item A \emph{polarization} on a connected curve $X$ is a tuple of rational numbers $\un q=\{\un q_{C_i}\}$, one for each irreducible component $C_i$ of $X$, such that $|\un q|:=\sum_i \un q_{C_i}\in \Z$.
We call $|\un q|$ the total degree of $\un q$. Given any subcurve $Y \subseteq X$, we set $\un{q}_Y:=\sum_j \un{q}_{C_j}$ where the sum runs
over all the irreducible components $C_j$ of $Y$.
\item A polarization $\un q$ is called \emph{integral} at a subcurve $Y\subseteq X$ if
$\un q_Z \in \Z$ for any connected component $Z$ of $Y$ and of $Y^c$. A polarization is called
\emph{general} if it is not integral at any non-trivial subcurve $Y\subset X$.
\end{enumerate}
\end{defi}





The choice of a polarization on $X$ allows to define the concepts of stability and semistability. 

\begin{defi}\label{sheaf-ss-qs}
Let $\un q$ be a polarization on $X$ and let $I$ be a torsion-free rank-1  sheaf on $X$ of Euler characteristic $\chi(I)=|\un q|$ (not necessarily simple).
We say that $I$ is \emph{(semi)stable} with respect to $\un q$, or simply $\un q$-(semi)stable, if
for every non-trivial subcurve $Y\subset X$, we have that
\begin{equation}\label{multdeg-sh1}
\chi(I_Y)\geq \un q_Y \: \: (\text{resp. } >)
\end{equation}
where $I_Y$ is the quotient of the restriction $I_{|Y}$ modulo its biggest zero-dimensional subsheaf (or, in other words, $I_Y$ is the biggest 
 torsion-free quotient of $I_{|Y}$).
\end{defi}


Given a polarization $\un q$ on $X$, we denote by $\ov{J}^{ss}_X(\un q)$ (resp. $\ov{J}^{s}_X(\un q)$) the subscheme of
$\bJbar_X$ parametrizing simple rank-1 torsion-free sheaves $I$ on $X$ which are $\un q$-semistable (resp. $\un q$-stable).
By \cite[Prop. 34]{est1}, the inclusions
$$\ov{J}^{s}_X(\un q)\subseteq \ov{J}^{ss}_X(\un q)\subset \bJbar_X$$
are open.

\begin{fact}[Esteves]\label{F:Este-Jac}
Let $X$ be a connected reduced curve.
\noindent
\begin{enumerate}[(i)]
\item \label{F:Este-Jac2} If $\un q$ is general then $\ov{J}^{ss}_X(\un q)=\ov{J}^{s}_X(\un q)$ is a projective scheme over $k$ (not necessarily reduced).
\item \label{F:Este-Jac3} $\displaystyle \bJbar_X=\bigcup_{{\un q} \text{ general}} \ov{J}_X^s(\un q).$
\end{enumerate}
\end{fact}
\begin{proof}
See \cite[Fact 2.19]{MRV1} and the references therein. 
\end{proof}

If $\un q$ is general, we set $\ov{J}_X(\un q):=\ov{J}^{ss}_X(\un q)=\ov{J}^{s}_X(\un q)$ and we call it the
\emph{fine compactified Jacobian} with respect to the polarization $\un q$. We denote by $J_X(\un q)$ the open subset of $\ov{J}_X(\un q)$ parametrizing
line bundles on $X$. Note that $J_X(\un q)$ is isomorphic to the disjoint union of a certain number of copies of the generalized Jacobian $J(X)=\Pic^{\un 0}(X)$ of $X$.

If $X$ has planar singularities, then any fine compactified Jacobian of $X$ enjoys the following properties.

\begin{fact}\label{F:prop-Jac}
Let $X$ be a connected reduced curve with planar singularities and $\un q$ a general polarization on $X$. Then
\begin{enumerate}[(i)]
\item \label{F:prop-Jac1} $\ov{J}_X(\un q)$ is a connected reduced scheme with locally complete intersection singularities and trivial dualizing sheaf. 
\item \label{F:prop-Jac2} The smooth locus of $\ov{J}_X(\un q)$ coincides with the open subset $J_X(\un q)\subseteq \ov{J}_X(\un q)$ parametrizing line bundles; in particular $J_X(\un q)$ is dense in $\ov{J}_X(\un q)$ and of pure dimension equal to $p_a(X)$. 
\item \label{F:prop-Jac3} $J_X(\un q)$ is the disjoint union of a number of copies of $J(X)$ and such a number  is independent of the chosen polarization $\un q$ and it is denoted by $c(X)$.
In particular, all the fine compactified Jacobian of $X$ have $c(X)$ irreducible components, all of  dimension equal to
the arithmetic genus $p_a(X)$ of $X$, and they  are all birational among them.
\end{enumerate}
\end{fact}
\begin{proof}
See \cite[Thm. A]{MRV1}.
\end{proof}

In \cite[\S 5.1]{MRV1},  we prove a formula for the number $c(X)$ (which is called the complexity of $X$) in terms of the combinatorics of the curve $X$. The above properties depends heavily on the fact that the curve $X$ has planar singularities and indeed we expect that many of the above properties are false without this assumption (see the discussion in \cite[Rmk. 2.7]{MRV1}).





\subsection{Universal fine compactified Jacobians} \label{S:defo}


The aim of this subsection is to introduce and describe universal fine compactified Jacobians following the presentation given in \cite[\S 4, \S 5]{MRV1} and \cite[\S 3]{MRV2}.

Consider the effective semiuniversal deformation of a reduced curve $X$ (in the sense of \cite{Ser})
\begin{equation}\label{E:eff-fam}
\xymatrix{
X \ar[d]\ar@{^{(}->}[r]\ar@{}[dr]|{\square}&  \X\ar[d]^{\pi}\\
o:=[\m_X] \ar@{^{(}->}[r] &  \Spec R_X,
}
\end{equation}
where $R_X$ is a Noetherian complete local $k$-algebra with maximal ideal $\m_{X}$ and residue field $k$. Note that if $X$ has locally complete intersection singularities (e.g. if $X$ has planar singularities), 
then $\Spec R_X$ is formally smooth or, equivalently, $R_X$ is a power series ring over $k$ (see e.g. \cite[Fact 4.1]{MRV1} and the references therein). For any (schematic) point $s\in \Spec R_X$, we will denote by 
$\X_s:=\pi^{-1}(s)$  the fiber of $\pi$ over $s$ and by $\X_{\ov s}:=\X_{s}\otimes_{k(s)} \ov{k(s)}$ an associated geometric fiber.
For later use, we recall the following 

\begin{lemma}\label{L:openU}
 Let $U$ be the open subset of $\Spec R_X$ consisting of all the  points $s\in \Spec R_X$ such that
the fiber $\X_{\ov s}$ of the universal family $\pi:\X\to \Spec R_X$ is smooth or it has a unique singular point which is a
node. If $X$ has locally planar singularities, then the codimension of the complement of $U$ inside $\Spec R_X$ is at least two.
\end{lemma}
\begin{proof}
See \cite[Lemma 4.3]{MRV1}. 
\end{proof}


The scheme $\Spec R_X$ admits two stratifications into closed subsets according to either the arithmetic genus or the geometric genus of the
normalization of the geometric fibers of the family $\pi$. More precisely, using the notation introduced in \ref{N:curves}, consider the two functions
\begin{equation}
\begin{aligned}
p_a^{\nu}: \Spec R_X & \longrightarrow \bbN,\\
 s & \mapsto p_a^{\nu}(\X_{\ov s}):=p_a(\X_{\ov s}^{\nu}),\\
\end{aligned}
\hspace{2cm}
\begin{aligned}
g^{\nu}: \Spec R_X & \longrightarrow \bbN,\\
s & \mapsto g^{\nu}(\X_{\ov s})=g^{\nu}(\X_{\ov s}^{\nu}).
\end{aligned}
\end{equation}
Since  the number of connected components of $\X_{\ov s}^{\nu}$ is the number $\gamma(\X_{\ov s})$
of irreducible components of $\X_{\ov s}$, we have the relation
\begin{equation}\label{E:geo-ari}
p_a^{\nu}(\X_{\ov s})=g^{\nu}(\X_{\ov s})-\gamma(\X_{\ov s})+1\leq g^{\nu}(\X_{\ov s}).
\end{equation}
The functions $p_a^{\nu}$ and $g^{\nu}$ are lower semi-continuous (see \cite[Lemma 3.2]{MRV2}). 
Moreover, using \eqref{E:geo-ari} and the fact that the arithmetic genus $p_a$ stays constant in the family $\pi$ because of flatness, we get
that
$$
p_a(X^{\nu})=p_a^{\nu}(X)\leq p_a^{\nu}(\X_{\ov s})\leq g^{\nu}(\X_{\ov s})\leq p_a(\X_{\ov s})=p_a(X).
$$
Therefore for any $p_a(X^{\nu})\leq l\leq p_a(X)$ we have two closed subsets of $\Spec R_X$:
\begin{equation}\label{E:strata-RX}
(\Spec R_X)^{g^{\nu}\leq l}:=\{s\in \Spec R_X \: : \: g^{\nu}(\X_{\ov s})\leq l \}\subseteq (\Spec R_X)^{p_a^{\nu}\leq l}:=\{s\in \Spec R_X \: : \: p_a^{\nu}(\X_{\ov s})\leq l \}.
\end{equation}
If $X$ has planar singularities, then  the stratification (called \emph{equigeneric stratification}) by the latter closed subsets has the following remarkable properties.

\begin{fact}\label{F:Diaz-Har}
Assume that $X$ is a reduced curve with planar singularities. Then, for any $p_a(X^{\nu})\leq l\leq p_a(X)$, we have that:
\begin{enumerate}[(i)]
\item \label{F:Diaz-Harris1} The closed subset
$(\Spec R_X)^{p_a^{\nu}\leq l}\subset \Spec R_X$ has codimension  at least  $p_a(X)-l$. Hence, the same is true for the closed subset $(\Spec R_X)^{g^{\nu}\leq l}\subseteq \Spec R_X$. 
\item \label{F:Diaz-Harris2} Each generic point $s$ of $(\Spec R_X)^{p_a^{\nu}\leq l}$ is such that $\X_{\ov s}$ is a nodal curve.
\end{enumerate}
\end{fact}
\begin{proof}
See \cite[Thm. 3.3]{MRV2}. 
\end{proof}

The schemes $\bJ_X\subseteq \bJbar_X$ of \S\ref{S:comp-Jac} can be deformed over $\Spec R_X$. More precisely, there a scheme $\bJbar_{\X}$ endowed with a morphism
$u:\bJbar_{\X}\to \Spec R_X$, which is locally of finite type and universally closed, which represents
the Zariski (or, equivalently, \'etale of fppf) sheafification of the functor
$$\bJbar_{\X}^*:\{\Spec R_X-\text{schemes} \} \longrightarrow \{ \text{Sets} \}$$
which sends a scheme $T\to \Spec R_X$ to the set of isomorphism classes of $T$-flat, coherent sheaves on
$\X_T:=T\times_{\Spec R_X} \X$ whose fibers over $T$ are simple rank-1  torsion-free sheaves. 
The fact that $\bJbar_{\X}$ represents the  Zariski  sheafification of the functor $\bJbar_{\X}^* $ amounts to the existence of a coherent sheaf $\wh{\I}$ on $\X\times_{\Spec R_X} \bJbar_{\X}$, flat over
$\bJbar_{\X}$,  such that for every $\F\in \bJbar_{\X}^*(T)$ there exists a unique $\Spec R_X$-map $\alpha_{\F}:T\to \bJbar_{\X}$ with the property that $\F=(\id_{\X}\times \alpha_{\F})^*(\wh{\I})\otimes \pi_2^*(N)$ for some $N\in \Pic(T)$, where $\pi_2:\X\times_{\Spec R_X} T\to T$ is the projection onto the second factor.
The sheaf $\wh{\I}$ is uniquely determined up  to tensor product with the pullback of an invertible sheaf on $\bJbar_{\X}$ and it is called a \emph{universal sheaf} on $\bJbar_{\X}$.
Moreover, there exists an open subscheme $\bJ_{\X}\subseteq \bJbar_{\X}$, smooth over $\Spec R_X$,  
that parametrizes families of line bundles on the family $\pi:\X\to \Spec R_X$. 
Furthermore, the geometric fiber of $\bJbar_{\X}$ (resp. of $\bJ_{\X}$) over $s\in \Spec R_X$ is isomorphic to $\bJbar_{\X_{\ov s}}$ (resp. $\bJ_{\X_{\ov s}}$) and the pull-back of $\wh{\I}$ to 
$\X_{\ov s}\times \bJbar_{\X_{\ov s}}$ is a universal sheaf for $\bJbar_{\X_{\ov s}}$. In particular, the fiber of $\bJbar_{\X}$ (resp. of $\bJ_{\X}$) over the closed point $o\in \Spec R_X$ is isomorphic to $\bJbar_X$ (resp. $\bJ_X$) and the restriction  of $\wh{\I}$ to $X\times \bJbar_X$ is equal to a universal sheaf as in \S\ref{S:comp-Jac}.
A proof of the above results can be found in \cite[Fact 4.1]{MRV1}, where they are deduced from results of Altmann-Kleiman \cite{AK}, \cite{AK2} and Esteves \cite{est1}.

\vspace{0.1cm}

We now introduce universal fine compactified Jacobians, which are certain open subsets of $\bJbar_{\X}$ that are projective over $\Spec R_X$ and whose central fiber
is a fine compactified Jacobian of $X$.
The universal fine compactified Jacobian will depend on  a general polarization $\un q$ on $X$ as in Definition \ref{pola-def}. Indeed, the polarization $\un q$ induces
a polarization on each geometric fiber of the effective semiuniversal deformation family
$\pi:\X\to \Spec R_X$, in the following way. For any (schematic) point $s\in \Spec R_X$, denote by $\psi_s:\X_{\ov s} \to \X_s$ the natural base change map.
There is a natural map
\begin{equation}\label{E:map-subcurve}
\begin{aligned}
\Sigma_s:\{\text{Subcurves of } \X_{\ov s}\} & \longrightarrow \{\text{Subcurves of } X\}\\
\X_{\ov s}\supseteq Z & \mapsto \ov{\psi_s(Z)}\cap X\subseteq X,
\end{aligned}
\end{equation}
where $\ov{\psi_s(Z)}$ is the Zariski closure inside $\X$ of the subcurve $\psi_s(Z)\subseteq \X_s$ and the intersection  $\ov{\psi_s(Z)}\cap X$ is endowed with the reduced scheme structure 
(see \cite[\S 5]{MRV1} for more details).
Using the above map $\Sigma_s$,  we can define a polarization $\un q^s$ on $\X_{\ov s}$ starting from a polarization $\un q$ on $X$ by the rule:
\begin{equation}\label{E:pola-Xs}
\un q^s_Z:=\un q_{\Sigma_s(Z)} \: \: \text{ for any subcurve } Z\subseteq \X_{\ov s}. 
\end{equation}
It turns out that if $\un q$ is a general polarization on $X$, then $\un q^s$ is a general polarization on $\X_{\ov s}$ for any point $s\in \Spec R_X$, see \cite[Lemma-Definition 5.3]{MRV1}.

Given a general polarization $\un q$ on $X$, it is proved in \cite[Thm. 5.4]{MRV1} that  there exists an open subscheme $\J_{\X}(\un q)\subseteq \bJbar_{\X}$, called 
 the \emph{universal fine compactified Jacobian} of $X$ with respect to the polarization $\un q$, 
which is projective over $\Spec R_X$ and such that the geometric fiber of  $u:\J_{\X}(\un q)\to \Spec R_X$ over a point
$s\in \Spec R_X$ is isomorphic to $\J_{\X_{\ov s}}(\un q^s)$. In particular, the fiber of $\J_{\X}(\un q)\to \Spec R_X$
over the closed point $o\in \Spec R_X$ is isomorphic to $\J_X(\un q)$. We denote by $J_{\X}(\un q)$ the open subset of $\J_{\X}(\un q)$ parametrizing line bundles,
i.e. $J_{\X}(\un q)=\J_{\X}(\un q)\cap \bJ_{\X}\subseteq \bJbar_{\X}$.




If the curve $X$ has planar singularities, then the universal fine compactified Jacobians of $X$ have several nice properties that we collect in the following statement.

\begin{fact}\label{F:Punivfine}
Assume that $X$ is a reduced and connected curve with planar singularities and let $\un q$ be a general polarization on $X$.
Then we have:
\begin{enumerate}[(i)]
\item \label{F:Punivfine1} The scheme $\J_{\X}(\un q)$ is smooth and irreducible.
\item \label{F:Punivfine2} The surjective map $u:\J_{\X}(\un q)\to \Spec R_X$ is projective and flat of relative dimension $p_a(X)$.
\item \label{F:Punivfine3} The smooth locus of $u$ is $J_{\X}(\un q)$.
\end{enumerate}
\end{fact}
\begin{proof}
See \cite[Thm. 5.5]{MRV1}.
\end{proof}

\section{Punctual Hilbert schemes} \label{S:Hilb}

The aim of this section is to recall some properties of the punctual Hilbert schemes (i.e. Hilbert schemes of points) on curves with planar singularities and on smooth surfaces, which will be
needed in Section \ref{S:Poinc}.

For any  projective scheme $Z$ and $n\geq 1$, let $\Hilb_Z^n$ be the (punctual) Hilbert scheme parametrizing $0$-dimensional subschemes $D\subset Z$ of length $n$, i.e. such that $k[D]:=\Gamma(D,\O_D)$ is a $k$-algebra of dimension $n$.
The Hilbert scheme $\Hilb_Z^n$ is endowed with a universal divisor $\D$ giving rise to the following diagram:
\begin{equation}\label{E:univ-div}
\xymatrix{
& \D \ar@{^{(}->}[r]\ar[dl]_{h}\ar[dr]^{f} & \Hilb_Z^n\times Z \\
\Hilb_Z^n & & Z
}
\end{equation}
where the morphism $h$ is finite and flat of degree $n$. The sheaf $\A:=h_*\O_{\D}$ is a coherent sheaf of algebras on $\Hilb^n_Z$ which is locally free of rank $n$. The fiber of $\A$ over
$D\in \Hilb_Z^n$ is canonically isomorphic to the $k$-algebra $k[D]$ of regular functions on $D$.
We refer the reader to \cite[Chap. 5, 6]{FGA} for a detailed account of the theory of Hilbert schemes.

The punctual Hilbert scheme $\Hilb_Z^n$ contains a remarkable open subset $\Hilbc_Z^n\subseteq \Hilb_Z^n$, called the \emph{curvilinear Hilbert scheme} of $Z$, consisting of all the $0$-dimensional subschemes $D\in \Hilb_Z^n$ such that $Z$ can be embedded into a smooth curve, or equivalently such that
$$k[D]\cong \prod_{i}\frac{k[x]}{(x^{n_i})}.$$

In what follows, we will be concerned with the punctual Hilbert schemes of curves and surfaces.
Observe that if a curve $X$ is contained in a smooth surface $S$, then $X$ has planar singularities.
Indeed, the converse is also true due to the following result of Altman-Kleiman \cite{AK0}.

\begin{fact}[Altman-Kleiman]\label{F:cur-sur}
If $X$ is a connected projective reduced curve with planar singularities, then there exists a smooth projective integral surface $S$ such that $X\subset S$.
\end{fact}

If $X\subset S$ is as above, then we get a closed embedding
$$\Hilb_X^n\subseteq \Hilb_S^n.$$
In the next two subsections, we will review some of the properties of $\Hilb_S^n$ and of $\Hilb_X^n$, that we will need later on.

\subsection{Punctual Hilbert schemes of surfaces}\label{S:Hilb-sur}

Throughout this subsection, we fix a projective smooth integral surface $S$ over an algebraically closed field $k$.

The following properties of $\Hilb^n_S$ are due to J. Fogarty (see \cite[Thm. 7.2.3]{FGA} for a modern proof).

\begin{fact}[Fogarty] \label{F:Hilb-sur}
For any $n\in \bbN$ and any projective smooth connected surface $S$, the punctual Hilbert scheme $\Hilb^n_S$ is smooth
and irreducible of dimension $2n$.
\end{fact}

Let $\Sym^n(S)$ be the $n$-th symmetric product of $S$, i.e. $\Sym^n(S)=S^n/\Sigma_n$ where $\Sigma_n$ is the symmetric group on $n$ letters acting on the $n$-th product $S^n$ by permuting the factors.  The $n$-th symmetric product $\Sym^n(S)$ parametrizes $0$-cycles $\displaystyle \zeta=\sum_{p\in \supp \zeta} \zeta_p \cdot p$ on $S$ of length $n$. There is a surjective morphism, called the \emph{Hilbert-Chow morphism} (see \cite[Sec. 7.1]{FGA}), defined by
\begin{equation}\label{E:Hilb-Chow}
\begin{aligned}
\HC: \Hilb_S^n & \longrightarrow \Sym^n(S), \\
D & \mapsto \sum_{p\in S} l(\cO_{D,p})\cdot  p,
\end{aligned}
\end{equation}
where $l(\cO_{D,p})$ is the length of the Artinian ring $\cO_{D,p}$.

The fiber of the Hilbert-Chow morphism over a divisor $\sum_i n_i p_i\in \Sym^n S$  is isomorphic to
(see \cite[p. 820]{Iar0})
\begin{equation}\label{E:fiberHC}
\HC^{-1}\left(\sum_i n_i p_i\right)\cong \prod_i \Hilb^{n_i}(\wh{\O}_{S,p_i}),
\end{equation}
where, for any $m\geq 1$ and any $p\in S$, $\Hilb^{m}(\wh{\O}_{S,p}):=\Hilb^{m}(k[[x,y]])$ is the \emph{local Hilbert scheme} parametrizing ideals $I\subset k[[x,y]]$ of colength $m$, i.e. such that $\displaystyle \frac{k[[x,y]]}{I}$ is a $k$-algebra of dimension $m$. Denote by $\Hilbc^m(k[[x,y]])\subseteq \Hilb^{m}(k[[x,y]])$ the open subset
(called the \emph{curvilinear local Hilbert scheme}) parametrizing ideals $I\subset  k[[x,y]]$ such that
$\displaystyle \frac{k[[x,y]]}{I}\cong \frac{k[z]}{(z^m)}$.

The following result was proved  by J. Brian\c con \cite{Bri} (see also \cite{Iar} and \cite{Gra}).

\begin{fact}[Brian\c con]\label{F:Hilb-Chow}
For $m\geq 1$, the local Hilbert scheme $\Hilb^{m}(k[[x,y]])$ is irreducible of dimension $m-1$.
In particular, the curvilinear local Hilbert scheme $\Hilbc^m(k[[x,y]])\subseteq \Hilb^{m}(k[[x,y]])$ is an open dense
subset.
\end{fact}

Note that Facts \ref{F:Hilb-sur} and \ref{F:Hilb-Chow}, together with \eqref{E:fiberHC}, imply that $\HC$ is a resolution of singularities (see also \cite[Thm. 7.3.4]{FGA}).

\vspace{0,1cm}

Consider now the reduced fiber product $\wt{\Hilb}_S^n:=(S^n\times_{\Sym^n(S)} \Hilb_S^n)_{\rm red}$, i.e. the reduced scheme associated to the fiber product of $S^n$ and $\Hilb_S^n$ over $\Sym^n(S)$. The scheme $\wt{\Hilb}_S^n$ was introduced by M. Haiman \cite[Def. 3.2.4]{Hai} under the name of \emph{isospectral Hilbert scheme} of $S$.
Consider the diagram
\begin{equation}\label{E:diag-iso}
\xymatrix{
\wt{\Hilb}_S^n \ar[r]^{\psi} \ar[d]_{\sigma} & \Hilb_S^n \ar[d]\\
S^n  \ar[r]&  \Sym^n(S).
}
\end{equation}
Clearly, there is a natural action of $\Sigma_n$ on $\wt{\Hilb}_S^n$ that makes $\sigma$ a $\Sigma_n$-equivariant
morphism and $\psi$ a $\Sigma_n$-invariant morphism.
In \cite{Hai}, Haiman proved the following properties of $\wt{\Hilb}^n_S$.

\begin{fact}[Haiman]\label{F:iso-Hilb}
Assume that either ${\rm char}(k)=0$ or ${\rm char}(k)>n. $\footnote{Haiman stated his results in \cite{Hai} under the assumption that ${\rm char}(k)=0$. However,  his results are true also if ${\rm char}(k)>n$ as observed by M. Groechenig in \cite[Rmk. 4.9]{Gro1}.} Then
\noindent
\begin{enumerate}[(i)]
\item The isospectral Hilbert scheme $\wt{\Hilb}^n_S$ is normal, Gorenstein and integral of dimension $2n$.
\item The morphism $\psi:\wt{\Hilb}^n_S\to \Hilb_S^n$ is finite and flat of degree $n!$.
\end{enumerate}
\end{fact}

\vspace{0,1cm}

The inverse image of the curvilinear Hilbert scheme $\Hilbc_S^n\subseteq \Hilb_S^n$ via the map $\psi$ of \eqref{E:diag-iso} admits a modular description that we now recall.
Denote by $\Flag_S^n$ the moduli space of flags
$$D_1\subset\ldots \subset D_n,$$
where $D_i\in \Hilbc_S^n$ has length $i$ for every $i=1,\ldots, n$. There is a natural morphism
\begin{equation}\label{E:psi-cur}
\begin{aligned}
{}^c \psi: \Flag_S^n & \longrightarrow \Hilbc_S^n, \\
(D_1\subset \ldots \subset D_n) & \mapsto D_n.
\end{aligned}
\end{equation}

\begin{fact}\label{F:flag-iso}
Assume that either ${\rm char}(k)=0$ or ${\rm char}(k)>n$. Then there is a cartesian diagram
$$
\xymatrix{
\Flag_S^n \ar[r]^{{}^c\psi}\ar@{^{(}->}[d] \ar@{}[dr]|{\square}& \Hilbc_S^n \ar@{^{(}->}[d]\\
\wt{\Hilb}_S^n \ar[r]_{\psi} & \Hilb_S^n
}
$$
\end{fact}
For a proof, see \cite[Prop. 3.7]{arin2}. Moreover, in loc. cit., it is also shown that
the composition of the inclusion $\Flag_S^n\hookrightarrow \wt{\Hilb}_S^n$ given in Fact \ref{F:flag-iso} with the map $\sigma:\wt{\Hilb}_S^n\to S^n$ of \eqref{E:diag-iso} is equal to the modular map
\begin{equation}\label{E:sig-cur}
\begin{aligned}
{}^c \sigma: \Flag_S^n & \longrightarrow S^n, \\
(D_1\subset \ldots \subset D_n) & \mapsto \left( \supp \ker(\cO_{D_i} \to \cO_{D_{i-1}})\right)_i.
\end{aligned}
\end{equation}

\subsection{Punctual Hilbert scheme of curves with planar singularities}\label{S:Hilb-cur}

Throughout this subsection, we fix a connected projective reduced curve $X$ with planar singularities over an algebraically closed field $k$.

\vspace{0,1cm}

Note that if $D\in \Hilb^n_X$ then its ideal sheaf $I_D$ is a torsion-free rank-1 sheaf on $X$ (in the sense of Definition \ref{D:sheaves}), which is however, in general, neither a line bundle (unless $X$ is smooth) nor simple (unless $X$ is
irreducible). We refer to \cite[Exa. 38]{est1} for an example of $D\in\Hilb^n_X$ with $I_D$ not simple.
We introduce the following subschemes of $\Hilb^d(X)$:
\begin{equation}\label{E:subsch1}
\begin{sis}
& \Hilbr^n_X:=\{D\in \Hilb^n_X\: : \: D \text{ is reduced and contained in the  smooth locus } X_{\rm sm} \subseteq X  \},\\
& \Hilbl^n_X:=\{D\in \Hilb^n_X\: : \: I_D \text{ is a line bundle}\}, \\
& \Hilbs^n_X:=\{D\in \Hilb^n_X\: : \: I_D \text{ is simple}\}.\\
\end{sis}
\end{equation}
By \cite[Prop. 5.15]{AK}, the natural inclusions
\begin{equation}\label{E:Hilb-ls}
\Hilbr^n_X\subseteq \Hilbl^n_X\subseteq \Hilbs^n_X\subseteq \Hilb^n_X
\end{equation}
are open inclusions.

The punctual Hilbert scheme of a curve with planar singularities was studied by Altman-Iarrobino-Kleiman \cite{AIK} and by Brian\c con-Granger-Speder \cite{BGS}, who proved the following

\begin{fact}[Altman-Iarrobino-Kleiman, Brian\c con-Granger-Speder]\label{F:Hilb-cur}
Let $X$ be a connected projective reduced curve with planar singularities.  Then the Hilbert scheme $\Hilb^n_X$
satisfies the following properties:
\begin{enumerate}[(i)]
\item \label{F:Hilb-cur1} $\Hilb^n_X$
is a connected and reduced projective scheme of pure dimension $n$ with locally complete intersection singularities.
\item \label{F:Hilb-cur2} $\Hilbr^n_X$ is dense in $\Hilb^n_X$.
\item \label{F:Hilb-cur3} $\Hilbl^n_X$ is the smooth locus of $\Hilb^n_X$.
\end{enumerate}
\end{fact}
\begin{proof}
Part \eqref{F:Hilb-cur1} follows from \cite[Cor. 7]{AIK} (see also \cite[Prop. 1.4]{BGS}) and \cite[Prop. 3.1]{BGS}. Part \eqref{F:Hilb-cur2} follows from \cite[Prop. 1.4]{BGS}.
Part \eqref{F:Hilb-cur3} follows from \cite[Prop. 2.3]{BGS}.
\end{proof}
Note that the above properties \eqref{F:Hilb-cur2} and \eqref{F:Hilb-cur3} of $\Hilb_X^n$ are inherited by its open subset $\Hilbs_X^n$. This also holds for the reducedness and the locally complete intersection singularities part of \eqref{F:Hilb-cur1}.

\vspace{0,1cm}

The punctual Hilbert scheme of $X$ and the moduli space $\bJbar_X$ are related via the Abel map, which is defined as follows. Given a line bundle $M$ on $X$, we define the \emph{($M$-twisted) Abel map of degree $d$} by
\begin{equation}\label{E:Abel-map}
\begin{aligned}
A_M^d: \Hilbs^d_X & \longrightarrow \bJbar_X, \\
D & \mapsto I_D\otimes M. \\
\end{aligned}
\end{equation}
Note that the image of $A_M(\Hilbs^d_X)$ is contained in $ \bJbar_X^{1-p_a(X)-d+\deg M}$, since for any $D\in \Hilb^d X$ we have that 
$$\chi(I_D\otimes M)=\chi(I_D)+\deg M=\chi(\O_X)-\chi(\O_D)+\deg M=1-p_a(X)-d+\deg M.$$

The following result  shows that, locally on the codomain, the $M$-twisted Abel maps of degree $p_a(X)$ are smooth and surjective (for a suitable choices of $M\in \Pic(X)$) for Gorenstein curves.

\begin{fact}\label{F:Abel-sur}
Let $X$ be a connected projective reduced Gorenstein curve of arithmetic genus $g=p_a(X)$.
For any  $\chi\in \bbZ$, there exists a cover of $\bJbar_X^{\chi}$ by $k$-finite type open subsets $\{U_{\beta}\}$ such that, for each such $U_{\beta}$, there exists $M_{\beta}\in  \Pic^{\chi+2g-1}(X)$ with the property that
 $\Hilbs^{g}_X\supseteq  V_{\beta}:=(A_{M_{\beta}}^g)^{-1}(U_{\beta})\stackrel{A^g_{M_{\beta}}}{\longrightarrow} U_{\beta}$ is smooth and surjective.
\end{fact}
\begin{proof}
See \cite[Prop. 2.5]{MRV1}.
\end{proof}

\begin{remark}\label{R:Ab-irr}
\noindent 
\begin{enumerate}[(i)]
\item The integer $g=p_a(X)$ is the smallest integer for which the above Fact \ref{F:Abel-sur}  is true for any $X$ (see \cite[Rmk. 2.6]{MRV1}).
\item If the curve $X$ is irreducible (and Gorenstein) of arithmetic genus $g$, then we can get a global result although using a bigger punctual Hilbert scheme, namely: 
for any integer $\chi\in \bbZ$ and for any line bundle $M$ on $X$ of degree $3g-2+\chi$, the  $M$-twisted Abel map $A_M^{2g-1}:\Hilbs^{2g-1}(X)=\Hilb^{2g-1}(X)\to \bJbar_X^{\chi}$ is smooth and surjective  (see \cite[Thm.8.6]{AK}). It is easy to see that $2g-1$ is the smallest integer for which the above property holds for any $X$ of arithmetic genus $g$. 
\end{enumerate}
\end{remark}

\vspace{0,2cm}

Consider now the curvilinear Hilbert scheme $\Hilbc_X^n \subset \Hilb_X^n$ of $X$.
Observe that, since $X$ is assumed to have planar singularities, $D\in \Hilb_X^n$ belongs to $\Hilbc_X^n$
if and only if $I_D\not\subseteq I_p^2$ for every $p\in \Xsing$, where $I_p$ denotes the defining ideal of $p\in X$.
Furthermore, if we chose a projective smooth integral surface $S$ such that $X\subset S$ (see Fact \ref{F:cur-sur}), then we have the equality
\begin{equation}\label{E:curXS}
\Hilbc_X^n=\Hilbc_S^n\cap \Hilb_X^n\subset \Hilb_S^n.
\end{equation}

\begin{lemma}\label{L:cod-curv2}
The complement of $\Hilbc^n_X$ inside $\Hilb_X^n$ has codimension at least two.
\end{lemma}
\begin{proof}
First of all, chose a projective smooth integral surface $S$ such that $X\subset S$
(which is possible by Fact \ref{F:cur-sur}).
The Hilbert-Chow morphism of \eqref{E:Hilb-Chow} induces the following commutative diagram
\begin{equation}\label{E:diag-HC}
\xymatrix{
\Hilb_X^n \ar@{^{(}->}[r] \ar[d]_{\HC} & \Hilb_S^n \ar[d]^{\HC} \\
\Sym^n(X) \ar@{^{(}->}[r]& \Sym^n(S).
}
\end{equation}
Note that if $D\in \Hilb_X^n$ is such that $\HC(D)=\sum_i n_ip_i\in \Sym^n(X)$ then $D$ can be written as a disjoint union
$$D=\bigcup_{p_i\in \HC(D)} D_{|p_i},$$
where $D_{|p_i}$ is a $0$-dimensional subscheme of $X$ supported at $p_i$ and of length $n_i$.
We can look at $D_{|p_i}$ as an element of $\Hilb^{n_i}(\wh{O}_{S,p_i})$.
Clearly $D\in \Hilbc_X^n$ if and only if $D_{|p_i}\in \Hilbc^{n_i}(\wh{\O}_{S,p_i})$ for every $p_i\in \HC(D)$.

Consider now an irreducible component $W$ of $\Hilb_X^n\setminus \Hilbc_X^n$ and endow it with the reduced scheme structure. The above discussion implies that  there exists a singular point
$p\in X_{\rm sing}$ and an integer $m\geq 2$ such that for the generic $D\in W$ we have that
\begin{equation*}\label{E:irrW}
\begin{sis}
& mp\subseteq \HC(D), \\
& D_{|p}\in \Hilb^m(\wh{\O}_{S,p})\setminus \Hilbc^m(\wh{\O}_{S,p}).
\end{sis}
\end{equation*}
Therefore there exists an open and dense subset $U\subseteq W$ that admits an embedding
\begin{equation}\label{E:openU}
\begin{aligned}
U &\hookrightarrow \left[\Hilb^m(\wh{\O}_{S,p})\setminus \Hilbc^m(\wh{\O}_{S,p})\right] \times \Hilb^{n-m}_X \\
D & \mapsto \left(D_{|p}, \bigcup_{p\neq q\in \HC(D)} D_{|q}\right).
\end{aligned}
\end{equation}
Fact \ref{F:Hilb-Chow} and Fact \ref{F:Hilb-cur}\eqref{F:Hilb-cur1} imply that the right hand side of \eqref{E:openU} has dimension $m-2+(n-m)=n-2$;
therefore $W$, and hence $\Hilb_X^n\setminus \Hilbc_X^n$, can have dimension at most $n-2$. This concludes the proof since $\Hilb_X^n$ is pure of dimension $n$ by
Fact \ref{F:Hilb-cur}\eqref{F:Hilb-cur1}.

\end{proof}

If we intersect the open subsets of \eqref{E:Hilb-ls} with $\Hilbc_X^n\subset \Hilb_X^n$ we obtain the following
chain of open inclusions:
\begin{equation}\label{E:Hilb-rc}
\Hilbr^n_X\subseteq \Hilbcl_X^n:=\Hilbc_X^n\cap \Hilbl_X^n\subseteq \Hilbcs_X^n:=\Hilbc_X^n\cap \Hilbs_X^n\subseteq \Hilbc^n_X\subseteq \Hilb^n_X.
\end{equation}


\section{Definition of the Poincar\'e sheaf}\label{S:Poinc}

The aim of this section is to introduce the Poincar\'e sheaf $\ov \P$ on $\bJbar_X^{1-g}\times \bJbar_X^{1-g}$, where $X$ is a reduced connected projective curve of arithmetic genus $g:=p_a(X)$. 
Let us start by describing the restriction of $\ov \P$ to the open subset $\Open:=\bJbar_X^{1-g}\times \bJ_X^{1-g}\cup \bJ_X^{1-g}\times \bJbar_X^{1-g}\subseteq \bJbar_X^{1-g}\times \bJbar_X^{1-g}$ consisting of pairs of torsion-free rank-$1$  simple sheaves $I$ on $X$ of Euler characteristic $1-g$ (or equivalently degree $0$) such that at least one of the two sheaves is a line bundle.

\subsection{The Poincar\'e line bundle $\P$}\label{S:Poinc1}

Consider the open subset $X\times \Open \subseteq X\times \bJbar_X^{1-g}\times \bJbar_X^{1-g}$ and, for any $1\leq i<j\leq 3$, denote by $p_{ij}$ the projection onto the product of the $i$-th and $j$-th factors.
Consider the trivial family of curves
$$p_{23}:X\times\Open\to \Open.$$
For any coherent sheaf $\F$ on $X\times\Open$, flat over $\Open$, the complex $Rp_{23*}(\F)$ is perfect of amplitude $[0,1]$, i.e. there is a Zariski open cover $\Open=\bigcup_{\alpha} U_{\alpha}$ and, for each open subset $U_{\alpha}$, a complex $\G_*^{\alpha}:=\{\G_0^{\alpha}\to \G_1^{\alpha}\}$ of locally free sheaves of finite rank over $U_{\alpha}$ which is quasi-isomorphic to $Rp_{23*}(\F)_{|U_{\alpha}}$ (see \cite[Observation 43]{est1}). The line bundles
$\det(\G_*^{\alpha}):=\det(\G_0^{\alpha})\otimes \det(\G_1^{\alpha})^{-1}$ on $U_{\alpha}$ glue together to give a (well-defined) line bundle on $\Open$, which is denoted by $\D_{p_{23}}(\F)$ and is called the \emph{determinant of cohomology of} $\F$ \emph{with respect to} $p_{23}$ (see \cite[Sec. 6.1]{est1} for details).

Choose now a universal sheaf $\I$ on $X\times \bJbar_X^{1-g}$ as in \S\ref{S:comp-Jac}
and form the line bundle on $\Open$, called the \emph{Poincar\'e line bundle}:
\begin{equation}\label{E:Poin-sheaf}
\P:=\D_{p_{23}}(p_{12}^*\I\otimes p_{13}^*\calI)^{-1}\otimes \D_{p_{23}}(p_{12}^*\I)\otimes \D_{p_{23}}(p_{13}^*\I).
\end{equation}

\begin{remark}\label{R:why-open}
The above definition of $\P$  makes sense since $p_{12}^*\I$ and $p_{13}^*\I$ are coherent sheaves flat over $\Open$ (because $\I$ is coherent and flat over $\bJbar_X^{1-g}$) and $p_{12}^*\I\otimes p_{13}^*(\I)$ is flat over $\Open$
since $p_{12}^*(\I)$ is a line bundle over $X\times \bJ_X^{1-g}\times \bJbar_X^{1-g}$ and $p_{13}^*(\I)$ is a line bundle over $X\times \bJbar_X^{1-g}\times \bJ_X^{1-g}$. However $p_{12}^*(\I)\otimes p_{13}^*(\I)$ is not flat over $\bJbar_X^{1-g}\times \bJbar_X^{1-g}$ (in general), hence definition \eqref{E:Poin-sheaf} does not extend over $\bJbar_X^{1-g}\times \bJbar_X^{1-g}$.
\end{remark}

\begin{remark}\label{R:dep-sheaf}
The above definition of $\P$ is independent of the chosen universal sheaf $\I$ since we are working over
$\Open\subseteq \bJbar_X^{1-g}\times \bJbar_X^{1-g}$.
Indeed, one could define a Poincar\'e line bundle $\P$ on $\bJbar_X\times \bJ_X\cup \bJ_X\times \bJbar_X$ using the same formula \eqref{E:Poin-sheaf}. Then, considering another universal sheaf $\wt{\I}=\I\otimes \pi_2^*(N)$ for some $N\in \Pic(\bJbar_X)$ (see \S\ref{S:comp-Jac}) and defining a new Poincar\'e line bundle
$\wt\P$ with respect to $\wt\I$, one can check that (see \cite[Prop. 2.2]{egk} for a similar computation)
\begin{equation}\label{E:chan-Poin}
\wt{\P}_{|\bJbar_X^{\chi_1}\times \bJ_X^{\chi_2}}=\P_{|\bJbar_X^{\chi_1}\times \bJ_X^{\chi_2}}\otimes p_1^*(N_{|\bJbar_X^{\chi_1}})^{1-g-\chi_2}\otimes p_2^*(N_{|\bJ_X^{\chi_2}})^{1-g-\chi_1}
\end{equation}
for every $\chi_1, \chi_2 \in \bbZ$, where $p_1$ (resp. $p_2$) denote the projection of $\bJbar_X\times \bJ_X$ onto $\bJbar_X$ (resp. $\bJ_X$).



\end{remark}


The following Lemma will be used throughout the following.

\begin{lemma}\label{L:compu-det}
Let $S$ be a scheme and consider the trivial family $p_2:X\times S\to S$. Let $I$ be a rank-$1$ torsion-free sheaf on $X$ and write $I=I_{E_1}\otimes I_{E_2}^{-1}$ as in Lemma \ref{L:sheaves}. Let $\F$ be a coherent sheaf on $X\times S$, flat over $S$, and assume that $\F$ is locally free along $p_1^{-1}(E_i)$ for $i=1,2$. Then
\begin{equation}\label{E:magic-for}
\D_{p_2}(\F\otimes p_1^*I)\otimes \D_{p_2}(\F)^{-1}=\D_{p_2}\left(\F_{|p_1^{-1}(E_2)}\right)\otimes \D_{p_2}\left(\F_{|p_1^{-1}(E_1)}\right)^{-1}.
\end{equation}
\end{lemma}
\begin{proof}
Consider the exact sequences associated to the effective divisors $E_i\subseteq X$:
\begin{equation}\label{E:succE1}
0\to I_{E_1} \to \O_{X} \to \O_{E_1}\to 0,
\end{equation}
\begin{equation}\label{E:succE2}
0\to I_{E_2}=\O(-E_2) \to \O_{X} \to \O_{E_2}\to 0.
\end{equation}
Tensoring the sequence \eqref{E:succE2} with $I$ and using the fact that $I=I_{E_1}\otimes I_{E_2}^{-1}$, we get the new sequence
\begin{equation}\label{E:succE12}
0\to I_{E_2}\otimes I=I_{E_1} \to I \to I_{|E_2}=\O(E_2)_{|E_2}=\O_{E_2}\to 0,
\end{equation}
which remains exact since $I_{E_2}^{-1}=\O(E_2)$ is a line bundle and $(I_{E_1})_{|E_2}=\O_{E_2}$ because $E_1$ and $E_2$ have disjoint supports by construction.
By pulling back via $p_{1}: X\times S \to X$ the exact sequences \eqref{E:succE1} and \eqref{E:succE12} and tensoring
them with $\F$, we get the following two sequences which remain exact by the hypothesis on $\F$:
\begin{equation*}\label{E:seq-F}
\begin{sis}
& 0\to \F\otimes p_1^*I_{E_1} \to \F \to \F_{|p_1^{-1}(E_1)}\to 0, \\
& 0\to \F\otimes p_1^*I_{E_1} \to \F\otimes p_1^*I \to \F_{|p_1^{-1}(E_2)}\to 0.
\end{sis}
\end{equation*}
Using the additivity of the determinant of cohomology (see \cite[Prop. 44(4)]{est1}), we get
\begin{equation*}
\begin{sis}
& \D_{p_2}\left(\F\right)=
\D_{p_2}\left(\F\otimes p_{1}^{*}I_{E_1}\right) \otimes \D_{p_2}\left(\F_{|p_{1}^{-1}(E_1)}\right), \\
& \D_{p_2}\left(\F\otimes p_{1}^*I\right)= \D_{p_2}\left(\F\otimes p_{1}^*I_{E_1}\right)\otimes  \D_{p_2}\left(\F_{|p_{1}^{-1}(E_2)}\right).\\
\end{sis}
\end{equation*}
By taking the difference of the above two equalities, we get the desired formula \eqref{E:magic-for}.

\end{proof}

\begin{cor}\label{C:compdet}
Same assumptions as in Lemma \ref{L:compu-det}. Moreover, let $L$ be a line bundle on $X$.
Then
\begin{equation*}
\D_{p_2}(\F\otimes p_1^*L\otimes p_1^*I)= \D_{p_2}(\F\otimes p_1^*L)\otimes
\D_{p_2}(\F\otimes p_1^*I)\otimes \D_{p_2}(\F)^{-1}.
\end{equation*}
\end{cor}
\begin{proof}
It follows from the previous Lemma \ref{L:compu-det} together with the fact that (for $i=1,2$):
$$p_1^*L_{|p_1^{-1}(E_i)}=p_1^*(L_{|E_i})=p_1^*\O_{E_i}=\O_{p_1^{-1}(E_i)}.$$
\end{proof}

\subsection{Definition of the Poincar\'e sheaf  $\ov\P$ }\label{S:Poinc2}

In this subsection we construct a Poincar\'e sheaf $\ov\P$ on $\bJbar_X^{1-g}\times \bJbar_X^{1-g}$, which is an extension of $\P$.
The definition of $\ov\P$ is as follows.

\begin{defi}\label{D:Poinc}
Let $X$ be a reduced and connected curve.
Denote by
$$j:\Open=\bJ_X^{1-g}\times \bJbar_X^{1-g}\cup \bJbar_X^{1-g}\times \bJ_X^{1-g}\hookrightarrow \bJbar_X^{1-g}\times \bJbar_X^{1-g}$$
the natural inclusion.
The $\O_{\bJbar_X^{1-g}\times \bJbar_X^{1-g}}$-module
$$\ov\P:=j_*(\P)$$
is called the \emph{Poincar\'e sheaf}.
\end{defi}

If $X$ has planar singularities, then the Poincar\'e sheaf enjoys the following properties.

\begin{thm}\label{T:Poinc}
Let $X$ be a reduced connected curve with planar singularities of arithmetic genus $g:=p_a(X)$. Assume that either ${\rm char}(k)=0$ or  ${\rm char}(k)>g$.
\begin{enumerate}[(i)]
\item \label{T:Poinc1} $\ov\P$ is a maximal Cohen-Macaulay (coherent) sheaf on $\bJbar_X^{1-g}\times \bJbar_X^{1-g}$;
\item \label{T:Poinc2} $\ov\P$ is  flat with respect to the second projection $p_2:\bJbar_X^{1-g}\times \bJbar_X^{1-g}\to \bJbar_X^{1-g}$ and, for every $I\in \bJbar_X^{1-g}$, the restriction
$\ov \P_{I}:=\ov\P_{|\bJbar_X^{1-g}\times \{I\}}$ is a maximal Cohen-Macaulay sheaf on $\bJbar_X^{1-g}$.
\end{enumerate}
\end{thm}

\begin{remark}\label{R:CM-ext}
Under the assumptions of Theorem \ref{T:Poinc}, we observe that, since the complement of $\Open$
inside $\bJbar_X^{1-g}\times \bJbar_X^{1-g}$ has codimension greater or equal than two by Fact \ref{F:propJbig}\eqref{F:propJbig2}, the sheaf $\ov\P$ is the Cohen-Macaulay extension of $\P$; in other words, $\ov\P$
can be characterized as  the unique Cohen-Macauly sheaf on $\bJbar_X^{1-g}\times \bJbar_X^{1-g}$ whose restriction to $\Open$  is equal to $\P$ (see \cite[Thm. 5.10.5]{EGAIV2}).
\end{remark}

In proving Theorem \ref{T:Poinc}, we will adapt the strategy used by D. Arinkin \cite{arin2} to prove the same result for integral curves: we will first construct a sheaf $\Q^n$ on
$\Hilb_X^n\times \bJbar_X^{1-g}$ for any $n\geq 1$ and then we will descend it to a sheaf $\ov\P$ on $\bJbar_X^{1-g}\times \bJbar_X^{1-g}$ using the Abel map \eqref{E:Abel-map}.

\subsubsection{{\bf The sheaf $\Q$ on $\displaystyle \left(\coprod_{n\in \bbN} \Hilb^n_X\right)\times \bJbar_X^{1-g}$}}\label{S:sheafQ}

Choose an embedding $i:X\hookrightarrow S$ as in Fact \ref{F:cur-sur}, fix an integer $n\in \bbN$ and consider the following diagram (using the notations of \S\ref{S:Hilb-sur}):
\begin{equation}\label{E:diag-Q}
\xymatrix{
\Hilb^n_X\times \bJbar_X^{1-g} \ar@{^{(}->}[r] & \Hilb^n_S\times \bJbar_X^{1-g} \ar[d]^{p_1} &
\wt{\Hilb}^n_S\times \bJbar_X^{1-g} \ar[l]_{\psi\times \id} \ar[r]^{\sigma\times \id} & S^n\times \bJbar_X^{1-g} &
X^n\times \bJbar_X^{1-g} \ar@{_{(}->}[l]_{i^n\times \id}\\
& \Hilb^n_S &&&
}
\end{equation}
where the maps $\sigma\times \id$ and $i^n\times \id$ are clearly $\Sigma_n$-equivariant.

Choose a universal sheaf $\I$ on $X\times \bJbar_X^{1-g}$ as in \S\ref{S:comp-Jac} and define a sheaf
$\I^n$ on $X^n\times \bJbar_X^{1-g}$ by:
\begin{equation}\label{E:sheaf-In}
\I^n:=p_{1,n+1}^*(\I)\otimes \ldots \otimes p_{n,n+1}^*(\I),
\end{equation}
where $p_{i,n+1}:X^n\times \bJbar_X^{1-g}\to X\times \bJbar_X^{1-g}$ denotes the projection onto the $i$-th and $(n+1)$-th factor. Observe that $\I^n$ is clearly $\Sigma_n$-equivariant.

Define now a coherent sheaf $\Q^n$ on $\Hilb_S^n\times \bJbar_X^{1-g}$ by the formula
\begin{equation}\label{E:sheaf-Q}
\Q^n:=\left[(\psi\times \id)_*(\sigma\times \id)^*(i^n\times \id)_* \I^n \right]^{\sign}\otimes p_1^*(\det \A)^{-1},
\end{equation}
where $\A$ is the locally free rank $n$ sheaf on $\Hilb_S^n$ defined after the diagram \eqref{E:univ-div}
and the upper index $\sign$ stands for the space of anti-invariants with respect to the natural action of $\Sigma_n$.

The sheaf $\Q^n$ enjoys the following properties, as shown by Arinkin in \cite[Prop. 4.1, Section 4.1]{arin2}\footnote{Arinkin stated his results in \cite{arin2} under the assumption that ${\rm char}(k)=0$. However,  his results are true also if ${\rm char}(k)>n$ as observed by M. Groechenig in \cite[Rmk. 4.9]{Gro1}.}.

\begin{fact}[Arinkin]\label{F:prop-Q}
Assume that either ${\rm char}(k)=0$ or ${\rm char}(k)>n$.
Let $X$ be a connected projective reduced curve with planar singularities and choose
 an embedding $i:X\hookrightarrow S$ as in Fact \ref{F:cur-sur}. Then the sheaf $\Q^n$ defined by \eqref{E:sheaf-Q} satisfies the following properties:
\begin{enumerate}[(i)]
\item \label{F:prop-Q1} $\Q^n$ is supported schematically on $\Hilb^n_X\times \bJbar_X^{1-g}$ and it does not depend on the chosen embedding $i:X\hookrightarrow S$.
\item \label{F:prop-Q2} $\Q^n$ is a maximal Cohen-Macaulay sheaf on $\Hilb_X^n\times \bJbar_X^{1-g}$.
\item \label{F:prop-Q3} $\Q^n$ is flat over $\bJbar_X^{1-g}$.
\item \label{F:prop-Q4} For any $I\in \bJbar_X^{1-g}$, the restriction $\Q^n_{|\Hilb_X^n\times \{I\}}$ is a maximal Cohen-Macaulay sheaf on $\Hilb_X^n$.
\end{enumerate}
\end{fact}

Denote by $\Q$ the sheaf on $\displaystyle \left(\coprod_{n\in \bbN} \Hilb^n_X\right)\times \bJbar_X^{1-g}$ which is equal to $\Q^n$ on $\Hilb_X^n\times \bJbar_X^{1-g}$.

\subsubsection{{\bf The sheaf $\Q'$ on $\displaystyle \left(\coprod_{n\in \bbN} \Hilb^n_X\right)\times \bJbar_X^{1-g}$}}\label{S:sheafQ'}

The restriction of $\Q^n$ to $\Hilbc^n_X\times \bJbar_X^{1-g}$ coincides the restriction of another sheaf $\Q'^n$ on $\Hilb_X^n\times \bJbar_X^{1-g}$ that we now introduce.

Consider the universal divisor $\D\subseteq \Hilb_X^n\times X$; see \eqref{E:univ-div}.
Recall that  $\A:=h_*\O_{\D}$ is a coherent sheaf of algebras on $\Hilb^n_X$ which is locally free
of rank $n$. Denote by $\A^{\times}$ the subsheaf of $\A$ of invertible elements. Clearly, $\A^{\times}$ is the sheaf of sections of a flat abelian group scheme over $\Hilb_X^n$, whose fiber over
$D \in \Hilb_X^n$ is canonically isomorphic to the group $k[D]^{\times}$ of invertible elements of the algebra $k[D]$ of regular functions on $D$.
Clearly $\A^{\times}$ acts on $\A$ and therefore also on the line bundle $\det \A$; the action of $\A^{\times}$ on $\det \A$ is given by the norm character $\N:\A^{\times}\to \O^{\times}$.

Consider the pull-back $p_1^{-1}(\A^{\times})$ (resp. $p_1^{-1}(\A)$) of $\A^{\times}$ (resp. $\A$) to $\Hilb_X^n\times\bJbar_X^{1-g}$. For any sheaf $\F$ of $p_1^{-1}(\A)$-algebras on
$\Hilb_X^n\times \bJbar_X^{1-g}$, we will denote by $\F_{\N}$  the maximal  quotient of $\F$ on which $p_1^{-1}(\A^{\times})$ acts via the norm character $\N$.

Define now a coherent sheaf $\Q'^n$ on $\Hilb_X^n\times \bJbar_X^{1-g}$ by the formula
\begin{equation}\label{E:sheaf-Q'}
\Q'^n:=\left[\bigwedge^n(h\times \id)_*(f\times \id)^* \I \right]_{\N}\otimes p_1^*(\det \A)^{-1}=
\end{equation}
\begin{equation*}
=\left[\bigwedge^n(h\times \id)_*(f\times \id)^* \I \right]_{\N} \otimes \left[\bigwedge^n(h\times \id)_*
\O_{\D\times \bJbar_X^{1-g}}\right] ,
\end{equation*}
where $\I$ is a universal sheaf  on $X\times \bJbar_X^{1-g}$ as in \S\ref{S:comp-Jac} and where the maps involved in the above formula are collected in the diagram below.

\begin{equation*}
\xymatrix{
&& \D\times \bJbar_X^{1-g} \ar@{^{(}->}[r]\ar[dl]_{h\times \id}\ar[dr]^{f\times \id} & \Hilb_X^n\times X\times \bJbar_X^{1-g}\\
\Hilb_X^n & \Hilb_X^n \times \bJbar_X^{1-g}\ar[l]_{p_1}& & X\times \bJbar_X^{1-g}
}
\end{equation*}

\begin{remark}\label{R:Q'-lb}
The restriction of $\Q'^n$ to the open subset $\Hilb_X^n\times \bJ_X^{1-g}\cup \Hilbr_X^n\times \bJbar_X^{1-g}\subseteq \Hilb_X^n\times \bJbar_X^{1-g}$ is a line bundle and is equal to \begin{equation}\label{E:restr-Q'}
\Q'^n_{|\Hilb_X^n\times \bJ_X^{1-g}\cup \Hilbr_X^n\times \bJbar_X^{1-g}}=\det\left((h\times \id)_*(f\times \id)^* \I\right) \otimes \det((h\times \id)_*\O_{\D\times \bJbar_X^{1-g}}) ^{-1}.
\end{equation}
Indeed, the universal sheaf $\I$ is a line bundle on the open subset
$$(f\times \id)(h\times \id)^{-1}\left(\Hilb_X^n\times \bJ_X^{1-g}\cup \Hilbr_X^n\times \bJbar_X^{1-g}\right)=X\times \bJ_X^{1-g}\cup X_{\rm sm}\times \bJbar_X^{1-g}\subseteq X\times \bJbar_X^{1-g}.$$
This implies that $\bigwedge^n(h\times \id)_*(f\times \id)^* \I=\det\left((h\times \id)_*(f\times \id)^* \I\right) $
is a line bundle on $\Hilb_X^n\times \bJ_X^{1-g}\cup \Hilbr_X^n\times \bJbar_X^{1-g}$ on which $p_1^{-1}(\A^{\times})$ acts via the norm character $\N$. The expression \eqref{E:restr-Q'} now follows.
\end{remark}

The relation between $\Q^n$ and $\Q'^n$ is clarified by the following result of D. Arinkin (see \cite[Prop. 4.4, Section 4.2]{arin2}\footnote{The result in  \cite[Prop. 4.4]{arin2} is stated only for an integral curve $X$ (with locally planar singularities). However, the proof of loc. cit. consists in choosing an embedding of $X$ into a smooth and projective surface $S$ and then using \cite[Lemma 3.6]{arin2} which  is a statement about $\Hilbc_S^n$. Therefore, the same proof works for a reduced curve $X$ with locally planar singularities using Fact \ref{F:cur-sur}.}).

\begin{fact}[Arinkin]\label{F:Q-Q'}
Assume that either ${\rm char}(k)=0$ or ${\rm char}(k)>n$. Let $X$ be a connected projective reduced curve with planar singularities.

The sheaves $\Q^n$ and $\Q'^n$ coincide on the open subset $\Hilbc_X^n\times \bJbar_X^{1-g}\subseteq \Hilb_X^n\times \bJbar_X^{1-g}$.
\end{fact}

Denote by $\Q'$ the sheaf on $\displaystyle \left(\coprod_{n\in \bbN} \Hilb^n_X\right)\times \bJbar_X^{1-g}$ which is equal to $\Q'^n$ on $\Hilb_X^n\times \bJbar_X^{1-g}$.

\begin{remark}\label{R:Q-I}
The sheaves $\Q^n$ and $\Q'^n$ depend on the choice of the universal sheaf $\I$ on $X\times \bJbar_X^{1-g}$.
By taking another universal sheaf $\wt{\I}=\I\otimes \pi_2^*(N)$ for some $N\in \Pic(\bJbar_X^{1-g})$ (see \S\ref{S:comp-Jac}) and defining $\wt{\Q}^n$ and $\wt{\Q}'^n$ by replacing $\I$ with $\wt{\I}$ in formulas \eqref{E:sheaf-Q} and \eqref{E:sheaf-Q'}, then we have that
$$\begin{sis}
& \wt{\Q}^n=\Q^n\otimes \pi_2^*(N)^{\otimes n}, \\
& \wt{\Q}'^n=\Q'^n\otimes \pi_2^*(N)^{\otimes n}.
\end{sis}$$

\end{remark}

\subsubsection{{\bf The relation between $\Q'$ and $\P$}}\label{S:Q'-P}

We want now to compare the sheaf $\Q'$ to the Poincar\'e line bundle $\P$ on
$\displaystyle \Open= \bJ_X^{1-g}\times \bJbar_X^{1-g}\cup \bJbar_X^{1-g}\times \bJ_X^{1-g}$ (see \S\ref{S:Poinc1}) via the Abel map of \eqref{E:Abel-map}.

Consider an open cover $\bJbar_X^{1-g}=\bigcup_{\beta} U_{\beta}$ as in Fact \ref{F:Abel-sur}, in such a way that  for each $U_{\beta}$ there exists $M_{\beta}\in \Pic^{g}(X)$ with the property that $V_{\beta}:=(A^g_{M_{\beta}})^{-1}(U_{\beta})\stackrel{A_{M_{\beta}}^g}{\longrightarrow} U_{\beta}$ is smooth and surjective.
Fix one such $U_{\beta}$ and consider the smooth and surjective map
\begin{equation}\label{E:Am-id}
 \Hilbs_X^{g}\times \bJbar_X^{1-g} \supseteq V_{\beta}\times \bJbar_X^{1-g}\stackrel{A_{M_{\beta}}^g\times \id}{\longrightarrow} U_{\beta}\times \bJbar_X^{1-g}\subset \bJbar_X^{1-g}\times \bJbar_X^{1-g}.
\end{equation}
Define the open subset
\begin{equation}\label{E:openWm}
W_{\beta}:=\left[\left(\Hilbr_X^{g}\cap V_{\beta}\right)\times \bJbar_X^{1-g}\right] \cup
\left[ V_{\beta} \times \bJ_X^{1-g} \right] \subseteq V_{\beta}\times \bJbar_X^{1-g}\subseteq \Hilbs_X^{g}\times \bJbar_X^{1-g}.
\end{equation}
and observe that $(A_{M_{\beta}}^g\times \id)(W_{\beta})\subseteq \Open$.

\begin{prop}\label{P:comp-QP}
Same notations as above. Assume that either ${\rm char}(k)=0$ or that ${\rm char}(k)>g$.
The restrictions of $\Q'$ and of $(A_{M_{\beta}}^g\times \id)^*\P$ to $W_{\beta}$ differ by the
 pull-back of a line bundle from $\bJbar_X^{1-g}$.
\end{prop}
\begin{proof}
Denote by $\pi_{ij}$ and $\pi_i$ the projections of $X\times \Hilb_X^{g}\times \bJbar_X^{1-g}$ (or of its open subsets
$X\times W_{\beta}\subseteq X\times V_{\beta}\times \bJbar_X^{1-g}$) onto the factors corresponding to the subscripts and consider the following commutative diagram
\begin{equation}\label{E:diag-P-Q}
\xymatrix{
& X\times \bJbar_X^{1-g} & & \\
\D\times \bJbar_X^{1-g} \ar@{^{(}->}[r] \ar[ur]^{f\times \id}\ar[dr]_{h\times \id}& X\times \Hilb_X^{g}\times \bJbar_X^{1-g} \ar[u]^{\pi_{13}}\ar[d]_{\pi_{23}} & X\times W_{\beta} \ar@{_{(}->}[l] \ar[d]^{\pi_{23}}
\ar[rr]^(0.4){\id\times A_{M_{\beta}}^g\times \id} & & X\times \Open \ar[d]^{p_{23}}\\
& \Hilb_X^{g}\times \bJbar_X^{1-g}  & W_{\beta} \ar@{_{(}->}[l] \ar[rr]^(0.4){A_{M_{\beta}}^g\times \id} & & \Open
}
\end{equation}

From the definition of the Abel map \eqref{E:Abel-map}, it follows that the pull-back of the universal sheaf $\I$ via the map $\id\times A_{M_{\beta}}^g:X\times V_{\beta}\to X\times U_{\beta}\subseteq X\times \bJbar_X^{1-g}$ is equal to
\begin{equation}\label{E:back-I}
(\id\times A_{M_{\beta}}^g)^*\I=\I(\D)_{|X\times V_{\beta}}\otimes p_1^*(M_{\beta})\otimes p_2^*(N),
\end{equation}
where $\I(\D)$ is the ideal sheaf of the universal divisor $\D\subset X\times \Hilb_X^{g}$, $ p_1$ and $p_2$ are the projection maps from $X\times U_{\beta}$ onto $X$ and $U_{\beta}$, respectively, and $N$ is some line bundle on $V_{\beta}$.

Applying the base change property of the determinant of cohomology (see \cite[Prop. 44(1)]{est1}) to the definition
\eqref{E:Poin-sheaf} of $\P$ and using \eqref{E:back-I}, we get that
\begin{equation}\label{E:back-P}
(A_{M_{\beta}}^g\times \id)^*\P=\D_{\pi_{23}}(\pi_{12}^*\I(\D)\otimes \pi_1^*M_{\beta}\otimes \pi_2^*N \otimes \pi_{13}^*\calI)^{-1}\otimes \D_{\pi_{23}}(\pi_{12}^*\I(\D)\otimes \pi_1^*M_{\beta}\otimes \pi_2^*N) \otimes \D_{\pi_{23}}(\pi_{13}^*\I).
\end{equation}
Applying the projection property of the determinant of cohomology (see \cite[Prop. 44(3)]{est1}) and using that
$\pi_{12}^*\I(\D)\otimes \pi_1^*M_{\beta}$ and $\pi_{13}^*\calI$ have relative  Euler characteristic equal to $1-g$, we get
\begin{equation}\label{E:indep-N}
\begin{sis}
& \D_{\pi_{23}}(\pi_{12}^*\I(\D)\otimes \pi_1^*M_{\beta}\otimes \pi_2^*N \otimes \pi_{13}^*\calI)= \D_{\pi_{23}}(\pi_{12}^*\I(\D)\otimes \pi_1^*M_{\beta} \otimes \pi_{13}^*\calI)\otimes (\pi_2^*N)^{1-g},\\
& \D_{\pi_{23}}(\pi_{12}^*\I(\D)\otimes \pi_1^*M_{\beta}\otimes \pi_2^*N)=
\D_{\pi_{23}}(\pi_{12}^*\I(\D)\otimes \pi_1^*M_{\beta})\otimes (\pi_2^*N)^{1-g}.\\
\end{sis}
\end{equation}
Substituting \eqref{E:indep-N} into \eqref{E:back-P}, we get
\begin{equation}\label{E:back-P2}
(A_{M_{\beta}}^g\times \id)^*\P=\D_{\pi_{23}}(\pi_{12}^*\I(\D)\otimes \pi_1^*M_{\beta} \otimes \pi_{13}^*\calI)^{-1}\otimes \D_{\pi_{23}}(\pi_{12}^*\I(\D)\otimes \pi_1^*M_{\beta}) \otimes \D_{\pi_{23}}(\pi_{13}^*\I).
\end{equation}

\un{CLAIM:} The two line bundles on $\Hilb_X^{g}\times \bJbar_X^{1-g}$
$$
\begin{aligned}
\M:= &\D_{\pi_{23}}(\pi_{12}^*\I(\D)\otimes \pi_1^*M_{\beta} \otimes \pi_{13}^*\calI)^{-1}\otimes \D_{\pi_{23}}(\pi_{12}^*\I(\D)\otimes \pi_1^*M_{\beta}) \otimes \D_{\pi_{23}}(\pi_{13}^*\I),\\
\N:= &\D_{\pi_{23}}(\pi_{12}^*\I(\D) \otimes \pi_{13}^*\calI)^{-1}\otimes \D_{\pi_{23}}(\pi_{12}^*\I(\D))
\otimes \D_{\pi_{23}}(\pi_{13}^*\I),
\end{aligned}
$$
differ by the pull-back of a line bundle from $\bJbar_X^{1-g}$.

Indeed, since $\Hilb_X^{g}$ is a connected and reduced projective scheme (by Fact \ref{F:Hilb-cur}) and
$\bJbar_X^{1-g}$ is reduced and locally Noetherian (by Fact \ref{F:propJbig} and Fact \ref{F:Este-Jac}), the Claim will follow from the seesaw principle (see \cite[Sec. II.5, Cor. 6]{Mum}) if we show that
\begin{equation}\label{E:eq-fiber}
\M_{|\Hilb_X^{g}\times \{I\}}= \N_{|\Hilb_X^{g}\times \{I\}} \text{ for any } I\in \bJbar_X^{1-g}.
\end{equation}
By the base change property of the determinant of cohomology, we get
\begin{equation}\label{E:rest-MN}
\begin{sis}
\M_{|\Hilb_X^{g}\times \{I\}}= &\D_{\pi_2}(\I(\D)\otimes \pi_1^*M_{\beta} \otimes \pi_{1}^*I)^{-1}\otimes
\D_{\pi_2}(\I(\D)\otimes \pi_1^*M_{\beta}),\\
\N_{|\Hilb_X^{g}\times \{I\}}= &\D_{\pi_2}(\I(\D) \otimes \pi_{1}^*I)^{-1}\otimes \D_{\pi_2}(\I(\D)),
\end{sis}
\end{equation}
where $\pi_i$ (for $i=1,2$) is the projection of $X\times \Hilb_X^{g}$ onto the $i$-th factor. Using these formulas, the equality \eqref{E:eq-fiber} follows from Corollary \ref{C:compdet}.

\vspace{0,2cm}

Consider now the exact sequence associated to the universal divisor $\D\subset X\times \Hilb_X^{g}$:
\begin{equation}\label{E:seqD}
0 \to \I(\D) \to \O_{X\times \Hilb_X^{g}}\to \O_{\D}\to 0.
\end{equation}
By pulling back \eqref{E:seqD} via $\pi_{12}:X\times W_{\beta}\to X\times \Hilb_X^{g}$,
tensoring it either with $ \pi_{13}^*\calI$  (it remains exact since, by the definition of $W_{\beta}$, $ \pi_{13}^*\calI$ is a line bundle on $\pi_{12}^{-1}(D)\cap (X\times W_{\beta})\subseteq (X_{\rm sm}\times\Hilbr_X^{g}\times\bJbar_X^{1-g})\cup(X\times\Hilb_X^{g}\times\J_X^0)$) and using the additivity property of the determinant of cohomology, we get
\begin{equation}\label{E:add-D}
\begin{sis}
& \D_{\pi_{23}}(\pi_{12}^*\I(\D) \otimes \pi_{13}^*\calI)^{-1}=
\D_{\pi_{23}}( \pi_{13}^*\calI)^{-1}\otimes
\D_{\pi_{23}}\left(( \pi_{13}^*\calI)_{|\pi_{12}^{-1}(\D)}\right),\\
&  \D_{\pi_{23}}(\pi_{12}^*\I(\D))=\D_{\pi_{23}}(\O_{X\times W_{\beta}}) \otimes \D_{\pi_{23}}\left(\O_{\pi_{12}^{-1}(\D)}\right)^{-1}.\\
\end{sis}
\end{equation}
By the base change property of the determinant of cohomology, we get that
\begin{equation}\label{E:null}
\D_{\pi_{23}}(\O_{X\times W_{\beta}})=\O_{W_{\beta}}.
\end{equation}
By the definition of $W_{\beta}$, the restriction of the sheaf $ \pi_{13}^*\calI$ over the relative divisor $\pi_{12}^{-1}(\D)\subset X\times W_{\beta}\to W_{\beta}$ is a line bundle. Therefore
$\pi_{23*}(\pi_{13}^*\calI_{|\pi_{12}^{-1}(\D)})$ and $\pi_{23*}(\O_{\pi_{12}^{-1}(\D)})$
are locally free sheaves of rank $g$ over $W_{\beta}$.
From the definition of the determinant of cohomology (see \S\ref{S:Poinc1}) and the commutative diagram \eqref{E:diag-P-Q}, it follows that over $W_{\beta}$ we have the equality
\begin{equation}\label{E:semp3}
\begin{sis}
\D_{\pi_{23}}\left( \pi_{13}^*\calI_{|\pi_{12}^{-1}(\D)}\right)=&
\det\left[\pi_{23*}\left(\pi_{13}^*\calI_{|\pi_{12}^{-1}(\D)}\right)\right]
 =\det\left[(h\times \id)_*(f\times \id)^*\I\right],   \\
 \D_{\pi_{23}}\left(\O_{\pi_{12}^{-1}(\D)}\right)=&
\det\left[\pi_{23*}\left(\O_{\pi_{12}^{-1}(\D)}\right)\right]=
 \det\left[(h\times \id)_*\O_{\D\times \bJbar_X^{1-g}}\right].   \\
\end{sis}
\end{equation}
Observe that since $W_{\beta}$ is  contained in the open subset $\Hilb_X^{g}\times \bJ_X^{1-g}\cup \Hilbr_X^{g}\times \bJbar_X^{1-g}\subseteq \Hilb_X^{g}\times \bJbar_X^{1-g}$, then the restriction of
$\Q'$ to $W_{\beta}$ is given by the expression \eqref{E:restr-Q'}. Therefore, using \eqref{E:add-D}, \eqref{E:null} and \eqref{E:semp3}, we get
\begin{equation}\label{E:eq-QN}
\Q'_{|W_{\beta}}=\Q'^{g}_{|W_{\beta}}=\N_{|W_{\beta}}.
\end{equation}
We now conclude the proof by combining \eqref{E:eq-QN}, \eqref{E:back-P2} and the above Claim.

\end{proof}

\begin{remark}
Note that we cannot hope to have equality in Proposition \ref{P:comp-QP} since $\Q'$ is only well-defined up to the pull-back of a line bundle from $\bJbar_X^{1-g}$ (see Remark \ref{R:Q-I}) while $\P$ is well-defined (see Remark \ref{R:dep-sheaf}).
\end{remark}

\vspace{0,2cm}

\subsubsection{{\bf Descending $\Q$ to $\ov\P$}}

We are now ready to prove Theorem \ref{T:Poinc}. In order to do that, we need the following result which will allow us to descend the sheaf $\Q$ of \S\ref{S:sheafQ}
to our desired  Poincar\'e sheaf $\ov\P$.


\begin{lemma}\label{L:descent}
Let $f:Y\rightarrow Z$ be a faithfully flat morphism of finite type  between locally Noetherian schemes.
Let $\F$ be a (maximal) Cohen-Macaulay coherent sheaf on $Y$ and $\G$ be a Cohen-Macaulay coherent sheaf on an open subset
$j:V\hookrightarrow Z$.
Assume that there exists an open subset $i:U\hookrightarrow Y$ such that:
\begin{enumerate}[(i)]
\item The complement of $U$ inside $Y$ has codimension at least two.
\item $f(U)\subseteq V$.
\item $(f_{|U})^*(\G_{|f(U)})=\F_{|U}\otimes f^*(N)_{|U}$ for some line bundle $N$ on $Z$.
\end{enumerate}
Then there exists a unique coherent sheaf $\wt\G$ on $Z$ such that
\begin{enumerate}[(a)]
\item $\wt\G$ is a (maximal) Cohen-Macaulay sheaf.
\item $\wt\G_{|V}=\G$.
\item $f^*(\wt\G)=\F\otimes f^*(N)$.
\end{enumerate}
\end{lemma}
\begin{proof}
First of all, observe that $f$ is both quasi-compact (hence a fpqc morphism) by \cite[Sec. 1.5]{EGAIV1} and of finite presentation (hence a fppf morphism) by \cite[Sec. 1.6]{EGAIV1}. Moreover, by replacing $\G$ with $\G\otimes N^{-1}$, we can assume that $N=\O_Z$.

Let us first prove the uniqueness of $\wt\G$. From hypothesis (i) and (ii) and the fact that $f$ is open with equidimensional fibers (being fppf, see \cite[Thm. 2.4.6]{EGAIV2} and
\cite[Cor. 14.2.2]{EGAIV3}),  we get that the complement of $V$ inside $Z$ has codimension at least two.
Since $\wt\G$ is Cohen-Macaulay by (a) and $\wt\G_{|V}=\G$ by (b), we get that $\wt\G=j_*\G$ (by \cite[Thm. 5.10.5]{EGAIV2}); hence $\wt\G$ is unique.

Let us now prove the existence of $\wt{G}$.
Using fpqc descent for quasi-coherent sheaves (see \cite[Thm. 4.23]{FGA}), the existence of a quasi-coherent sheaf $\wt\G$ satisfying (c) will follow if we find a descent data for $\F$ relative to the
fpqc morphism $f:Y\to Z$,  i.e. an isomorphism  $\phi:p_1^*(\F)\stackrel{\cong}{\longrightarrow} p_2^*(\F)$  satisfying the cocycle condition $p_{13}^*(\phi)=p_{12}^*(\phi)\circ p_{23}^*(\phi)$, where
$p_i:Y\times_Z Y\to Y$ and $p_{ij}:Y\times_Z Y\times_Z Y\to Y\times_Z Y$ denote the projection onto the $i$-th and $ij$-th factors, respectively.
Because of (iii), a descent data exists for $\F_{|U}$ relative to the fpqc morphism $f:U\to f(U)$, i.e. there exists  an isomorphism  $\psi:q_1^*(\F_{|U})\stackrel{\cong}{\longrightarrow} q_2^*(\F_{|U})$
such that $q_{13}^*(\psi)=q_{12}^*(\psi)\circ q_{23}^*(\psi)$, where $q_i:U\times_{f(U)} U\to U$ and $q_{ij}:U\times_{f(U)} U\times_{f(U)} U\to U\times_{f(U)} U$ denote the projection onto the $i$-th
and $ij$-th factors, respectively.  Observe now that, since  $\F$ is Cohen-Macaulay and the complement of $U$ inside $Y$ has codimension at least two by (i), it holds that
$\F= i_*(\F_{|U})$ by \cite[Thm. 5.10.5]{EGAIV2}. By taking the pushforward of $\psi$
with respect to the open embedding $i\times i:U\times_{f(U)} U\hookrightarrow Y\times_Z Y$, we obtain an isomorphism
$$p_1^*(\F)=p_1^*(i_*(\F_{|U}))=(i\times i)_*(q_1^*(\F_{|U})) \xrightarrow{\phi:=(i\times i)_*(\psi)} (i\times i)_*(q_2^*(\F_{|U}))=p_2^*(i_*(\F_{|U}))=p_2^*(\F),$$
which clearly satisfies the cocycle condition, since $\psi$ does. Therefore, by fpqc descent, we obtain a quasi-coherent sheaf $\wt\G$ on $Z$ which satisfies (c).
Observe now that $\wt\G$ is of finite type by faithful descent (see \cite[Prop. 2.5.2]{EGAIV2}), hence coherent because
$Y$ is locally Noetherian (see \cite[Sec. 6.1]{EGAI}). Moreover, $\wt\G$ is (maximal) Cohen-Macaulay by faithful descent (see \cite[Prop. 6.4.1]{EGAIV2}), hence (a) is satisfied.
Finally,  since the descent data for $\F$ that were used above to construct $\wt\G$ are induced by the descent data for $\F_{|U}=f^*(\G)_{|U}$, it follows that $\wt\G_{|f(U)}=\G_{|f(U)}$.
Therefore, the two Cohen-Macaulay sheaves $\wt\G_{|V}$ and $\G$ on $V$ have the same restriction to the open subset $f(U)\subseteq V$ whose complement has codimension
at least two by what observed above; hence $\wt\G_{|V}=\G$  by \cite[Thm. 5.10.5]{EGAIV2} and (b) is satisfied, q.e.d.

\end{proof}

\begin{proof}[Proof of Theorem \ref{T:Poinc}]
Consider an open cover $\bJbar_X^{1-g}=\bigcup_{\beta} U_{\beta}$ as in Fact \ref{F:Abel-sur}, in such a way that  for each $U_{\beta}$ there exists $M_{\beta}\in \Pic^{g}(X)$ with the property that $V_{\beta}:=(A_{M_{\beta}}^g)^{-1}(U_{\beta})\stackrel{A_{M_{\beta}}^g}{\longrightarrow} U_{\beta}$ is smooth and surjective.

We want to apply the descent Lemma \ref{L:descent} to the  smooth and surjective (hence faithfully flat of finite type) morphism $A_{M_{\beta}}^g\times \id: V_{\beta}\times \bJbar_X^{1-g}\to U_{\beta}\times \bJbar_X^{1-g}$
with respect to the sheaf $\Q$ on $V_{\beta}\times \bJbar_X^{1-g}$
(which is a maximal Cohen-Macaulay sheaf by Fact \ref{F:prop-Q}\eqref{F:prop-Q2}) and to the line bundle $\P$ defined on the open subset
$\left(U_{\beta}\times \bJbar_X^{1-g}\right)\cap \Open \subseteq U_{\beta}\times \bJbar_X^{1-g}$.
Let us check that the hypothesis of Lemma \ref{L:descent} are satisfied if we choose the open subset
\begin{equation*}\label{E:Wm'}
W'_{\beta}:= W_{\beta}\cap \Hilbc^{g}_X\times \bJbar_X^{1-g}= \left[\left(\Hilbr_X^{g}\cap V_{\beta}\right)\times \bJbar_X^{1-g}\right] \cup
\left[\left(\Hilbc_X^{g}\cap V_{\beta}\right)\times \bJ_X^{1-g}\right] \subseteq V_{\beta}\times \bJbar_X^{1-g},
\end{equation*}
where $W_{\beta}$ is defined in \eqref{E:openWm}.  We have already observed in \S\ref{S:Q'-P} that
$$(A_{M_{\beta}}^g\times \id)(W'_{\beta})\subseteq(A_{M_{\beta}}^g\times \id)(W_{\beta})\subseteq \Open $$
which gives the hypothesis (ii) of the Lemma. The hypothesis (iii) follows from Fact \ref{F:Q-Q'} and
Proposition \ref{P:comp-QP}. In order to prove the hypothesis (i), observe that the complement of $W'_{\beta}$ inside $V_{\beta}\times \bJbar_X^{1-g}$ is given by the closed subset
$$\left[\left(V_{\beta}\cap \left( \Hilb^{g}_X\setminus \Hilbc^{g}_X\right)\right)\times \bJbar_X^{1-g}\right] \cup
\left[\left(V_{\beta}\cap\left(\Hilb^{g}_X\setminus \Hilbr^{g}_X\right)\right)\times \left(\bJbar_X^{1-g}\setminus \bJ_X^{1-g}\right)\right].
$$
This closed subset has codimension at least two since: $\Hilb^{g}_X\setminus \Hilbc^{g}_X$ has codimension at least two by Lemma \ref{L:cod-curv2};
$\Hilb^{g}_X\setminus \Hilbr^{g}_X$ has codimension at least one by Fact \ref{F:Hilb-cur}\eqref{F:Hilb-cur2}; $\bJbar_X^{1-g}\setminus \bJ_X^{1-g}$ has codimension at least one
by Fact \ref{F:propJbig}\eqref{F:propJbig2}. The hypothesis (i) of Lemma \ref{L:descent} is therefore satisfied.

Therefore, we can now apply Lemma \ref{L:descent} in order to obtain a unique maximal Cohen-Macauly sheaf $\wt \P_{\beta}$ on $U_{\beta}\times \bJbar_X^{1-g}$ that agrees with $\P$ on the open subset $\left(U_{\beta}\times \bJbar_X^{1-g}\right)\cap \Open$ and whose pull-back via $A_{M_{\beta}}^g\times \id$ agrees with $\Q$, up to the pull-back of a line bundle on $\bJbar_X^{1-g}$.
 Because of the uniqueness of $\wt\P_{\beta}$ and the fact that $\P$ is defined on the whole $\Open$, the sheaves $\wt\P_{\beta}$ glue together to give a maximal Cohen-Macaulay sheaf $\wt\P$ on $\bJbar_X^{1-g}\times \bJbar_X^{1-g}$ that agrees with $\P$ on the open subset
$\Open \subseteq \bJbar_X^{1-g}\times \bJbar_X^{1-g}$. We have already observed in Remark \ref{R:CM-ext} that this is enough to ensure that $\wt \P=\ov\P:=j_{*}(\P)$. Part \eqref{T:Poinc1} of the theorem now follows.

Part \eqref{T:Poinc2}: it is clearly enough to prove the desired properties for the sheaves $\wt\P_{\beta}$ on
$U_{\beta}\times \bJbar_X^{1-g}$. By construction (see Proposition \ref{P:comp-QP} and Lemma \ref{L:descent}(c)),
we have that
$$(A_{M_{\beta}}^g\times \id)^*(\wt\P_{\beta})=\Q_{|V_{\beta}\times \bJbar_X^{1-g}}\otimes \pi_2^*N,$$
for some line bundle $N$ on $\bJbar_X^{1-g}$.
Now the flatness of $\wt\P_{\beta}$ with respect to the second projection follows from the analogous property of $\Q$ (see  Fact \ref{F:prop-Q}\eqref{F:prop-Q3}).
For any fixed $I\in \bJbar_X^{1-g}$, the fact that $(\wt\P_{\beta})_{|U_{\beta}\times \{I\}}$ is maximal Cohen-Macaulay follows from the fact that $(\Q\otimes \pi_2^*N)_{|V_{\beta}\times\{I\}}=\Q_{|V_{\beta}\times\{I\}}$ is maximal Cohen-Macaulay (see Fact \ref{F:prop-Q}\eqref{F:prop-Q4}) using faithful descent with respect to the smooth and surjective morphism
$A_{M_{\beta}}^g:V_{\beta}\to U_{\beta}$ (see \cite[Prop. 6.4.1]{EGAIV2}).

\end{proof}

\section{Properties of the Poincar\'e sheaf}\label{S:Poinc3}

Throughout this section, we assume that $X$ is a connected reduced curve with planar singularities of arithmetic genus $g:=p_a(X)$ and that either ${\rm char}(k)=0$ or  ${\rm char}(k)>g$.
The aim of this section is to prove several properties of the Poincar\'e sheaf $\ov\P$ on $\bJbar_X^{1-g}\times \bJbar_X^{1-g}$ constructed in Section \ref{S:Poinc}.

First of all, $\ov\P$ is symmetric with respect to the two factors.

\begin{prop}\label{P:symm}
The Poincar\'e sheaf is equivariant under the permutation $\sigma: \bJbar_X^{1-g}\times \bJbar_X^{1-g}\to \bJbar_X^{1-g}\times \bJbar_X^{1-g}$ of the two factors.

In particular, for any $I\in \bJbar_X^{1-g}$ we have that
$$\ov\P_I:=\ov\P_{|\bJbar_X^{1-g}\times \{I\}}=\ov\P_{|\{I\}\times \bJbar_X^{1-g}}.$$
\end{prop}
\begin{proof}
From the definition \eqref{E:Poin-sheaf} it is clear that
the Poincar\'e line bundle  $\P$ on $\Open$  is equivariant under the permutation of the two
factors. The same result for $\ov\P$ follows now from Remark \ref{R:CM-ext}.
\end{proof}

We now study the behavior of $\ov\P$ under the Serre dualizing functor. Recall that, for any Cohen-Macaulay scheme $Z$ with dualizing sheaf $\omega_Z$,  the Serre dualizing functor is defined as
$$\begin{aligned}
\bbD_{Z}: \Dc^b(Z)& \to \Dc^b(Z),\\
\cK^{\bullet} & \mapsto (\cK^{\bullet})^{D}:= \RHom(\cK^{\bullet},\omega_{Z}).
\end{aligned}$$
We will need the following well-known facts.

\begin{fact}\label{F:dual-func}
Let $Z$ be a Cohen-Macaulay scheme with dualizing functor $\bbD_Z$. Then:
\begin{enumerate}[(i)]
\item \label{F:dual-func1}
$\bbD_Z$ is an involution, or in other words $((\cK^{\bullet})^{D})^{D}=\cK^{\bullet}$ for any $\cK^{\bullet}\in \Dc^b(Z)$.
\item \label{F:dual-func2}
A coherent sheaf $\F$ on $Z$ is maximal Cohen-Macaulay if and only if $\F^{D}$ is concentrated in degree zero, i.e. if
$\F^{D}=\SHom(\F,\omega_Z)$.
\end{enumerate}
\end{fact}
\begin{proof}
For part \eqref{F:dual-func1} see \cite[Chap. V, Prop. 2.1]{Har2}.
For part \eqref{F:dual-func2} see \cite[Cor. 3.5.11]{BH}.
\end{proof}

Note that $\bJbar_X^{1-g}$ is a Gorenstein (and in particular Cohen-Macaulay)  scheme by Fact \ref{F:propJbig}\eqref{F:propJbig1}. 
Therefore, the same is true for $\bJbar_X^{1-g}\times \bJbar_X^{1-g}$.

\begin{prop}\label{P:dual}
The Serre dual complex $\ov\P^{D}$  and the dual sheaf $\ov P^{\vee}:=\SHom(\ov\P, \cO_{\bJbar_X^{1-g}\times \bJbar_X^{1-g}})$  of the Poincar\'e sheaf $\ov\P$ satisfy the following properties:
\begin{enumerate}[(i)]
\item \label{P:dual1} $\ov\P^{D}$ is concentrated in degree $0$ and it is equal to 
\begin{equation}\label{E:forduals}
\ov\P^{D}=\ov \P^{\vee}\otimes \omega_{\bJbar_X^{1-g}\times \bJbar_X^{1-g}}.
\end{equation}
\item \label{P:dual2} $\ov\P^{D}$ (and hence also $\ov \P^{\vee}$) is a maximal Cohen-Macaulay sheaf on $\bJbar_X^{1-g}\times \bJbar_X^{1-g}$.
\item \label{P:dual3}  $\ov \P^{\vee}$ and $\ov \P^D$ are equivariant with respect to the permutation $\sigma: \bJbar_X^{1-g}\times \bJbar_X^{1-g}\to \bJbar_X^{1-g}\times \bJbar_X^{1-g}$ of the two factors.
In particular,
$$ (\ov\P^{\vee})_{|\bJbar_X^{1-g}\times \{I\}}= (\ov\P^{\vee})_{|\{I\}\times \bJbar_X^{1-g}}:=(\ov \P^{\vee})_{I},$$
for every $I\in \bJbar_X^{1-g}$, and similarly for $(\ov \P^D)_I$. 
\item \label{P:dual4}  $\ov \P^D$ (and hence also $\ov \P^{\vee}$)  is flat with respect to the two projections $p_1,p_2:\bJbar_X^{1-g}\times \bJbar_X^{1-g}\to \bJbar_X^{1-g}$ and, for every $I\in \bJbar_X^{1-g}$, the restriction $(\ov \P^{D})_{I}=(\ov \P^{\vee})_I\otimes \omega_{\bJbar_X^{1-g}}$  is a maximal Cohen-Macaulay sheaf on $\bJbar_X^{1-g}$. Moreover,  it holds that
$$(\ov \P^{\vee})_{I}=(\ov \P_{I})^{\vee} \: \text{ and } \: (\ov \P^{D})_{I}=(\ov \P_{I})^{D}.$$
\end{enumerate}
\end{prop}
\begin{proof}
Part \eqref{P:dual1}: the fact that $\ov\P^{D}$ is concentrated in degree $0$ follows from the fact  that $\ov\P$ is a maximal Cohen-Macaulay sheaf on $\bJbar_X^{1-g}\times \bJbar_X^{1-g}$ (by Theorem \ref{T:Poinc}\eqref{T:Poinc1}) together with Fact \ref{F:dual-func}\eqref{F:dual-func2}. Formula \eqref{E:forduals} follows from the previous fact together with the  fact  that the dualizing sheaf of $\bJbar_X^{1-g}\times \bJbar_X^{1-g}$ is a line bundle (because $\bJbar_X^{1-g}\times \bJbar_X^{1-g}$ is a Gorenstein scheme).

Part \eqref{P:dual2} follows by combining Facts \ref{F:dual-func}\eqref{F:dual-func1} and \ref{F:dual-func}\eqref{F:dual-func2}.

Part \eqref{P:dual3} follows from the corresponding statement for $\ov\P$, see Proposition \ref{P:symm}.

Part \eqref{P:dual4}: combining Theorem \ref{T:Poinc}\eqref{T:Poinc2}  with \cite[Lemma 2.1(2)]{arin2} (which can be applied since $\bJbar_X^{1-g}$ is Gorenstein), we deduce that 
$\ov \P^{D}$ (and hence also $\ov \P^{\vee}$ by part \eqref{P:dual1}) is flat with respect to the two projections $p_1, p_2$  and that $(\ov \P^{D})_{I}=(\ov \P_{I})^{D}$ for every $I\in \bJbar_X^{1-g}$.  Moreover, since 
$\ov \P_{I}$ is maximal Cohen-Macaulay by Theorem \ref{T:Poinc}\eqref{T:Poinc2}, Fact \ref{F:dual-func} implies that also $(\ov \P^{D})_{I}$ is maximal Cohen-Macaulay. We conclude using that 
$(\ov \P^{D})_{I}=(\ov \P^{\vee})_I\otimes \omega_{\bJbar_X^{1-g}}$ by formula \eqref{E:forduals} and the analogous formula $(\ov \P_I)^{D}=(\ov \P_I)^{\vee}\otimes \omega_{\bJbar_X^{1-g}}$.

\end{proof}

Let us now study the behavior of the Poincar\'e sheaf $\ov\P$ under the natural multiplication map
\begin{equation}\label{E:multmap}
\begin{aligned}
\mu: \bJ_X^{1-g}\times \bJbar_X^{1-g}  & \longrightarrow \bJbar_X^{1-g}, \\
(L,I) & \mapsto L\otimes I.
\end{aligned}
\end{equation}

\begin{prop}\label{P:mult}
Consider the following diagram
$$\xymatrix{
\bJbar_X^{1-g}\times \bJbar_X^{1-g} & \bJ_X^{1-g}\times \bJbar_X^{1-g}\times \bJbar_X^{1-g} \ar[l]_{\pi_{23}} \ar[r]^{\pi_{13}}
\ar[d]^{\mu\times {\rm id}_{\bJbar_X^{1-g}}}& \bJ_X^{1-g}\times \bJbar_X^{1-g}, \\
& \bJbar_X^{1-g}\times \bJbar_X^{1-g} &
}
$$
The Poincar\'e sheaf satisfies the following property
$$(\mu\times {\id}_{\bJbar_X^{1-g}})^*(\ov\P)=\pi_{13}^*(\P)\otimes \pi_{23}^*(\ov\P).$$
In particular, for any $(L,I)\in \bJ_X^{1-g}\times \bJbar_X^{1-g}$, it holds  that $\ov\P_{L\otimes I}=\P_L\otimes \ov\P_I$.
\end{prop}
\begin{proof}
We are going to apply Lemma \ref{L:seesaw} with $T=\bJ_X^{1-g}$, $Z=\bJbar_X^{1-g}\times \bJbar_X^{1-g}$, $\F= \pi_{13}^*(\P)\otimes \pi_{23}^*(\ov\P)$ and $\G=(\mu\times {\id}_{\bJbar_X^{1-g}})^*(\ov\P)$.
Let us check that the hypothesis of the Lemma are satisfied.

First of all, $\bJ_X^{1-g}$ and $\bJbar_X^{1-g}$ are reduced locally Noetherian schemes by Facts \ref{F:propJbig}\eqref{F:propJbig1} and \ref{F:Este-Jac}\eqref{F:Este-Jac3}.
Moreover, $\bJbar_X^{1-g}\times \bJbar_X^{1-g}$ can be covered by the countably many connected and proper open subsets $\J_X(\un q)\times \J_X(\un q')$  by Facts \ref{F:Este-Jac} and Fact
\ref{F:prop-Jac}\eqref{F:prop-Jac1}.

For any $L\in \bJ_X^{1-g}$ and any $\J_X(\un q)\times \J_X(\un q')\subseteq \bJbar_X^{1-g}\times \bJbar_X^{1-g}$, the sheaf $\pi_{13}^*(\P)\otimes \pi_{23}^*(\ov\P)_{|\{L\}\times \J_X(\un q)\times \J_X(\un q')}$ is simple
because $\P$ is a line bundle and, by Definition \ref{D:Poinc}, $\ov\P_{|\J_X(\un q)\times \J_X(\un q')}$ is push-forward of the line bundle $\P$ from the open subset
$J_X(\un q)\times \J_X(\un q')\cup \J_X(\un q)\times J_X(\un q')$ of $\J_X(\un q)\times \J_X(\un q')$  whose complement has codimension greater or equal than two. Therefore, hypothesis
\eqref{hyp4} of Lemma \ref{L:seesaw} is satisfied.

In order to check that hypothesis \eqref{hyp6} of Lemma \ref{L:seesaw} is satisfied, we will check that for any $(I_1,I_2)\in \bJ_X^{1-g}\times \bJ_X^{1-g}\subseteq \bJbar_X^{1-g}\times \bJbar_X^{1-g}$ it holds that
\begin{equation}\label{E:restr-Z}
(\mu\times {\id}_{\bJbar_X^{1-g}})^*(\ov\P)_{|\bJ_X^{1-g}\times \{I_1\}\times \{I_2\}}=\pi_{13}^*(\P)\otimes \pi_{23}^*(\ov\P)_{|\bJ_X^{1-g}\times \{I_1\}\times \{I_2\}}.
\end{equation}
This will imply that hypothesis \eqref{hyp6} is satisfied since $\bJ_X^{1-g}\times \bJ_X^{1-g}$ is dense in $\bJbar_X^{1-g}\times \bJbar_X^{1-g}$ by Fact \ref{F:propJbig}\eqref{F:propJbig2}.
After identifying $\bJ_X^{1-g}\times \{I_1\}\times \{I_2\}$ with $\bJ_X^{1-g}$,  \eqref{E:restr-Z} is equivalent to
\begin{equation}\label{E:restr-Zbis}
t_{I_1}^*(\P_{I_2})=\P_{I_2},
\end{equation}
where $t_{I_1}:\bJ_X^{1-g}\to \bJ_X^{1-g}$ is the translation map sending $L$ into $L\otimes I_1$. Equality \eqref{E:restr-Zbis} follows now from  \cite[Lemma 5.4]{MRV2}.

Finally, in order to check that hypothesis \eqref{hyp3} of Lemma \ref{L:seesaw} is satisfied, we need to prove that for any $L\in \bJ_X^{1-g}$ Êwe have that
\begin{equation}\label{E:restr-T}
(\mu\times {\id}_{\bJbar_X^{1-g}})^*(\ov\P)_{|\{L\}\times \bJbar_X^{1-g}\times \bJbar_X^{1-g}}=\pi_{13}^*(\P)\otimes \pi_{23}^*(\ov\P)_{|\{L\}\times \bJbar_X^{1-g}\times \bJbar_X^{1-g}}.
\end{equation}
Identifying $\{L\}\times \bJbar_X^{1-g}\times \bJbar_X^{1-g}$ with $\bJbar_X^{1-g}\times \bJbar_X^{1-g}$, \eqref{E:restr-T}Ê is equivalent to
\begin{equation}\label{E:restr-Tbis}
(t_L\times \id)^*(\ov\P)=\pi_2^*(\P_L)\otimes \ov\P,
\end{equation}
where $t_L:\bJbar_X^{1-g}\to \bJbar_X^{1-g}$ is the translation map sending $I$ to $I\otimes L$ and $\pi_2: \bJbar_X^{1-g}\times \bJbar_X^{1-g}\to \bJbar_X^{1-g}$ is the projection onto the second factor.
Since the sheaves appearing on the left and right hand side of \eqref{E:restr-Tbis} are Cohen-Macaulay sheaves by Theorem \ref{T:Poinc}\eqref{T:Poinc1}, it is enough, by Remark  \ref{R:CM-ext},
to show that we have the equality of sheaves on $\Open$:
\begin{equation}\label{E:restr-Tter}
(t_L\times \id)^*(\P)=\pi_2^*(\P_L)\otimes \P.
\end{equation}
From the definition of $\P$ in \S\ref{S:Poinc1} (keeping the same notation) and using the base change property of the determinant of cohomology (see \cite[Prop. 44(1)]{est1}), we get
\begin{equation}\label{E:equaz1}
(t_L\times \id)^*\P=\D_{p_{23}}(p_{12}^*\I\otimes p_1^*L\otimes p_{13}^*\calI)^{-1}\otimes \D_{p_{23}}(p_{12}^*\I\otimes p_1^*L)\otimes \D_{p_{23}}(p_{13}^*\I),
\end{equation}
\begin{equation}\label{E:equaz2}
\begin{aligned}
\pi_2^*(\P_L)\otimes \ov\P&= \D_{p_{23}}(p_1^*L\otimes p_{13}^*\calI)^{-1}\otimes \D_{p_{23}}(p_1^*L)\otimes \D_{p_{23}}(p_{13}^*\I)\otimes \\
 &\otimes \D_{p_{23}}(p_{12}^*\I\otimes  p_{13}^*\calI)^{-1}\otimes \D_{p_{23}}(p_{12}^*\I)\otimes \D_{p_{23}}(p_{13}^*\I).
\end{aligned}
\end{equation}
Since $L$ is a line bundle of degree zero on $X$, we can find two reduced Cartier divisors $E_1=\sum_{j=1}^n q_j^1$ and $E_2=\sum_{j=1}^n q_j^2$ of the same degree on $X$,
supported on the smooth locus of $X$, such that
$L=I_{E_1}\otimes I_{E_2}^{-1}=\O_X(-E_1+E_2)$. Using \eqref{E:equaz1} and \eqref{E:equaz2},Ê together with the easy fact that $\D_{p_{23}}(p_1^*L)=\O$, and
 applying three times Lemma  \ref{L:compu-det} to the sheaves $p_{12}^*\I$, $p_{13}^*\I$ and $p_{12}^*\I\otimes p_{13}^*\I$,  we get  Ê
\begin{equation}\label{E:equaz3}
\begin{aligned}
 (t_L\times \id)^*(\P) & \otimes \pi_2^*(\P_L)^{-1}\otimes \P^{-1}= \D_{p_{23}}(\left(p_{12}^*\I\otimes  p_{13}^*\calI\right)_{|p_1^{-1}(E_2)})^{-1}\otimes
\D_{p_{23}}(\left(p_{12}^*\I\otimes  p_{13}^*\calI\right)_{|p_1^{-1}(E_1)})\otimes\\
& \otimes \D_{p_{23}}(p_{12}^*\I_{|p_1^{-1}(E_2)}) \otimes \D_{p_{23}}(p_{12}^*\I_{|p_1^{-1}(E_1)})^{-1} \otimes
\D_{p_{23}}(p_{13}^*\I_{|p_1^{-1}(E_2)}) \otimes \D_{p_{23}}(p_{13}^*\I_{|p_1^{-1}(E_1)})^{-1}
\end{aligned}
\end{equation}
Observe now that for any coherent sheaf $\F$ on $X\times \Open$ whose restriction to $p_1^{-1}(E_i)=p_1^{-1}(\sum_j q_j^i)$ (for $i=1,2$) is a line bundle,  from the definition of the determinant of cohomology
it follows that
\begin{equation*}\label{E:equaz4}
\D_{p_{23}}(\F_{|p_1^{-1}(E_i)})=\bigotimes_{j} \F_{|\{q_j^i\}\times \Open}.
\end{equation*}
This implies that for any $i=1,2$ we have
\begin{equation}\label{E:equaz5}
\D_{p_{23}}(\left(p_{12}^*\I\otimes  p_{13}^*\calI\right)_{|p_1^{-1}(E_i)})= \D_{p_{23}}(p_{12}^*\I_{|p_1^{-1}(E_i)}) \otimes\D_{p_{23}}(p_{13}^*\I_{|p_1^{-1}(E_i)}).
\end{equation}
Substituting \eqref{E:equaz5} into \eqref{E:equaz3}, we get that
\begin{equation*}
(t_L\times \id)^*(\P)  \otimes \pi_2^*(\P_L)^{-1}\otimes \P^{-1}=\O_{\Open},
\end{equation*}
which shows that \eqref{E:restr-Tter} holds true.

Therefore, all the hypothesis of Lemma \ref{L:seesaw} are in our case satisfied and the thesis of that Lemma concludes our proof.

\end{proof}

The following Lemma, which is a generalization of the classical seesaw principle (see \cite[Sec. II.5, Cor. 6]{Mum}), was used in the proof of Proposition \ref{P:mult}.

\begin{lemma}[Seesaw principle]\label{L:seesaw}
Let $Z$ and $T$ be two reduced locally Noetherian schemes such that $Z$ admits an open cover $\displaystyle Z=\bigcup_{m\in \bbN} U_{m}$, with $U_{m}$ proper and connected.
Let $\F$ and $\G$ be two coherent sheaves on $Z\times T$, flat over $T$, such that
\begin{enumerate}[(i)]
\item \label{hyp3} $\F_{|Z \times \{t\}}\cong \G_{|Z\times \{t\}}$ for every $t\in T$;
\item  \label{hyp4} $\F_{|U_{m}\times \{t\}}$ is simple for every $t\in T$ and $m\in \bbN$;
\item \label{hyp6} for every connected component $W$ of $Z$, there exists $z_0\in W$ and isomorphism $\psi: \F_{|\{z_0\}\times T}\stackrel{\cong}{\longrightarrow} \G_{|\{z_0\}\times T}$  of line bundles.
\end{enumerate}
Then $\F\cong \G$.
\end{lemma}
\begin{proof}
Clearly, it is enough to prove the result for every connected component of $Z$; hence, we can assume that $Z$ is connected. Moreover, up to reordering the open subsets $U_{m}$,
we can assume that $z_0\in U_0$ and that  for every $m\in \bbN$ the open subset
$\displaystyle V_m:=\bigcup_{0\leq n\leq m} U_m$ is connected.
For every $m\in \bbN$, we set $\F_m:=\F_{|V_m\times T}$ and $\G_m:=\G_{|V_m\times T}$. The hypothesis \eqref{hyp3} and \eqref{hyp4} on $\F$ and $\G$, together with the fact that
$V_{m-1}\cap U_m\neq \emptyset$ (since $V_m$ is connected),  imply the following (for every $m\in \bbN$):
\begin{enumerate}[(a)]
\item \label{hyp3bis} $(\F_m)_{|V_m \times \{t\}}\cong (\G_m)_{|V_m\times \{t\}}$ for every $t\in T$;
\item  \label{hyp4bis} $(\F_m)_{|V_{m}\times \{t\}}$ is simple for every $t\in T$;
\end{enumerate}

\un{CLAIM:} For every $m\in \bbN$, there exists a unique isomorphism $\phi_m:\F_m \stackrel{\cong}{\longrightarrow} \G_m$ such that its restriction to $\{z_0\}\times T$
coincides with $\psi$.

Indeed, because of \eqref{hyp3bis}Ê and \eqref{hyp4bis}, the sheaf $(p_2)_*({\mathcal Hom}(\F_m,\G_m))$ is a line bundle on $T$,
where $p_2:V_m\times T \to T$ denotes the projection onto the second factor. Moreover, since $z_0\in U_0\subseteq V_m$ and using \eqref{hyp6}, we get that  $(p_2)_*({\mathcal Hom}(\F_m,\G_m))=\O_T$.
The isomorphism $\psi$ of \eqref{hyp6} defines a non-zero constant section of $(p_2)_*({\mathcal Hom}(\F_0,\G_0))=\O_T$, which gives rise to an isomorphism
$\phi_m: \F_m \stackrel{\cong}{\longrightarrow} \G_m$ via the natural evaluation morphism $\F_m\otimes p_2^*\left((p_2)_*({\mathcal Hom}(\F_m,\G_m))\right)\to \G_m$.
By construction, $\phi_m$ is the unique isomorphism whose restriction to  $\{z_0\}\times T$ is the isomorphism $\psi$ of \eqref{hyp6}, q.e.d.

\vspace{0.1cm}

The Claim implies that for any $m\geq 1$, we have that $(\phi_{m})_{|V_{m-1}\times T}=\phi_{m-1}$;  hence, the isomorphisms $\phi_m$ glue together producing an isomorphism
$\phi:\F \stackrel{\cong}{\longrightarrow} \G$.

\end{proof}

From the above Lemma, we get the following Corollary which will be used later on.

\begin{cor}\label{C:seesaw}
Let $Z$ and $T$ be two reduced locally Noetherian schemes and assume that $Z$ admits an open cover $ \displaystyle Z=\bigcup_{\alpha\in \A} U_{\alpha}$, with $U_{\alpha}$ proper and connected. Let $\L$ and $\M$ be two line bundles on $Z\times T$ such that:
\begin{enumerate}[(a)]
\item $\L_{|Z\times \{t\}}=\M_{|Z\times \{t\}}$ for every $t\in T$,
\item $\L_{|\{z\}\times T}=\M_{|\{z\}\times T}$ for every $z\in Z$.
\end{enumerate}
Then $\L=\M$.
\end{cor}

Combining Propositions \ref{P:symm} and \ref{P:mult}, we immediately get the following Corollary, describing the behavior of $\ov\P$ under translations.

\begin{cor}\label{C:transla}
Give two line bundles $M_1,M_2\in \Pic^0(X)$, consider the translation morphism
\begin{equation}\label{E:transla}
\begin{aligned}
t_{(M_1,M_2)}: \bJbar_X^{1-g}\times \bJbar_X^{1-g} & \longrightarrow \bJbar_X^{1-g}\times \bJbar_X^{1-g}, \\
(I_1, I_2) & \mapsto (I_1\otimes M_1, I_2\otimes M_2).
\end{aligned}
\end{equation}
The we have
\begin{equation}\label{E:Poin-tra}
t_{(M_1,M_2)}^*\ov\P=\ov\P\otimes p_1^*(\ov\P_{M_2})\otimes p_2^*(\ov\P_{M_1}).
\end{equation}
\end{cor}

\vspace{0.1cm}

Let us finally  examine  the behavior of the Poincar\'e sheaf $\ov\P$ under the duality involution
\begin{equation}\label{E:duality}
\begin{aligned}
\nu: \bJbar_X^{1-g} & \longrightarrow \bJbar_X^{1-g}, \\
I & \mapsto I^{\vee}:=\SHom(I,\O_X).
\end{aligned}
\end{equation}

\begin{prop}\label{P:duality}
The Poincar\'e sheaf $\ov\P$ satisfies the following properties:
\begin{enumerate}[(i)]
\item \label{P:duality1}  $(\nu\times \id_{\bJbar_X^{1-g}})^{*}\ov\P=(\id_{\bJbar_X^{1-g}}\times \nu)^{*}\ov\P=\ov\P^{\vee}$. In particular, $\ov\P_{I^{\vee}}=(\ov\P^{\vee})_I$ for any $I\in \bJbar_X^{1-g}$.
\item \label{P:duality2} $(\nu\times \nu)^{*}\ov\P=\ov\P$.
\end{enumerate}
\end{prop}
\begin{proof}
First of all, notice that (ii) follows from (i) observing that
$$(\nu\times\nu)=(\nu\times \id_{\bJbar_X^{1-g}})\circ(\id_{\bJbar_X^{1-g}}\times \nu)$$
and using the fact that, since $\nu$ is an involution, if (i) holds for $\ov\P$ then it also holds for $\ov \P^{\vee}$.

Let us prove (i) for the morphism $(\nu\times\id_{\bJbar_X^{1-g}})$; the case of $(\id_{\bJbar_X^{1-g}}\times \nu)$
is dealt with in the same way. Using Remark \ref{R:CM-ext}, it is enough to prove the result for the Poincar\'e line bundle $\P$ on  $\Open$.
Consider the following diagram
$$
\xymatrix{
X\times  \Open \ar[rr]_{\id_X\times\nu\times \id_{\bJbar_X^{1-g}}}  \ar[d]_{p_{23}} & & X \times \Open \ar[d]_{p_{23}}\\
\Open  \ar[rr]_{\nu\times \id_{\bJbar_X^{1-g}}} & & \Open.
}
$$
By definition \eqref{E:Poin-sheaf}, we get
\begin{equation}\label{E:P-dual}
\P^\vee=\D(p_{12}^{*}\I\otimes p^*_{13}\I)\otimes \D(p_{12}^*\I)^{-1}\otimes \D(p_{13}^*\I)^{-1},
\end{equation}
where $\D$ denotes the determinant of cohomology with respect to the projection map $p_{23}$.

By the base change property of the determinant of cohomology (see e.g. \cite[Prop. 44(2)]{est1}) and using the equalities
\begin{equation}
\left\{\begin{array}{l}
 (\id_X\times\nu\times \id_{\bJbar_X^{1-g}})^*(p^*_{13} \I)= p_{13}^* \I,\\
 (\id_X\times\nu\times \id_{\bJbar_X^{1-g}})^*(p^*_{12} \I)= p_{12}^* \I^\vee,
\end{array} \right.
\end{equation}
we get that
\begin{equation}\label{E:pull-dual}
(\nu\times\id_{\bJ_X})^*\P= \D(p_{12}^*\I^\vee\otimes p^*_{13}\I)^{-1}\otimes\D(p_{12}^*\I^\vee)\otimes \D(p_{13}^*\I).
\end{equation}

\un{CLAIM 1:} For any  $I\in \bJ_X^{1-g}$, we have that
$$\P^\vee_{|\{I\}\times \bJbar_X^{1-g}}=\left[(\nu\times\id_{\bJ_X})^*\P\right]_{|\{I\}\times \bJbar_X^{1-g}}.$$

Indeed, using \eqref{E:P-dual} and \eqref{E:pull-dual} together with the base change property of the determinant of cohomology, we get
\begin{equation}\label{E:res1}
\begin{sis}
\P^\vee_{|\{I\}\times \bJbar_X^{1-g}}=& \D(p_1^*I\otimes \I)\otimes \D(p_1^*I)^{-1}\otimes \D(\I)^{-1}= \D(p_1^*I\otimes \I)\otimes \D(\I)^{-1},\\
\left[(\nu\times\id_{\bJ_X})^*\P\right]_{|\{I\}\times \bJbar_X^{1-g}}= & \D(p_1^*I^{\vee}\otimes \I)^{-1}\otimes \D(p_1^*I^{\vee})\otimes \D(\I)= \D(p_1^*I^{\vee}\otimes \I)^{-1}\otimes \D(\I),\\
\end{sis}
\end{equation}
where $\D$ denotes now the determinant of cohomology with respect to the projection map $p_2:X\times \bJbar_X^{1-g}\to \bJbar_X^{1-g}$ and $p_1$ is the projection onto the first factor.
Write now $I=I_{E_1}\otimes I_{E_2}^{-1}$ as in Lemma \ref{L:sheaves}.  Since $I$ is a line bundle by assumption, we have that $E_1$ is Cartier; moreover, arguing similarly to the proof of Lemma
\ref{L:sheaves}, we can choose $E_1$ to be supported on the smooth locus of $X$. Hence, the representation $I^{\vee}=I^{-1}=I_{E_2}\otimes I_{E_1}^{-1}$ satisfies the same assumptions as in
Lemma \ref{L:sheaves}. We can now apply Lemma \ref{L:compu-det} twice in order to conclude that
\begin{equation}\label{E:doublelem}
\begin{sis}
& \D(p_1^*I\otimes \I)\otimes \D(\I)^{-1}= \D\left(\I_{|p_1^{-1}(E_2)}\right)\otimes \D\left(\I_{|p_1^{-1}(E_1)}\right)^{-1},\\
& \D(p_1^*I^{\vee}\otimes \I)\otimes \D(\I)^{-1} =\D\left(\I_{|p_1^{-1}(E_1)}\right)\otimes \D\left(\I_{|p_1^{-1}(E_2)}\right)^{-1}.\\
\end{sis}
\end{equation}
Claim 1 follows from \eqref{E:res1} and \eqref{E:doublelem}.

\un{CLAIM 2:} For any  $I\in \bJ_X^{1-g}$, we have that
$$\P^\vee_{|\bJbar_X^{1-g}\times \{I\}}=\left[(\nu\times\id_{\bJ_X})^*\P\right]_{|\bJbar_X^{1-g}\times \{I\}}.$$

Again, using \eqref{E:P-dual} and \eqref{E:pull-dual} together with the base change property of the determinant of cohomology, we get
\begin{equation}\label{E:res2}
\begin{sis}
\P^\vee_{|\bJbar_X^{1-g}\times \{I\}}=& \D(\I\otimes p_1^*I)\otimes \D(\I)^{-1} \otimes \D(p_1^*I)^{-1}=  \D(\I\otimes p_1^*I)\otimes \D(\I)^{-1},\\
\left[(\nu\times\id_{\bJ_X})^*\P\right]_{|\bJbar_X^{1-g}\times \{I\}}= & \D(\I^{\vee}\otimes p_1^*I)^{-1}\otimes \D(\I^{\vee})\otimes \D(p_1^*I)=
\D(\I^{\vee}\otimes p_1^*I)^{-1}\otimes \D(\I^{\vee}),\\
\end{sis}
\end{equation}
where $\D$ denotes now the determinant of cohomology with respect to the projection map $p_2:X\times \bJbar_X^{1-g}\to \bJbar_X^{1-g}$ and $p_1$ is the projection onto the first factor.
Write now $I=I_{E_1}\otimes I_{E_2}^{-1}$ as in Lemma \ref{L:sheaves}.  As before, since $I$ is a line bundle by assumption, we can assume  that $E_1$ is Cartier and supported on the smooth
locus of $X$.  We can apply Lemma \ref{L:compu-det} to the sheaf $\I$ and to its dual $\I^{\vee}$ in order to get  that
\begin{equation}\label{E:rest-div}
\begin{sis}
 \D(\I\otimes p_1^*I)\otimes \D(\I)^{-1} &= \D\left(\I_{|p_1^{-1}(E_2)}\right)\otimes \D\left(\I_{|p_1^{-1}(E_1)}\right)^{-1}, \\
 \D(\I^{\vee}\otimes p_1^*I)\otimes \D(\I^{\vee})^{-1} & =  \D\left(\I^{\vee}_{|p_1^{-1}(E_2)}\right)\otimes \D\left(\I^{\vee}_{|p_1^{-1}(E_1)}\right)^{-1}.
\end{sis}
 \end{equation}
Fix now $i\in \{1,2\}$ and denote by $\pi$  the projection $X\times \bJbar_X^{1-g}\supset p_1^{-1}(E_i)\to \bJbar_X^{1-g}$.
 Since $\I$ is locally free along the divisor $p_1^{-1}(E_i)$  (because $E_i$ is supported on the smooth locus of $X$), we have that $R^1\pi_*(I_{|p_1^{-1}(E_i)})=0$ and that
 $\pi_*(I_{|p_1^{-1}(E_i)})$ is locally free; the same statements hold with $\I$ replaced by $\I^{\vee}$. Therefore, from the definition of the determinant of cohomology (see the discussion in
 \S\ref{S:Poinc1}), it follows that
\begin{equation}\label{E:det-push}
\begin{sis}
&\D\left(\I_{|p_1^{-1}(E_i)}\right)=\det \pi_*(\I_{|p_1^{-1}(E_i)}) , \\
& \D\left(\I^{\vee}_{|p_1^{-1}(E_i)}\right)=\det \pi_*(\I^{\vee}_{|p_1^{-1}(E_i)}) .\\
\end{sis}
\end{equation}
 Applying the relative duality to the finite morphism $\pi$ (see \cite[Sec. III.6]{Har2}) and using that
the relative dualizing sheaf $\omega_{\pi}=p_1^*(\omega_{E_i})$ of $\pi$ is trivial because $E_i$ is a curvilinear $0$-dimensional scheme (hence Gorenstein), we have that
\begin{equation}\label{E:rel-dual}
\pi_*(\I_{|p_1^{-1}(E_i)})^{\vee}=\pi_*{\mathcal RHom}(\I_{|p_1^{-1}(E_i)},\omega_{\pi})=\pi_*(\I^{\vee}_{|p_1^{-1}(E_i)}).
\end{equation}
Claim 2 now follows by combining \eqref{E:res2}, \eqref{E:rest-div}, \eqref{E:det-push} and \eqref{E:rel-dual}.

\vspace{0,1cm}

Now, by applying Corollary \ref{C:seesaw} and using Claim 1 and Claim 2, we get that $\P^\vee$ is isomorphic to $(\nu\times\id_{\bJ_X})^*\P$ on $\bJbar_X^{1-g}\times \bJ_X^{1-g}$ and on $\bJ_X\times \bJbar_X^{1-g}$.
Finally, using the invariance of $\P$ under the permutation $\sigma$ of the two factors, it can be checked that the above isomorphisms  glue to an isomorphism between $\P^\vee$ and $(\nu\times\id_{\bJ_X})^*\P$
on the entire $\Open$, q.e.d.

\end{proof}

\section{Proof of the main results}\label{S:proof}

The aim of this section is to prove Theorem A and Theorem B from the introduction.

Fix two fine compactified Jacobians $\J_X(\un q), \J_X(\un q')\subseteq \bJbar_X^{1-g}$ associated to two general polarizations $q$ and $q'$ on a connected reduced $k$-curve $X$ with planar singularities and arithmetic genus $g:=p_a(X)$, as in \S\ref{S:comp-Jac}. Assume throughout this section that either ${\rm char}(k)=0$ or
${\rm char}(k)>g$.
With a slight abuse of notation, we will also denote by $\ov\P$ the restriction of the Poincar\'e sheaf (see \S\ref{S:Poinc2}) to $\J_X(\un q)\times \J_X(\un q')$.

The constructions of Section \ref{S:Poinc} can be repeated for the universal family
 $\bJbar_{\X}^{1-g}\times_{\Spec R_X} \bJbar_{\X}^{1-g}\to \Spec R_X$ and they provide a maximal Cohen-Macaulay sheaf $\Pbun$
over $\bJbar_{\X}^{1-g}\times_{\Spec R_X} \bJbar_{\X}^{1-g}$, which we call the \emph{universal Poincar\'e sheaf},
whose restriction to the  fiber over the closed point $o$ of $\Spec R_X$ coincides with the Poincar\'e sheaf $\ov\P$ on $\bJbar_X^{1-g}\times \bJbar_X^{1-g}$.
With a slight abuse of notation, we will also denote by $\Pbun$ the restriction of the universal Poincar\'e sheaf
to $\J_{\X}(\un q)\times_{\Spec R_X} \J_{\X}(\un q')$, where $u:\J_{\X}(\un q)\to \Spec R_X$ (resp. $u':\J_{\X}(\un q')\to \Spec R_X$) is the universal fine compactified Jacobian with respect to the polarization $\un q$ (resp. $\un q'$) as in \S\ref{S:defo}.

Consider now the complex
\begin{equation}\label{E:psi-univ}
\Psi^{\rm un}:=Rp_{13*}\left(p_{12}^*((\Pbun)^{\vee})\otimes p_{23}^*(\Pbun)\right)\in \Dc^b(\J_{\X}(\un q)\times_{\Spec R_X} \J_{\X}(\un q)),
\end{equation}
where $p_{ij}$ denotes the projection of $\J_{\X}(\un q)\times_{\Spec R_X}\J_{\X}(\un q')\times_{\Spec R_X} \J_{\X}(\un q)$ onto the  $i$-th and $j$-th factors.

Notice that, as $\Pbun$ is flat over $\J_{\X}(\un q')$, its pull back $p_{23}^*(\Pbun)$ is flat over 
$\J_{\X}(\un q)\times_{\Spec R_X}\J_{\X}(\un q')$ with respect to the (flat) morphism $p_{12}$. Therefore the tensor product
$p_{12}^*((\Pbun)^{\vee})\otimes p_{23}^*(\Pbun)$ is quasi isomorphic to the derived tensor product
$p_{12}^*((\Pbun)^{\vee})\lotimes p_{23}^*(\Pbun)$.  
\begin{remark}\label{tensorCM}
The derived dual $$R\mathcal{H}om\left(p_{12}^*((\Pbun)^{\vee})\otimes p_{23}^*(\Pbun), \mathcal{O}_{\J_{\X}(\un q)\times_{\Spec R_X}\J_{\X}(\un q')\times_{\Spec R_X} \J_{\X}(\un q)}\right)$$ is isomorphic in $\Dc\left(\J_{\X}(\un q)\times_{\Spec R_X}\J_{\X}(\un q')\times_{\Spec R_X} \J_{\X}(\un q)\right)$ to the sheaf
$p_{12}^*(\Pbun)\otimes p_{23}^*((\Pbun)^{\vee})$. In particular $p_{12}^*((\Pbun)^{\vee})\otimes p_{23}^*(\Pbun)$
is a maximal Cohen-Macaulay  sheaf.

To prove this isomorphism \footnote{The proof of the isomorphism  would be elementary  if $\J_{\X}(\un q)\times_{\Spec R_X}\J_{\X}(\un q')\times_{\Spec R_X} \J_{\X}(\un q)$ were smooth.}, let $L^{\bullet}\rightarrow (\Pbun)^{\vee}$ be a locally free resolution.
As $p_{12}$ is flat, it induces a locally free resolution $p_{12}^{*}(L^{\bullet})\rightarrow p_{12}^{*}((\Pbun)^{\vee})$ .
Since $p_{23}^*(\Pbun)$ is flat over 
$\J_{\X}(\un q)\times_{\Spec R_X}\J_{\X}(\un q')$ with respect to $p_{12}$, tensoring with
$p_{23}^*(\Pbun)$ still gives a resolution $$p_{12}^{*}(L^{\bullet})\otimes p_{23}^*(\Pbun)
\rightarrow p_{12}^{*}((\Pbun)^{\vee})\otimes p_{23}^*(\Pbun).$$
As $p_{23}$ is flat and $\Pbun$ is maximal Cohen-Macaulay, the pull-back $p_{23}^*(\Pbun)$ 
is maximal Cohen-Macaulay  too. It follows that $p_{12}^{*}(L^{\bullet})\otimes p_{23}^*(\Pbun)$ is a complex of 
Cohen-Macaulay  sheaves. Therefore $$R\mathcal{H}om\left(p_{12}^*((\Pbun)^{\vee})\otimes p_{23}^*(\Pbun), \mathcal{O}_{\J_{\X}(\un q)\times_{\Spec R_X}\J_{\X}(\un q')\times_{\Spec R_X} \J_{\X}(\un q)}\right)$$ is isomorphic, in $\Dc\left(\J_{\X}(\un q)\times_{\Spec R_X}\J_{\X}(\un q')\times_{\Spec R_X} \J_{\X}(\un q)\right)$, to 
$$\left(p_{12}^{*}(L^{\bullet})\otimes p_{23}^*(\Pbun)\right)^{\vee}=p_{12}^{*}(L^{\bullet})^{\vee}\otimes p_{23}^*(\Pbun)^{\vee}.$$
Finally, as  $(\Pbun)^{\vee}$ is maximal Cohen-Macaulay , the complex $(\Pbun)\rightarrow (L^{\bullet})^{\vee}$ is exact and, as
$p_{12}$ is flat and $p_{23}^*(\Pbun)^{\vee}$ is flat with respect to $p_{12}$, the complex  
$$p_{12}^{*}(\Pbun)\otimes p_{23}^*(\Pbun)^{\vee}\rightarrow p_{12}^{*}(L^{\bullet})^{\vee}\otimes p_{23}^*(\Pbun)^{\vee}$$
is also exact. Hence $p_{12}^{*}(\Pbun)\otimes p_{23}^*(\Pbun)^{\vee}$ is isomorphic to $p_{12}^{*}(L^{\bullet})^{\vee}\otimes p_{23}^*(\Pbun)^{\vee}$ in the derived category
$\Dc\left(\J_{\X}(\un q)\times_{\Spec R_X}\J_{\X}(\un q')\times_{\Spec R_X} \J_{\X}(\un q)\right)$.
\end{remark}

Theorem A will descend easily from the following key result.

\begin{thm}\label{T:psi-univ}
Notation as above.
Then there is a natural isomorphism in $\Dc^b(\J_{\X}(\un q)\times_{\Spec R_X} \J_{\X}(\un q))$
\begin{equation}\label{E:key-univ}
\vartheta^{\rm un}: \Psi^{\rm un}[g]\longrightarrow \O_{\Delta^{\rm un}},
\end{equation}
where $\O_{\Delta^{\rm un}}$ is the structure sheaf of the universal diagonal $\Delta^{\rm un}\subset \J_{\X}(\un q)\times_{\Spec R_X}\J_{\X}(\un q)$.
\end{thm}

Before proving Theorem \ref{T:psi-univ}, we need to bound the dimension of the support of the following complex
\begin{equation}\label{E:psi}
\Psi:=Rp_{13*}\left(p_{12}^*((\ov \P)^{\vee})\otimes p_{23}^*(\ov\P)\right)\in \Dc^b(\J_{X}(\un q)\times \J_{X}(\un q)).
\end{equation}


\begin{prop}\label{P:codim-psi}
Same assumptions as in Theorem \ref{T:psi-univ}. Then the complex $\Psi$ of \eqref{E:psi} satisfies
$$\codim\left( \supp(\Psi)\right)\geq g^{\nu}(X),$$
with strict inequality if $X$ is irreducible and singular.
\end{prop}
\begin{proof}
For any $I\in \bJbar_{X}^{1-g}$, set $\ov\P_{I}:=\ov\P_{|\{I\}\times \bJbar_{X}^{1-g}}$ and $\P_I:=\P_{|\{I\}\times \bJ_X^{1-g}}$.
Observe that, as the tensor product $p_{12}^*((\Pbun)^{\vee})\otimes p_{23}^*(\Pbun)$ is isomorphic to the derived tensor product
$p_{12}^*((\Pbun)^{\vee})\lotimes p_{23}^*(\Pbun)$, for any $(I_1,I_2)\in \J_X(\un q)\times \J_X(\un q)$ the derived restriction of $p_{12}^*((\Pbun)^{\vee})\otimes p_{23}^*(\Pbun)$ to 
$p_{13}^{-1}((I_1,I_2))$ is  isomorphic to $\ov\P_{I_1}^{\vee}\lotimes \ov\P_{I_2}$ in $\Dc^b(\J_{X}(\un q)\times \J_{X}(\un q))$. Therefore, using base change and Proposition \ref{P:dual}\eqref{P:dual4}, we have
\begin{equation}\label{E:supp-set}
\supp(\Psi):=\{(I_1,I_2)\in \J_X(\un q)\times \J_X(\un q) \: : \: \HH^i(\J_X(\un q'), \ov\P_{I_1}^{\vee}\lotimes \ov\P_{I_2})\neq 0 \text{ for some } i \in \bbZ \}.
\end{equation}

\un{Claim 1:} If $(I_1,I_2)\in \supp(\Psi)$ then $\P_{I_1}=\P_{I_2}$.

The proof of this Claim follows the same lines of the proof of \cite[Prop. 7.2]{arin2} using Proposition \ref{P:mult} (see also \cite[Prop. 6.1]{MRV2} and \cite[Prop. 1]{arin1}) and therefore is left to the reader.

\un{Claim 2:} If $(I_1,I_2)\in \supp(\Psi)$ then $(I_1)_{|X_{\rm sm}}=(I_2)_{|X_{\rm sm}}$.

Consider the L-twisted Abel map $A_L:X\to \bJbar_X$ for some $L\in \Pic(X)$, see \cite[\S 6.1]{MRV1}. Since $A_L(X_{\rm sm})\subseteq \bJ_X$ and $A_L^*(\P_I)=I_{|X_{\rm sm}}$ for any
$I\in \bJbar_X$ (see \cite[Prop. 5.6]{MRV2}), Claim 2 follows from Claim 1.

\vspace{0,1cm}

Consider now the map
\begin{equation*}
\begin{aligned}
(\mu\times \id):J(X)\times \J_X(\un q)\times \J_X(\un q)& \longrightarrow \J_X(\un q)\times \J_X(\un q),\\
(L,I_1,I_2)& \mapsto (L\otimes I_1,I_2),\\
\end{aligned}
\end{equation*}
and set $\Phi:=(\mu\times \id)^{-1}(\Psi)$. Since $\mu\times \id$ is a smooth and surjective map, hence with equidimensional fibers, it is enough to prove that
\begin{equation}\label{E:Phi-codim}
\codim(\supp(\Phi))\geq g^{\nu}(X) \: \: \text{ with strict inequality if }  X \text { is irreducible and singular. }
\end{equation}
Consider now the projection $p:\Phi\to \J_X(\un q)\times \J_X(\un q)$ on the last two factors. Using Claim 2, the fiber of $p$ over the point $(I_1,I_2)\in \J_X(\un q)\times \J_X(\un q)$ is contained in
the locus of $J(X)$ consisting of those line bundles $L\in J(X)$ such that $L_{|X_{\rm sm}}=(I_1)_{|X_{\rm sm}}^{-1}\otimes (I_2)_{|X_{\rm sm}}$. Arguing as in \cite[Cor. 6.3]{MRV2}, it follows that
\begin{equation}\label{E:fib-Phi1}
\codim(p^{-1}(I_1,I_2))\geq g^{\nu}(X) \: \: \text{ for any }Ê\: (I_1,I_2)\in \J_X(\un q)\times \J_X(\un q).
\end{equation}
Moreover, if $X$ is irreducible and $(I_1,I_2)\in J_X(\un q)\times J_X(\un q)$, then we can apply \cite[Thm. B(i)]{arin1}, together with Propositions \ref{P:duality} and \ref{P:mult}, in order to deduce that
\begin{equation*}
p^{-1}(I_1,I_2)=\{L\in J(X)\: :\:  H^i(\J_X(\un q'), \P_{(L\otimes I_1)^{-1} \otimes I_2})\neq 0 \:  \text{ for some }\in \bbZ\}=\{I_1^{-1}\otimes I_2\}.
\end{equation*}
Therefore, if $X$ is irreducible and singular, it holds that
\begin{equation}\label{E:fib-Phi2}
\codim(p^{-1}(I_1,I_2))=p_a(X)>g^{\nu}(X)  \: \: \text{ for any }Ê\: (I_1,I_2)\in J_X(\un q)\times J_X(\un q).
\end{equation}
Combining \eqref{E:fib-Phi1} and \eqref{E:fib-Phi2}, we get   condition \eqref{E:Phi-codim}, q.e.d.

\end{proof}


We can now give the proof of Theorem \ref{T:psi-univ}.

\begin{proof}[Proof of Theorem \ref{T:psi-univ}]

We will first explain how the natural morphism $\vartheta^{\rm un}$ is defined.
Consider the Cartesian diagram
\begin{equation}\label{E:diag-eta}
\xymatrix{
\wt{\Delta^{\rm un}}\ar@{^{(}->}[r]^(.16){\wt{i}}\ar[d]^{\wt{p}}\ar@{}[dr]|{\square} &
\ov J_{\X}(\un q)\times_{\Spec R_X} \ov J_{\X}(\un q')
\times_{\Spec R_X} \ov J_{\X}(\un q) \ar[d]^{p_{13}} \\
\Delta^{\rm un} \ar@{^{(}->}[r]^(.3)i & \ov J_{\X}(\un q)\times_{\Spec R_X} \ov J_{\X}(\un q)
}
\end{equation}
and notice that the morphism $\wt{p}$ can be identified with the projection onto the first factor
$$p_{1}:\ov J_{\X}(\un q)\times_{\Spec R_X} \ov J_{\X}(\un q')\rightarrow \ov J_{\X}(\un q).$$
By applying base change (see e.g. \cite{BBH08},
Prop. A.85) to the diagram \eqref{E:diag-eta},  we get a morphism
\begin{equation}\label{E:basechan}
 Li^* \Psi^{\rm un}=Li^* Rp_{13*}\left(p_{12}^*((\Pbun)^{\vee})\lotimes p_{23}^*(\Pbun)\right) \longrightarrow R\wt{p}_*L\wt{i}^*\left(p_{12}^*((\Pbun)^{\vee})\lotimes p_{23}^*(\Pbun)\right).
\end{equation}
By a standard spectral sequence argument, $L\wt{i}^*\left(p_{12}^*((\Pbun)^{\vee})\lotimes p_{23}^*(\Pbun)\right)$ is isomorphic to
$(\Pbun)^{\vee}\lotimes \Pbun$ in  $\Dc(\wt{\Delta^{\rm un}})$, hence it admits a natural morphism to its top degree cohomology (i.e. the usual tensor product $(\Pbun)^{\vee}\otimes \Pbun$) and to the structure sheaf $\O_{\wt{\Delta^{\rm un}}}$.
Composing with $R\wt{p}_*$ finally gives a morphism

\begin{equation}\label{E:mor-tens}
 R\wt{p}_*L\wt{i}^*\left(p_{12}^*((\Pbun)^{\vee})\lotimes p_{23}^*(\Pbun)\right)\simeq R\wt{p}_*\left((\Pbun)^{\vee}\lotimes \Pbun\right)\longrightarrow R\wt{p}_* \left((\Pbun)^{\vee}\otimes \Pbun\right) \longrightarrow
  R\wt{p}_*(\O_{\wt{\Delta^{\rm un}}}).
\end{equation}
Since the complex of sheaves $ R\wt{p}_*(\O_{\wt{\Delta^{\rm un}}})$ is concentrated in cohomological degrees
from $0$ to $g$, we get a morphism of complexes of sheaves
\begin{equation}\label{E:mor-to-g}
 R\wt{p}_*(\O_{\wt{\Delta^{\rm un}}})\longrightarrow  R^g\wt{p}_*(\O_{\wt{\Delta^{\rm un}}})[-g].
\end{equation}
Moreover, since the morphism $\wt{p}$ is proper of relative dimension $g$, with trivial relative dualizing sheaf and geometrically connected fibers (by Facts  \ref{F:Punivfine} and  \ref{F:prop-Jac}),
then the relative duality  applied to $\wt{p}$ gives that (see \cite[Cor. 11.2(g)]{Har}) :
\begin{equation}\label{E:iso-deg-g}
R^g \wt{p}_*(\O_{\wt{\Delta^{\rm un}}})\cong  \O_{\Delta^{\rm un}}.
\end{equation}
By composing the morphisms \eqref{E:basechan}, \eqref{E:mor-tens},  \eqref{E:mor-to-g} and using the isomorphism \eqref{E:iso-deg-g},
we get a morphism
\begin{equation}\label{E:basech2}
Li^* \Psi^{\rm un}\longrightarrow \O_{\Delta^{\rm un}}[-g].
\end{equation}
Since $i_*=(Ri_{*})$ is right adjoint to $Li^*$, the morphism \eqref{E:basech2}, shifted by $[g]$, gives rise to the morphism $\vartheta^{\rm un}$.

\vspace{0.1cm}

In order to show that $\vartheta^{\rm un}$ is an isomorphism, we divide the proof
 into several steps, which we collect under the name of Claims.

The first claim says that $\vartheta^{\rm un}$ is 
an isomorphism on an interesting open subset. 
More precisely, let $(\Spec R_X)_{\rm sm}$ be the open subset of $\Spec R_X$ consisting 
of all the (schematic) points $s\in \Spec R_X$ such that the geometric fiber $\X_{\ov s}$ of the universal 
family $\pi:\X\to \Spec R_X$  is smooth.
\vspace{0,1cm}

\un{CLAIM 1:} The morphism $\vartheta^{\rm un}$ is an isomorphism over the open subset $(u\times u)^{-1}(\Spec R_X)_{\rm sm}\subseteq \ov J_{\X}(\un q)\times_{\Spec R_X} \ov J_{\X}(\un q)$.

It is enough to prove that $\vartheta^{\rm un}$ is an isomorphism when restricted to $(u\times u)^{-1}(s)$ for any $s\in (\Spec R_X)_{\rm sm}$. 
This is a classical result due to Mukai (see \cite[Thm. 2.2]{mukai}), it may also be seen as 
a particular case of \cite[Prop. 7.1]{arin2}, which holds more generally for any irreducible curve.

\vspace{0,2cm}

\un{CLAIM 2:}  $\Psi^{\rm un}[g]$ is a Cohen-Macaulay sheaf of codimension $g$.

Let us first prove that
\begin{equation}\label{E:codim-psi}
\codim( \supp (\Psi^{\rm un}))=g.
\end{equation}
First of all, Claim 1 gives that $\codim(\supp  (\Psi^{\rm un}))\leq g$. In order to prove the reverse inequality, we stratify
the scheme $\Spec R_X$  into locally closed subsets according to the geometric genus of the geometric fibers of the universal family $\X\to \Spec R_X$:
$$(\Spec R_X)^{g^{\nu}=l}:=\{s\in \Spec R_X\: :\: g^{\nu}(\X_{\ov s})=l\}, $$
for any $g^{\nu}(X)\leq l\leq p_a(X)=g$. Fact \ref{F:Diaz-Har}\eqref{F:Diaz-Harris1} gives that $\codim (\Spec R_X)^{g^{\nu}=l}\geq g-l$.
On the other hand, on the fibers of $u\times u$ over
$(\Spec R_X)^{g^{\nu}=l}$, the sheaf $\Psi^{\rm un}$ has support of codimension at least $l$ by Proposition
\ref{P:codim-psi}. Therefore, we get
\begin{equation}\label{E:supp-strata}
\codim (\supp (\Psi^{\rm un}) \cap (u\times u)^{-1}((\Spec R_X)^{g^{\nu}=l}))\geq g \text{ for any }Šg^{\nu}(X)\leq l\leq g.
\end{equation}
Since the locally closed subsets $(\Spec R_X)^{g^{\nu}=l}$ form a stratification of $\Spec R_X$, we deduce that
$g\leq \codim( \supp (\Psi^{\rm un}))$, which concludes
the proof of \eqref{E:codim-psi}.

Observe next that, since $p_{23}$ has relative dimension $g$ and $p_{12}^*((\ov \P)^{\vee})\lotimes p_{23}^*(\ov\P)$ is concentrated in non-positive degrees,
we have that
\begin{equation}\label{E:coho-psi}
H^i(\Psi^{\rm un}[g])=0 \text{ for }Êi>0.
\end{equation}
Let now $\bbD$ be the dualizing functor of $\J_{\X}(\un q)\times_{\Spec R_X}\J_{\X}(\un q)$. Applying the relative duality (see \cite[Chap. VII.3]{Har2}) to the morphism $p_{13}$, which is
projective and flat of relative dimension $g$ by Fact  \ref{F:Punivfine}\eqref{F:Punivfine2}  and with trivial dualizing sheaf by Fact \ref{F:prop-Jac}\eqref{F:prop-Jac1},  we get
that
\begin{equation}\label{E:dual-psi}
\bbD(\Psi^{\rm un}[g])=Rp_{13*}\left(
R\mathcal{H}om\left(p_{12}^*((\Pbun)^{\vee})\otimes p_{23}^*(\Pbun), \mathcal{O}_{\J_{\X}(\un q)\times_{\Spec R_X}\J_{\X}(\un q')\times_{\Spec R_X} \J_{\X}(\un q)}\right)
\right).
\end{equation}
By Remark \ref{tensorCM}, the second term of \eqref{E:dual-psi} equals 
\begin{equation}\label{E:sdoppiamento}
Rp_{13*}\left(\left[p_{12}^*((\Pbun)^{\vee})\otimes p_{23}^*(\Pbun)\right]^{\vee}\right)=
Rp_{13*}\left(p_{12}^*(\Pbun)\otimes p_{23}^*(\Pbun)^{\vee}\right)=
\sigma^*(\Psi^{\rm un})
\end{equation}
where $\sigma$ is the automorphism of $\J_{\X}(\un q)\times_{\Spec R_X} \J_{\X}(\un q)$ that exchanges the two factors.
From \eqref{E:coho-psi}, \eqref{E:dual-psi} and \eqref{E:sdoppiamento} it follows that
\begin{equation}\label{E:coho-dual}
H^i(\bbD(\Psi^{\rm un}[g]))=0 \text{ for }Êi>g.
\end{equation}
Since $\Psi^{\rm un}[g]$ satisfies  \eqref{E:codim-psi}, \eqref{E:coho-psi}, \eqref{E:coho-dual}, Lemma 7.6 of  \cite{arin2} gives that $\Psi^{\rm un}[g]$ is a Cohen-Macaulay sheaf of codimension $g$.\footnote{An expanded proof of Cohen-Macaulay of
$\Psi^{\rm un}[g]$ can be given by copying the proofs of Claim 3 and Claim 4 in the proof of Theorem 8.1 in \cite{MRV1}.}

\vspace{0,2cm}

\un{CLAIM 3:} We have a set-theoretic equality $\supp(\Psi^{\rm un}[g])=\Delta^{\rm un}$.

Let $Z$ be an irreducible component of $\supp(\Psi^{\rm un}[g])$. Since $\Psi^{\rm un}[g]$ is a Cohen-Macaulay sheaf of codimension $g$, then by \cite[Thm. 6.5(iii), Thm. 17.3(i)]{Mat} we get that
\begin{equation}\label{E:comp1}
\codim Z=g.
\end{equation}
Let $\eta$ be the generic point of $(u\times u)(Z)$. Clearly, $(u\times u)(Z)$ is contained in $(\Spec R_X)^{p_a^{\nu}\leq p_a^{\nu}(\X_{\ov \eta})}$, from which we deduce,
using Fact \ref{F:Diaz-Har}\eqref{F:Diaz-Harris1}, that
\begin{equation}\label{E:comp2}
\codim (u\times u)(Z)\geq \codim (\Spec R_X)^{p_a^{\nu}\leq p_a^{\nu}(\X_{\ov \eta})}= g-p_a^{\nu}(\X_{\ov \eta}).
\end{equation}
Moreover, denoting by $\Psi_{\ov \eta}$ the analogue of $\Psi$ for the curve $\X_{\ov \eta}$, Proposition \ref{P:codim-psi} gives that
\begin{equation}\label{E:comp3}
\codim \supp(\Psi_{\ov \eta})\geq g^{\nu}(\X_{\ov \eta}).
\end{equation}
Putting together \eqref{E:comp1}, \eqref{E:comp2} and \eqref{E:comp3}, we compute that
\begin{equation*}
g=\codim Z=\codim (u\times u)(Z)+ \codim \supp(\Psi_{\ov \eta})\geq g-p_a^{\nu}(\X_{\ov \eta})+g^{\nu}(\X_{\ov \eta})\geq g,
\end{equation*}
which implies that $p_a^{\nu}(\X_{\ov \eta})=g^{\nu}(\X_{\ov \eta})$ (i.e., $\X_{\ov \eta}$ is irreducible) and that equality holds in \eqref{E:comp3}. By Proposition  \ref{P:codim-psi}, this can happen only if $\X_{\ov \eta}$ is
smooth. Then Claim 1 implies that  $Z=\Delta^{\rm un}$, q.e.d.





\vspace{0,2cm}

\un{CLAIM 4:} We have a scheme-theoretic equality $\supp(\Psi^{\rm un}[g])=\Delta^{\rm un}$.

Since the subscheme $\Delta^{\rm un}$ is reduced, Claim 3 gives the inclusion of subschemes $\Delta^{\rm un}\subseteq  \supp(\Psi^{\rm un}[g])$.
Moreover, Claim 1 says that this inclusion is generically an equality; in particular $\supp(\Psi^{\rm un}[g])$ is generically reduced.
Moreover, since $\Psi^{\rm un}[g]$ is a Cohen-Macaulay sheaf by Claim 2, we get that $\supp(\Psi^{\rm un}[g])$ is reduced by \cite[Lemma 8.2]{MRV2}.
Therefore, we must have the equality of subschemes $\Psi^{\rm un}[g]=\Delta^{\rm un}$.

\vspace{0,2cm}

We can now finish the proof of Theorem \ref{T:psi-univ}. Combining Claims 1,  2 and 4, we get that the sheaf $\Psi^{\rm un}[g]$ is a Cohen-Macaulay sheaf supported (schematically) on $\Delta^{\rm un}$, hence  
$$\vartheta^{\rm un}:\Psi^{\rm un}[g]\rightarrow \O_{\Delta^{\rm un}}$$
is a morphism of sheaves supported (schematically) on $\Delta^{\rm un}$, therefore it is an isomorphism if and only if $i^{*}\vartheta^{\rm un}$ is an isomorphism.

The morphism $i^{*}\vartheta^{\rm un}$ is, by definition, the morphism induced on degree-$g$ cohomology groups (i.e. the top cohomology groups) by the composition 
$$
\xymatrix{
Li^{*}
\Psi^{\rm un}\ar[r]^(.4){a_{1}\;\;\;\;\;\;} 
& R\wt{p}_*((\Pbun)^{\vee}\lotimes \Pbun) 
\ar[r]^{a_2} & R\wt{p}_*((\Pbun)^{\vee}\otimes \Pbun)
\ar[r]^(.6){a_3} &  R\wt{p}_*(\O_{\wt{\Delta^{\rm un}}}),
}
$$
where $a_{1}$ is the base change morphism of \eqref{E:basechan}
and $a_{2}$ and $a_{3}$  are the morphisms appearing in \eqref{E:mor-tens}.
Using the isomorphism
\eqref{E:iso-deg-g}, denoting by $H^{g}(a_{i})$ the morphisms induced by $a_{i}$ on the $g$-th cohomology sheaves, it remains to show that $H^{g}(a_{i})$ are isomorphisms.

The morphism $H^{g}(a_{1})$ is an isomorphism because $g$ is the relative dimension of $p_{13}$, hence $a_{1}$ is a top degree base change map. The morphism $H^{g}(a_{2})$ is an isomorphism by a spectral sequence argument, because the tensor product $(\Pbun)^{\vee}\otimes \Pbun$ is the degree-$0$
and top cohomology sheaf of $(\Pbun)^{\vee}\lotimes \Pbun$ and $g$ is the relative dimension of the flat morphism $\wt{p}$.

Finally, $H^{g}(a_{3})$ is an isomorphism because the kernel and the cokernel of 
$(\Pbun)^{\vee}\otimes \Pbun \rightarrow\O_{\wt{\Delta^{\rm un}}}$ are supported on  the locus where 
$\Pbun$ is not locally free , i.e. 
in the fiber product over $\Spec R_X$ of the singular loci of 
$\ov J_{\X}(\un q)$ and  $\ov J_{\X}(\un q')$. As this locus intersects each fibers of $\wt{p}$ in dimension at most $g-2$,
the morphism $(\Pbun)^{\vee}\otimes \Pbun \rightarrow\O_{\wt{\Delta^{\rm un}}}$ induces the desired isomorphism  
$R^g \wt{p}_*((\Pbun)^{\vee}\otimes \Pbun)\simeq R^g \wt{p}_*(\O_{\wt{\Delta^{\rm un}}}).$


\end{proof}

By passing to the central fiber, Theorem \ref{T:psi-univ} implies  the following result which is a generalization of the result of S. Mukai for Jacobians of smooth curves (see \cite[Thm. 2.2]{mukai})
and of D. Arinkin for compactified Jacobians of irreducible curves (see \cite[Prop. 7.1]{arin2}).

\begin{cor}\label{C:psi}
Let $X$ be a connected and reduced curve with planar singularities and arithmetic genus $g:=p_a(X)$ and let $\un q$ and $\un q'$ be two general polarizations on $X$. of total degree $1-g$. 
Assume that either ${\rm char}(k)=0$ or  ${\rm char}(k)>g$. 
Then there is a  natural isomorphism in $\Dc^b(\J_{X}(\un q)\times \J_{X}(\un q))$
\begin{equation}\label{E:key}
\vartheta: \Psi[g] \longrightarrow \O_{\Delta},
\end{equation}
where $\O_{\Delta}$ is the structure sheaf of the diagonal $\Delta \subset \J_{X}(\un q)\times \J_{X}(\un q)$.
\end{cor}
\begin{proof}
The natural morphism $\vartheta$ is defined similarly to  the morphism $\vartheta^{\rm un}$ in \eqref{E:key-univ}.
By base changing the isomorphism $\vartheta^{\rm un}$ to the central fiber, we obtain the following diagram in  $\Dc^b(\J_{X}(\un q)\times \J_{X}(\un q))$:
\begin{equation}\label{E:diag-bsch}
\xymatrix{
\Psi^{\rm un}[g]_{|\J_{X}(\un q)\times \J_{X}(\un q)} \ar[r]^{\vartheta^{\rm un}}_{\cong} \ar[d]_b & (\O_{\Delta^{\rm un}})_{|\J_{X}(\un q)\times \J_{X}(\un q)} \ar[d]_{\cong}\\
 \Psi[g] \ar[r]^{\vartheta} & \O_{\Delta}
}
\end{equation}
where $b$ is the base change morphism (see e.g. \cite[Rmk. 3.33]{Huy}). Note that the complex $\Psi[g]$ is supported in non-positive degree (as it follows easily from its definition) and, since the base change morphism is an
isomorphism in top degree (see e.g. \cite[Thm. 12.11]{Har}), we deduce that the morphism $b$ induces an isomorphism on the $0$-th cohomology sheaves:
\begin{equation}\label{E:top-iso}
\calH^0(b): \Psi^{\rm un}[g]_{|\J_{X}(\un q)\times \J_{X}(\un q)}=\calH^0(\Psi^{\rm un}[g]_{|\J_{X}(\un q)\times \J_{X}(\un q)})\stackrel{\cong}{\longrightarrow} \calH^0(\Psi[g]).
\end{equation}
Moreover, since the sheaf $\Psi^{\rm un}[g]_{|\J_{X}(\un q)\times \J_{X}(\un q)}$ is supported on $\Delta$, by the base change theorem (see e.g. \cite[Thm. 12.11]{Har}) we deduce that
$\Psi[g]$ has set-theoretic support on $\Delta$, which has codimension $g$ inside
$\ov J_X(\un q)\times \ov J_X(\un q)$. Therefore, Proposition 2.26 of \cite{Muk2} gives that $\Psi$ is supported in degree $g$, or in other words that
\begin{equation}\label{E:Psi-0deg}
\Psi[g]\stackrel{\cong}{\longrightarrow} \calH^0(\Psi[g]).
\end{equation}
From \eqref{E:top-iso} and \eqref{E:Psi-0deg}, it follows that the base change morphism $b$ is an isomorphism; using the diagram \eqref{E:diag-bsch}, we deduce that $\vartheta$ is an isomorphism.

\end{proof}

With the help of the above Corollary, we can now prove Theorem A from the introduction.

\begin{proof}[Proof of Theorem A]
Consider the integral functor
\begin{eqnarray*}
\Phi^{\ov{\P}^{\vee}[g]}:\Dqc^b(\ov{J}_X(\un q'))& \longrightarrow & \Dqc^b(\ov{J}_X(\un q))\\
\cplx{E} &\longmapsto & \bR p_{1*}(p_2^*(\cplx{E})\lotimes \ov{\P}^{\vee}[g])
\end{eqnarray*}
where $\ov\P^{\vee}$ is the dual sheaf of $\ov\P$ as in  Proposition \ref{P:dual}.
The composition $\Phi^{\ov{\P}^{\vee}[g]}\circ \Phi^{\ov\P}$ is the  integral functor with kernel given by $\Psi[g]$
\begin{eqnarray*}
\Phi^{\ov{\P}^{\vee}[g]}\circ \Phi^{\ov\P}=\Phi^{\Psi[g]}:\Dqc^b(\ov{J}_X(\un q))& \longrightarrow & \Dqc^b(\ov{J}_X(\un q))\\
\cplx{E} &\longmapsto & \bR p_{2*}(p_1^*(\cplx{E})\lotimes\Psi[g]),
\end{eqnarray*}
where $\Psi$ is the complex of \eqref{E:psi}, see e.g. \cite[Sec. 1.4]{HLS}.
 Since $\Psi[g]=\O_{\Delta}$ by Corollary \ref{C:psi}, we have that $\Phi^{\ov{\P}^{\vee}[g]}\circ \Phi^{\ov\P}=\id$. By exchanging the roles of $\J_X(\un q)$ and $\J_X(\un q')$, we get similarly that
 $\Phi^{\ov\P}\circ \Phi^{\ov{\P}^{\vee}[g]}=\id$, which proves the first statement of Theorem A.

 The second statement follows from the first one together with the fact that  $\Phi^{\ov\P}$  sends $\Dc^b(\ov{J}_X(\un q'))$ into $\Dc^b(\J_X(\un q))$ (and similarly for
 $\phi^{\ov\P^{\vee}[g]}$) because $\ov\P$ is a coherent sheaf and $\J_X(\un q)$ and $\J_X(\un q')$ are proper varieties.

\end{proof}

Finally, we can prove Theorem B from the introduction.

\begin{proof}[Proof of Theorem B]
First of all, let us check that the morphism $\rho_{\un q}$ is well-defined.

For any $I\in \bJbar_X^{1-g}$,  $\ov\P_I:=\ov\P_{|\ov{J}_X(\un q) \times \{I\}}$ is a Cohen-Macaulay sheaf on $\J_X(\un q)$  (by Theorem \ref{T:Poinc}\eqref{T:Poinc2})
whose restriction to the dense open subset $J_X(\un q)\subseteq \J_X(\un q)$ is a line bundle (see \S \ref{S:Poinc1}), which implies that  $\ov\P_I$ has rank-$1$ on each irreducible component of
$\J_X(\un q)$. Moreover, by the definition of the integral transform $\Phi^{\ov\P}$, it follows that $\Phi^{\ov\P}({\bf k}(I))=\ov\P_I$ which, using the fully faithfulness of $\Phi^{\ov\P}$ (see Theorem A), gives that
$$\Hom(\ov\P_I, \ov\P_I)=\Hom({\bf k}(I),{\bf k}(I))=k,$$
or, in other words, that $\ov\P_I$ is simple. Therefore, we get  a morphism
\begin{equation}\label{E:rho-tilde}
\begin{aligned}
\wt{\rho}_{\un q}: \bJbar_X^{1-g} & \longrightarrow  \Pmm(\J_X(\un q)) \\
I & \mapsto \ov\P_I
\end{aligned}
\end{equation}
whose image is contained in $\Pm(\J_X(\un q))\subseteq \Pmm(\J_X(\un q))$. Since $\wt{\rho}_{\un q}(\O_X)=\beta_{\un q}(\O_X)=\O_{\J_X(\un q)}\in \Pic^o(\J_X(\un q))\subset \Pmm(\J_X(\un q))$,
the morphism $\wt{\rho}_{\un q}$ induces the required morphism $\rho_{\un q}$ by passing to the connected components containing the structure sheaves.
This concludes the proof that $\rho_{\un q}$ is well-defined and  it also shows that \eqref{E:Bi} holds true.

The morphism  $\wt{\rho}_{\un q}$ (and hence also $\rho_{\un q}$) is equivariant with respect to the isomorphism $\beta_{\un q}$ since  for any $I\in \bJbar_X^{1-g}$ and $L\in \Pic^o(X)$
it holds that $\ov\P_I\otimes \P_L=\ov\P_{I\otimes L}$ by Proposition \ref{P:mult}.

Moreover, $\wt{\rho}_{\un q}$ induces a homomorphism of group schemes
$$\begin{aligned}
(\wt{\rho}_{\un q})_{|\bJ_X^{1-g}}: \bJ_X^{1-g}=\Pic^0(X) & \longrightarrow  \Pic(\J_X(\un q)) \\
I & \mapsto \ov\P_I=\P_I
\end{aligned}
$$
since
for any $I, I'\in \bJ_X^{1-g}$ we have that $\P_{I\otimes I'}=\P_I\otimes \P_{I'}$ by Proposition \ref{P:mult} and $\P_I^{-1}=\P_{I^{-1}}$ by Proposition \ref{P:duality}.
This shows that \eqref{E:Bii} holds true.

Consider now the restriction of $\wt{\rho}_{\un q}$ to  a fine compactified Jacobian $\J_X(\un q')\subseteq \bJbar_X^{1-g}$:
\begin{equation}\label{E:rho-qq'}
\begin{aligned}
\wt{\rho}_{\un q'/\un q}: \J_X(\un q') & \longrightarrow  \Pmm(\J_X(\un q)) \\
I & \mapsto \ov\P_I.
\end{aligned}
\end{equation}

\un{CLAIM 1:} $\wt{\rho}_{\un q'/\un q}$ is an open embedding.

The proof of this claim is similar to \cite[Proof of Theorem B]{arin2}; let us sketch the argument for the benefit of the reader.
Fix a polarization $\O(1)$ on $\J_{X}(\un q')$ and, for any coherent sheaf $S$ on $\J_X(\un q')$, denote by $\theta(S)\in \bbQ[t]$ the Hilbert polynomial of $S$ with respect to $\O(1)$.
Consider the locus $\L$ inside $\Pmm(\J_X(\un q'))$ consisting of the sheaves $F\in \Pmm(\J_X(\un q'))$ such that (with the notation of the proof of Theorem A):
\begin{equation}\label{E:Hilb-pol}
\theta(H^i(\Phi^{\ov{P}^{\vee}[g]}(F)))=
\begin{cases}
1 & \text{ if } i=0,\\
0 & \text{ if }Êi\neq 0.
\end{cases}
\end{equation}
Using the upper semicontinuity of the Hilbert polynomial, it follows that $\L$ is an open subset of $\Pmm(\J_X(\un q'))$ (see \cite[Proof of Theorem B]{arin2}).
From \eqref{E:Hilb-pol}, we deduce  that $F\in \L$ if and only if $\Phi^{\ov{P}^{\vee}[g]}(F)={\bf k}(I)$ for some $I\in  \J_X(\un q')$, which using Theorem A, is equivalent to the fact that
$F=\Phi^{\ov\P}({\bf k}(I))$. In other words, the image of $\wt{\rho}_{\un q'/\un q}$ is equal to the open subset $\L\subseteq \Pmm(\J_X(\un q))$.
Moreover, $\wt{\rho}_{\un q'/\un q}$ is an isomorphism into its image $\L=\Im \wt{\rho}_{\un q'/\un q}$ whose inverse is given by the morphism
$$
\begin{aligned}
\sigma: \L& \longrightarrow \J_X(\un q'), \\
F & \mapsto \supp \Phi^{\ov\P^{\vee}[g]}(F). \\
\end{aligned}
$$

Before proving part \eqref{E:Biv}, let us examine how the map $\rho_{\un q}$ behaves with respect to the decomposition of $X$ into its separating blocks.
Consider the partial normalization $\wt{X}\to X$ at the separating nodes of $X$ and denote by $Y_1,\ldots, Y_r$ the images in $X$ of the connected components of
$\wt{X}$  (in \cite[\S 6.2]{MRV1}, the curves $\{Y_1,\ldots,Y_r\}$ are called the separating blocks of $X$). Note that $Y_1,\ldots,Y_r$ are connected (reduced and projective) curves with planar singularities.
From \cite[Prop. 6.6(i)]{MRV1}, it follows that  we have an isomorphism
$$
\begin{aligned}
\bJbar_X & \stackrel{\cong}{\longrightarrow} \bJbar_{Y_1}\times \ldots \times \bJbar_{Y_r} \\
I & \mapsto (I_{|Y_1}, \ldots, I_{|Y_r}),
\end{aligned}
$$
which implies that
\begin{equation}\label{E:dec-PicX}
\Picbar(X)=\Picbar(Y_1)\times \ldots \times \Picbar(Y_r).
\end{equation}
Using \cite[Prop. 6.6]{MRV1}, we get the existence of general polarizations $\un q^i$ on $Y_i$ such that
$$\J_X(\un q)=\J_{Y_1}(\un q^1)\times \ldots\times \J_{Y_r}(\un q^r),$$
which implies that
\begin{equation}\label{E:dec-PicJ}
\Picbar(\J_X(\un q))=\Picbar(\J_{Y_1}(\un q^1))\times \ldots \times \Picbar(\J_{Y_r}(\un q^r)).
\end{equation}
Moreover, arguing as in \cite[Lemma 5.5]{MRV2}, the fibers of the Poincar\'e sheaf $\ov \P$ over a sheaf $I=(I_1, \ldots, I_r)\in \Picbar(X)=\Picbar(Y_1)\times \ldots \times \Picbar(Y_r)$ are such that
\begin{equation}\label{E:dec-Poinc}
\ov\P_I=\ov \P^1_{I_1}\boxtimes \ldots \boxtimes \ov \P^r_{I_r}:=p_1^*(\ov \P^1_{I_1})\times \ldots \times p_r^*(\ov \P^r_{I_r}),
\end{equation}
where $p_i:\Picbar( \J_X(\un q))\to \Picbar(\J_{Y_i}(\un q^i))$ denotes the projection onto the $i$-th factor in the decomposition \eqref{E:dec-PicJ} and $\ov\P^i$ is the Poincar\'e sheaf on
$\bJbar_{Y_i}^{1-p_a(Y_i)}\times \bJbar_{Y_i}^{1-p_a(Y_i)}$. Putting together the decompositions \eqref{E:dec-PicX}, \eqref{E:dec-PicJ} and \eqref{E:dec-Poinc}, we get that the morphism $\rho_{\un q}$
decomposes as
\begin{equation}\label{E:deco-rho}
\begin{aligned}
\rho_{\un q} =\prod_{i=1}^r \rho_{\un q^i}: \Picbar(X)=\prod_{i=1}^r\Picbar(Y_i) & \longrightarrow
\Picbar(\J_X(\un q))=\prod_{i=1}^r\Picbar(\J_{Y_i}(\un q^i)) \\
I=(I_1, \ldots, I_r) & \mapsto \ov\P_I=\ov \P^1_{I_1}\boxtimes \ldots \boxtimes \ov \P^r_{I_r}
\end{aligned}
\end{equation}

We can now easily prove part \eqref{E:Biv}. Indeed, if the curve $X$ is such that every singular point of $X$ that lies on two different irreducible components is a separating node, then the separating blocks
$\{Y_1,\ldots,Y_r\}$ are integral curves with planar singularities. Therefore, \cite[Thm. B]{arin2} implies that each $\rho_{\un q^i}:\Picbar(Y_i)\to \Picbar(\J_{Y_i}(\un q^i))=\Picbar(\bJbar_{Y_i}^0)$
is an isomorphism (for any $i=1,\ldots, r$).  We conclude that $\rho_{\un q}$ is an isomorphism by \eqref{E:deco-rho}.Ê


Finally, it remains to show that $\rho_{\un q}$ is an open embedding. According to \eqref{E:deco-rho}, we can (and will) assume that the curve $X$ does not have separating nodes.
Under this assumption, we will prove more generally that  $\wt{\rho}_{\un q}$ is an open embedding.
Since fine compactified Jacobians form an open cover of $\bJbar_X^{1-g}$ (see Fact \ref{F:Este-Jac}\eqref{F:Este-Jac3}), Claim 1  gives that
$\wt{\rho}_{\un q}$  is a local isomorphism.
Therefore it remains to show that $\wt{\rho}_{\un q}$ is injective on geometric points, or in other words it is enough to establish the following

\un{CLAIM 2:} Assume that $X$ does not have separating nodes. If $I_1,I_2\in \bJbar_X^{1-g}$ are such that $\ov\P_{I_1}=\ov\P_{I_2}$ then $I_1=I_2$.

In order to prove the Claim, we are going to extend the morphism $\wt{\rho}_{\un q}$ over the effective semiuniversal deformation $\pi:\X\to \Spec R_X$ of $X$ (see \S\ref{S:defo}).
Consider the universal fine compactified Jacobian $\J_{\X}(\un q)\to \Spec R_X$ with respect to the polarization $\un q$ (see \S\ref{S:defo}) and form the algebraic space
$w: \Pmm(\J_{\X}(\un q))\to \Spec R_X$ parametrizing coherent sheaves on $\J_{\X}(\un q)$, flat over $\Spec R_X$, which are relatively simple,  torsion-free of rank-$1$ (see \cite[Thm. 7.4, Prop. 5.13(ii)]{AK}).
Using the universal Poincar\'e sheaf $\Pbun$ over $u\times u: \bJbar_{\X}^{1-g}\times_{\Spec R_X} \bJbar_{\X}^{1-g}\to \Spec R_X$ introduced at the beginning of Section \ref{S:proof}, we can define a
$\Spec R_X$-morphism
\begin{equation}\label{E:rho-univ}
\begin{aligned}
\wt{\rho}^{\rm un}_{\un q}: \bJbar_{\X}^{1-g} & \longrightarrow  \Pmm(\J_{\X}(\un q)) \\
\I & \mapsto \ov\P^{\rm un}_{\I}:=\ov\P^{\rm un}_{|\J_{\X}(\un q)\times \{\I\}}
\end{aligned}
\end{equation}
whose restriction to the closed point $o$ of $\Spec R_X$ is the morphism $\wt{\rho}_{\un q}$ introduced in \eqref{E:rho-tilde}.

Now,  by Fact \ref{F:Este-Jac}\eqref{F:Este-Jac3}, we can choose two general polarizations $\un q^1$ and $\un q^2$ on $X$ such that $I_i\in \J_X(\un q^i)$ for $i=1,2$.
From a relative version of Claim 1, we deduce that the restrictions of $\wt{\rho}^{\rm un}_{\un q}$ to $\J_{\X}(\un q^i)$ (for $i=1,2$) are open embeddings. Consider the open subset
$V:=\wt{\rho}^{\rm un}_{\un q}(\J_{\X}(\un q^1))\cap \wt{\rho}^{\rm un}_{\un q}(\J_{\X}(\un q^2))\subseteq \Pmm(\J_{\X}(\un q))$ which clearly contains the point $\P_{I_1}=\P_{I_2}\in \Pmm(\J_X(\un q))$.
On the fiber product $\X\times_{\Spec R_X} V$ there are two coherent sheaves $\calJ^1$ and $\calJ^2$, flat over $V$ and relatively simple, torsion-free of rank-$1$,  that are obtained by push-forward via
$\id\times \wt{\rho}^{\rm un}_{\un q}$ of the universal sheaves on $\X\times_{\Spec R_X}\J_{\X}(\un q^1)$ and on $\X\times_{\Spec R_X}\J_{\X}(\un q^2)$.
If we denote (for $i=1,2$) by $\calJ^i_v$ the restriction of $\calJ^i$ to the fiber of  $ \X\times_{\Spec R_X} V\to V$ over $v\in V$, then by construction
we have that  $\ov\P^{\rm un}_{\calJ^1_v}=\ov\P^{\rm un}_{\calJ^2_v}$ for any $v\in V$.
Claim 2 will be proved if we show that $\calJ^1_v=\calJ^2_v$ for any $v\in V$, or in other words that
\begin{equation}\label{E:claim-eq}
\calJ^1=\calJ^2\otimes p_2^*(L) \hspace{0.2cm} \text{ for some line bundle } L \: \text{ on } V.
\end{equation}

Denote by $V_0$ the open subset of $V$ consisting of all the points $v\in V$ whose image $w(v)$ in $\Spec R_X$ is such that the fiber $\X_{w(v)}$ of the universal curve $\pi:\X\to \Spec R_X$ over $w(v)$
is smooth or it has a unique singular point which is a node.
By Lemma \ref{L:openU}, the complement of $V_0$ has codimension at least two inside $V$.
Note that, since $X$ does not have separating nodes and separating nodes are preserved under specialization (see \cite[Cor. 3.8]{MRV2}), for every $v\in V$ the curve $\X_{w(v)}$ does not have separating nodes or, equivalently, it is
an integral curve with at most one node.
Therefore, \cite[Thm. 4.1]{EK} (or \cite[Thm. B]{arin2}) implies that \eqref{E:claim-eq} is true over $V_0$, i.e. $\calJ^1_{|V_0}=\calJ^2_{|V_0}\otimes p_2^*(L_0)$ Êfor some line bundle
$L_0\in V_0$. Since $V$ is smooth
(because $\J_{\X}(\un q^i)$ is smooth by Fact \ref{F:Punivfine}\eqref{F:Punivfine1}) and the complement of $V_0$ has codimension at least two inside $V$, we can extend the line bundle $L_0$ on $V_0$
to a line bundle $L$ on $V$.  With this choice of $L$, the two sheaves appearing on the left and right hand side of \eqref{E:claim-eq} are Cohen-Macaulay sheaves on $V$ that agree on the open subset $V_0$ whose complement has codimension at least two; therefore the two sheaves agree on  $V$ by \cite[Thm. 5.10.5]{EGAIV2}, q.e.d.

\end{proof}


\noindent {\bf Acknowledgments.}
We thank E. Esteves, T. Hausel, J.L. Kass, L. Migliorini and T. Pantev for useful conversations related to this work. We are extremely grateful to M. Groechenig for suggesting to extend our original result to
the case of two different fine compactified Jacobians and also to consider the case of a base field of positive (large) characteristic. 

This project started while M. Melo was visiting the Mathematics Department of the University of Roma ``Tor Vergata'' funded by the ``Michele Cuozzo'' 2009 award. She wishes to express her gratitude to Michele Cuozzo's family and to the Department for this great opportunity.

M. Melo was supported by a Rita Levi Montalcini Grant, funded by the Italian government through MIUR. 
A. Rapagnetta and F. Viviani are supported by the MIUR national project FIRB 2012 ``Spazi di moduli e applicazioni".
M. Melo and F. Viviani are members of the Centre for Mathematics of the
University of Coimbra (CMUC) -- UID/MAT/00324/2013, funded by the Portuguese
Government through FCT/MEC and co-funded by the European Regional Development Fund through the Partnership Agreement PT2020.

\end{document}